\documentclass[11pt,letterpaper]{amsart}
\usepackage[pagewise]{lineno}

\usepackage{amssymb,amsmath,amsthm,amscd,epsf,latexsym,verbatim,graphicx,amsfonts,hyperref,epstopdf,xcolor,wasysym}
\usepackage{cancel}
\input epsf.tex
\usepackage{enumerate}
\usepackage[utf8]{inputenc}
\usepackage[T1]{fontenc}
\usepackage{graphicx}
\usepackage{palatino, url, multicol}
\usepackage{subcaption}
\usepackage[margin=1.5in]{geometry}           
\usepackage{tikz}
\usetikzlibrary{patterns}
\usepackage{mathtools}
\usepackage[ruled, vlined]{algorithm2e}
\usepackage{tikz-cd}
\usetikzlibrary{shapes.geometric, arrows}
\usepackage{amsmath}
\usepackage{mathdots}
\usepackage[all]{xy}
\usepackage{stackengine} 
\usepackage{caption} 
\DeclareCaptionLabelFormat{cont}{#1~#2\alph{ContinuedFloat}}
\usepackage{marginnote}

\theoremstyle{plain}
\newtheorem{theorem}{Theorem}[section]
\newtheorem{conjecture}[theorem]{Conjecture}
\newtheorem{lemma}[theorem]{Lemma}
\newtheorem{proposition}[theorem]{Proposition}
\newtheorem{corollary}[theorem]{Corollary}
\newtheorem{definition}[theorem]{Definition}
\newtheorem{remark}[theorem]{Remark}
\newtheorem{example}[theorem]{Example}

\newtheorem*{theoremcontinuousdiskextension}{Theorem~\ref{t:continuousdiskextension}}
\newtheorem*{renormalizationtheorem}{Theorem~\ref{T:q-root}}
\newtheorem*{cylindercorollary}{Corollary~\ref{C:cylinder}}

\title{Master Teapots and Entropy Algorithms for the Mandelbrot Set }

\author{Kathryn Lindsey, Giulio Tiozzo, Chenxi Wu}

\date{September 30, 2024}

\newif\ifdraft\drafttrue 
\draftfalse

\def\0{{\mathbf 0}}

\begin{document}
\maketitle

\begin{figure}
\begin{center}
\includegraphics[width = 0.9 \textwidth]{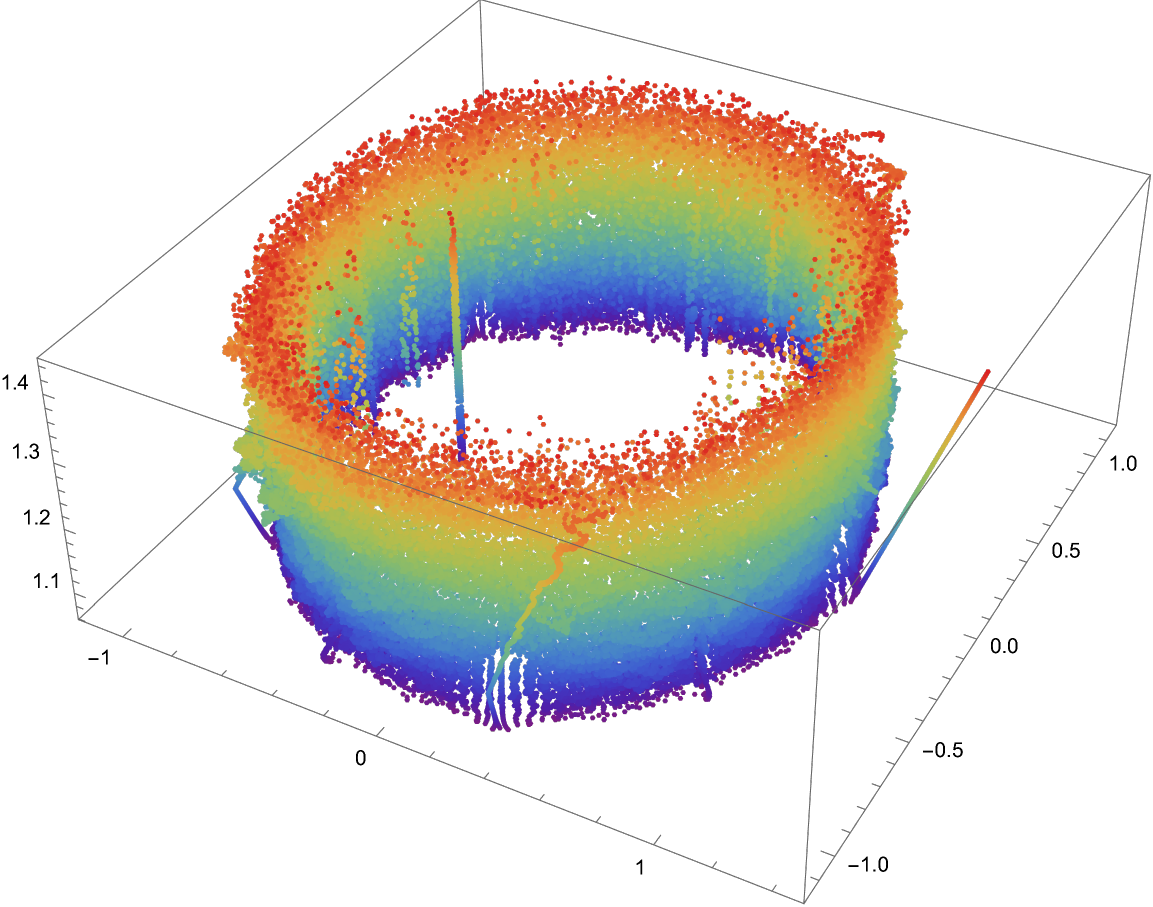}
\end{center}
\caption{The Master Teapot $\Upsilon_{1/5}$ for the $1/5$-vein.}
\end{figure}

\begin{abstract} 
We construct an analogue of W. Thurston's ``Master teapot" for each principal vein in the Mandelbrot set, 
and generalize geometric properties known for the corresponding object for real maps. 
In particular, we show that eigenvalues outside the unit circle move continuously, while 
we show ``persistence" for roots inside the unit circle. As an application, this shows that the outside 
part of the corresponding ``Thurston set" is path connected. 
In order to do this, we define a version of kneading theory for principal veins, and we prove the equivalence of 
several algorithms that compute the core entropy. 
\end{abstract}

\section{Introduction}

The \emph{core entropy} $h(f)$ of a postcritically finite (PCF) polynomial $f:\mathbb{C} \to \mathbb{C}$ is the topological entropy of the restriction of $f$ to its Hubbard tree.  It is well known that for any PCF polynomial $f$, the number $\lambda(f) = e^{h(f)}$, which is called the \emph{growth rate} of $f$, is a  \emph{weak Perron number} -- a real, positive algebraic integer that is greater than or equal to the moduli of its Galois conjugates. 

The study of algebraic properties of growth rates of polynomial maps has gained renewed interest in recent years, in particular due to W. Thurston, 
who plotted the closure of the set of Galois conjugates of growth rates of real PCF quadratic polynomials, defining what is 
known as the \emph{Thurston set} or \emph{entropy spectrum}, also nicknamed ``the bagel" for its shape (see Figure \ref{F:bagel3}, \cite{thurston}, \cite{TiozzoGaloisConjugates}).

All Galois conjugates of the growth rate are also eigenvalues of the transition matrix associated to the Markov partition for the 
polynomial. From a dynamical point of view, eigenvalues of the Markov transition matrix determine statistical properties of the dynamical system, with respect to the measure of maximal entropy: for instance, simplicity of the leading eigenvalue is equivalent to ergodicity, and absence of other eigenvalues of maximal modulus is equivalent to mixing;  the spectral gap yields the rate of mixing (see e.g. \cite[Ch.1]{Baladi-book}). 

On the other hand, another development has been the study of core entropy for complex quadratic polynomials, 
extending the well-known theory of topological entropy for real unimodal maps, going back to Milnor-Thurston \cite{MilnorThurston}. 
W. Thurston initiated the study of core entropy in his Cornell seminar  \cite{ThurstonPeople}, leaving several open questions. 

For quadratic polynomials, each rational angle $\theta \in \mathbb{Q}/\mathbb{Z}$ determines a postcritically finite parameter $c_{\theta}$ in the Mandelbrot set, and we denote by $h(\theta)$ the core entropy of the polynomial $f_{c_\theta}(z) := z^2 + c_{\theta}$.  
It was proven in \cite{TiozzoContinuity}, \cite{DudkoSchleicher} that the core entropy function $h:\mathbb{Q} / \mathbb{Z} \to \mathbb{R}$ extends to a continuous function from $\mathbb{R}/\mathbb{Z}$ to $\mathbb{R}$. See also \cite{GaoYanTiozzo} for the higher degree case. 

The goal of this paper is to study the eigenvalues associated to PCF quadratic polynomials; in particular, we associate to 
any principal vein in the Mandelbrot set a fractal 3-dimensional object, generalizing what Thurston called the \emph{Master Teapot}, 
and study its geometry. 
As observed in  \cite[Figure 7.7]{thurston}, there appear to be two different patterns. On the one hand, the roots outside the unit circle seem to move continuously with the parameter\footnote{See also the video \url{https://vimeo.com/259921275}. For visualizations of the Master Teapot, see also
\url{http://www.math.toronto.edu/tiozzo/teapot.html}. For the Thurston set and various related sets, see e.g. \url{http://www.math.toronto.edu/tiozzo/gallerynew.html}.}; on the other hand, roots inside the circle do not move continuously, but rather they display \emph{persistence}: namely, 
the set of roots increases as one progresses towards the tip of the vein. In this paper, we will rigorously prove these two phenomena 
for the teapots associated to principal veins in the Mandelbrot set. 

\subsection{Continuity of eigenvalues} 

For a postcritically finite polynomial $f$, let $T_f$ be its Hubbard tree (see Section \ref{S:background-M} for these and other basic definitions). The postcritical set of $f$ together with the branch points of $T_f$ determine a Markov partition for the action of $f : T_f \to T_f$.  Denote by $M_f$ the transition matrix associated to this Markov partition.  We consider the set $Z(f)$ of eigenvalues of $M_f$:
$$Z(f) \coloneqq \{\lambda \in \mathbb{C} \mid \textrm{det}(M_f - \lambda I) = 0\}.$$
The growth rate of $f$ is one element of the set $Z(f)$.  For a rational angle $\theta \in \mathbb{Q}/\mathbb{Z}$, we define $Z(\theta)$ to be $Z(f_{c_\theta})$.

Denote  as $Com^+(\mathbb{C})$ the collection of compact subsets of $\mathbb{C}\setminus\overline{\mathbb{D}}$, with the Hausdorff topology.  Define $Z^+  :  \mathbb{Q} / \mathbb{Z} \to Com^+(\mathbb{C})$ as the unit circle together with the set of eigenvalues of modulus greater than $1$, i.e.
$$Z^+(\theta) := S^1 \cup \left(Z(\theta) \cap (\mathbb{C} \setminus \mathbb{D}\right) ).$$

The first main result is the following:

\begin{theorem} \label{t:continuousdiskextension}
The map $Z^+ : \mathbb{Q} / \mathbb{Z} \to Com^+(\mathbb{C})$ admits a continuous extension from $\mathbb{R}/\mathbb{Z} \to Com^+(\mathbb{C})$.
\end{theorem}

Since the growth rate is the leading eigenvalue, this is a generalization of the main theorem of \cite{TiozzoContinuity} to all eigenvalues. 
In the proof, we adapt to the new situation the combinatorial tools such as the \emph{wedge} and the \emph{spectral determinant} from  \cite{TiozzoContinuity}.

\subsection{Entropy algorithms} 

The second focus of this paper is relating various algorithms for computing the core entropy of a quadratic polynomial. We prove that they all produce the same polynomials, up to cyclotomic factors. 

The easiest way to compute the entropy of a PCF map is by using the characteristic polynomial $P_{Mar}(t)$ of the transition matrix. This approach is the simplest, but has several drawbacks, as e.g. the shape of the Hubbard tree is not stable under perturbations of the parameter.

For this reason, Thurston came up with a different algorithm to compute core entropy (see \cite{ThurstonPeople}, \cite{Gao}), 
which is more stable, and is used e.g. in \cite{TiozzoContinuity} to prove the continuity. This gives rise to what we call the \emph{Thurston polynomial} $P_{Th}(t)$. 

A third way to compute entropy is through the celebrated \emph{kneading theory} of Milnor-Thurston \cite{MilnorThurston}, which applies 
to real multimodal maps. In this paper, we establish a new version of kneading theory which can be applied to complex polynomials 
lying on a principal vein. This gives rise to a  new \emph{principal vein kneading polynomial} $D(t)$. 
We developed this version so that it would have the property that the map from itineraries (of the critical point) to the kneading determinants is continuous. This continuity is needed for our proof of Theorem  \ref{t:persistence}. 

We prove that the roots of the polynomials given by these three algorithms coincide, off the unit circle. 

\begin{theorem} \label{T:equalpolys}
For any postcritically finite parameter the following 2 polynomials have the same roots off the unit circle:
\begin{enumerate}
\item the polynomial $P_{Th}(t)$ that we get from Thurston's algorithm;
\item the polynomial $P_{Mar}(t)$ that we get from the Markov partition.
\end{enumerate}
If, furthermore, the parameter is critically periodic and belongs to a principal vein (so that the  principal vein kneading polynomial is defined),  a third polynomial that has the same roots off the unit circle is
\begin{enumerate}
\item[(3)] the principal vein kneading polynomial $D(t)$.
\end{enumerate}
\end{theorem}

\begin{figure}
\begin{center}
\includegraphics[width = 0.8 \textwidth]{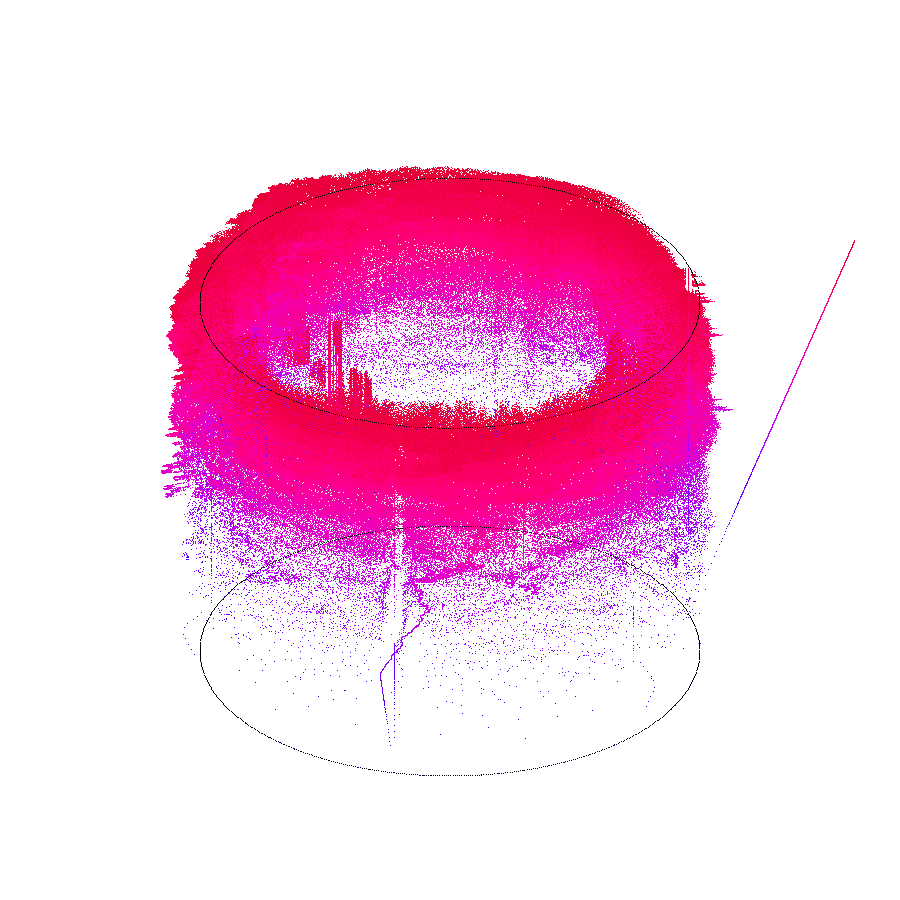}
\end{center}
\caption{The Master Teapot $\Upsilon_{1/3}$ for the $1/3$-vein. We plotted the roots associated to all critically periodic parameters with simplified itinerary of period up to $20$, obtaining $\sim 2.8 \times 10^6$ points.}
\end{figure}

\subsection{Teapots for principal veins} 

Thirdly, we investigate the multivalued function $Z$ restricted to \emph{principal veins} of the Mandelbrot set; the behavior of $Z$ on a principal vein is encapsulated in the geometry of various \emph{Master Teapots} associated to that vein.  

For natural numbers $p < q$ with $p$ and $q$ coprime, Branner-Douady \cite{BrannerDouady} showed the existence of the \emph{$\frac{p}{q}$-principal vein}, that is, a continuous arc that connects the ``tip'' of the $\frac{p}{q}$-limb of the Mandelbrot set to the main cardioid.

We denote as $\mathcal{V}_{p/q}$ the set of parameters in the Mandelbrot set which lie on the $\tfrac{p}{q}$-principal vein. 
We are particularly interested in the set $\mathcal{V}_{p/q}^{per}$ of all parameters $c \in \mathcal{V}_{p/q}$ such that the map $f_c:z \mapsto z^2+c$ is critically periodic. 
Finally, we define $\Theta_{p/q}^{per}$ to be the set of all angles $\theta \in \mathbb{Q}/\mathbb{Z}$ such that the external ray of angle $\theta$ lands at the root of a hyperbolic component on the $\tfrac{p}{q}$-principal vein.

For each $\lambda$ that arises as a growth rate associated to the $\tfrac{p}{q}$-principal vein, define 
$$\mathcal{Z}(\lambda) \coloneqq \{ z \in \mathbb{C} \mid \textrm{det}(M_{\theta} - zI) = 0\textrm{ for every } \theta \in \Theta^{per}_{p/q} \textrm{ such that } \lambda = e^{h(\theta)}\}.$$
Note that $\mathcal{Z}(\lambda)$ equals the set of eigenvalues of $\textrm{det}(M_{c(\lambda)} - zI)$, where $c(\lambda)$ is the critically periodic parameter of growth rate $\lambda$ closest to the main cardioid in the vein (c.f. Lemma \ref{l:closestrepresentative}). 
 
\begin{definition}
We define the \emph{$\frac{p}{q}$-Master Teapot} to be the set 
$$
\Upsilon_{p/q} \coloneqq \overline{ \left\{(z,\lambda) \in \mathbb{C} \times \mathbb{R} \mid \lambda = e^{h(\theta)}  \textrm{ for some } \theta \in \Theta^{per}_{p/q}, \    z\in \mathcal{ Z}(\lambda) \right\} },
$$
where the overline in the notation above denotes the topological closure.
\end{definition}
 
Note that $\Upsilon_{1/2}$ is Thurston's original Master Teapot from \cite{thurston}. Then, the Persistence Theorem of \cite{BrayDavisLindseyWu} states that if a point $z \in \mathbb{D}$ is in the height-$\lambda$ slice of the Master Teapot $\Upsilon_{1/2}$, then $z$ is also in all the higher slices, i.e. for $z \in \mathbb{D}$, $(z,\lambda) \in \Upsilon_{1/2}$ implies $\{z\} \times [\lambda,2] \in \Upsilon_{1/2}$. The present work generalizes this to all principal veins.  

In order to formulate the next theorem precisely, let us recall that a calculation shows that the core entropy of the tip of the $\frac{p}{q}$-principal vein equals $\log \lambda_q$, where $\lambda_q$ is the largest root of the polynomial $P(x) := x^q - x^{q-1} - 2$. Thus, by monotonicity, the growth rates of all parameters within such vein are contained in the interval $[1, \lambda_q]$. 
Note that $\lambda_q \to 1$ as $q \to \infty$.

We prove the persistence property for all principal veins: 

\begin{theorem}[Persistence Theorem]  \label{t:persistence}
Let $p < q$ coprime, and let $\Upsilon_{p/q}$ be the $\tfrac{p}{q}$-Master Teapot. 
If a point $(z,\lambda)$ belongs to $\Upsilon_{p/q}$ with $z \in \mathbb{D}$, then 
the ``vertical segment" $\{z\} \times [\lambda, \lambda_{q}]$ also lies in $\Upsilon_{p/q}$.
\end{theorem}
 
Finally, because of renormalization, points in the teapot behave nicely under taking $q^{th}$ roots. 
  
\begin{theorem} \label{T:q-root}
If $(z,\lambda) \in \Upsilon_{1/2}$ with $|z| \neq 1$, then for any $q$, if $w^q = z$ then the point 
$(w,+\sqrt[q]{\lambda})$ belongs to $\Upsilon_{p/q}.$
\end{theorem}
 
\begin{corollary} \label{C:cylinder}
The unit cylinder $[1, \lambda_q] \times S^1$ is contained in $\Upsilon_{p/q}$.
\end{corollary}

\subsection{The Thurston set} 
W. Thurston \cite{thurston} also investigated the one-complex-dimensional set obtained by projecting the Master Teapot (or a variant thereof) 
to its $z$-coordinate. 
This set displays a lot of structure, and is known as the \emph{Thurston set} or \emph{entropy spectrum}, also nicknamed ``the bagel" for its shape (see Figure \ref{F:bagel3}). 

The Thurston set has attracted considerable attention recently, and is also related to several other sets defined by taking roots of polynomials with restricted digits, as well as limit sets of iterated function systems (see, among others, \cite{BouschConnexite}, \cite{BouschPaires}, \cite{thompson}, \cite{CalegariKochWalker}, \cite{LindseyWu}, \cite{PerezSilvestri}). 

In \cite[Appendix]{TiozzoGaloisConjugates}, variations of the Thurston set are proposed and drawn for each principal vein. 
In particular, one considers the \emph{Thurston set} $\Sigma_{p/q}$ \emph{for the principal $p/q$-vein}, defined as

$$\Sigma_{p/q} := \overline{ \left\{z \in \mathbb{C}  \mid \textrm{det}(M_{\theta}-zI)=0  \textrm{ for some } \theta \in \Theta^{per}_{p/q}    \right\} }.$$

Using Theorem \ref{t:continuousdiskextension}, we obtain 

\begin{theorem}  \label{T:bagel-connected}
For any $(p, q)$ coprime, the Thurston set
$$\Sigma_{p/q} \cap \{ z   \in \mathbb{C}  \ : \ |z| \geq 1\}$$
is path connected.
\end{theorem}

The analogous property for the real case is proven in \cite{TiozzoGaloisConjugates}. 

\begin{remark}
With the above definitions, $\Sigma_{p/q}$ is not the projection of $\Upsilon_{p/q}$ onto the horizontal coordinate. 
The issue is that multiple different critically periodic parameters in a principal vein can have the same core entropy while having different characteristic polynomials $\chi(t) =  \textrm{det}(M_c -t I)$.
Rather, $\Sigma_{p/q}$ is the projection of a ``combinatorial" version of the Master Teapot (see Section \ref{ss:combinatorialveins}).
This version of the Thurston set differs slightly from the one considered in \cite{BrayDavisLindseyWu}. 
\end{remark}

\begin{figure}
\begin{center}
\includegraphics[width = 0.8 \textwidth]{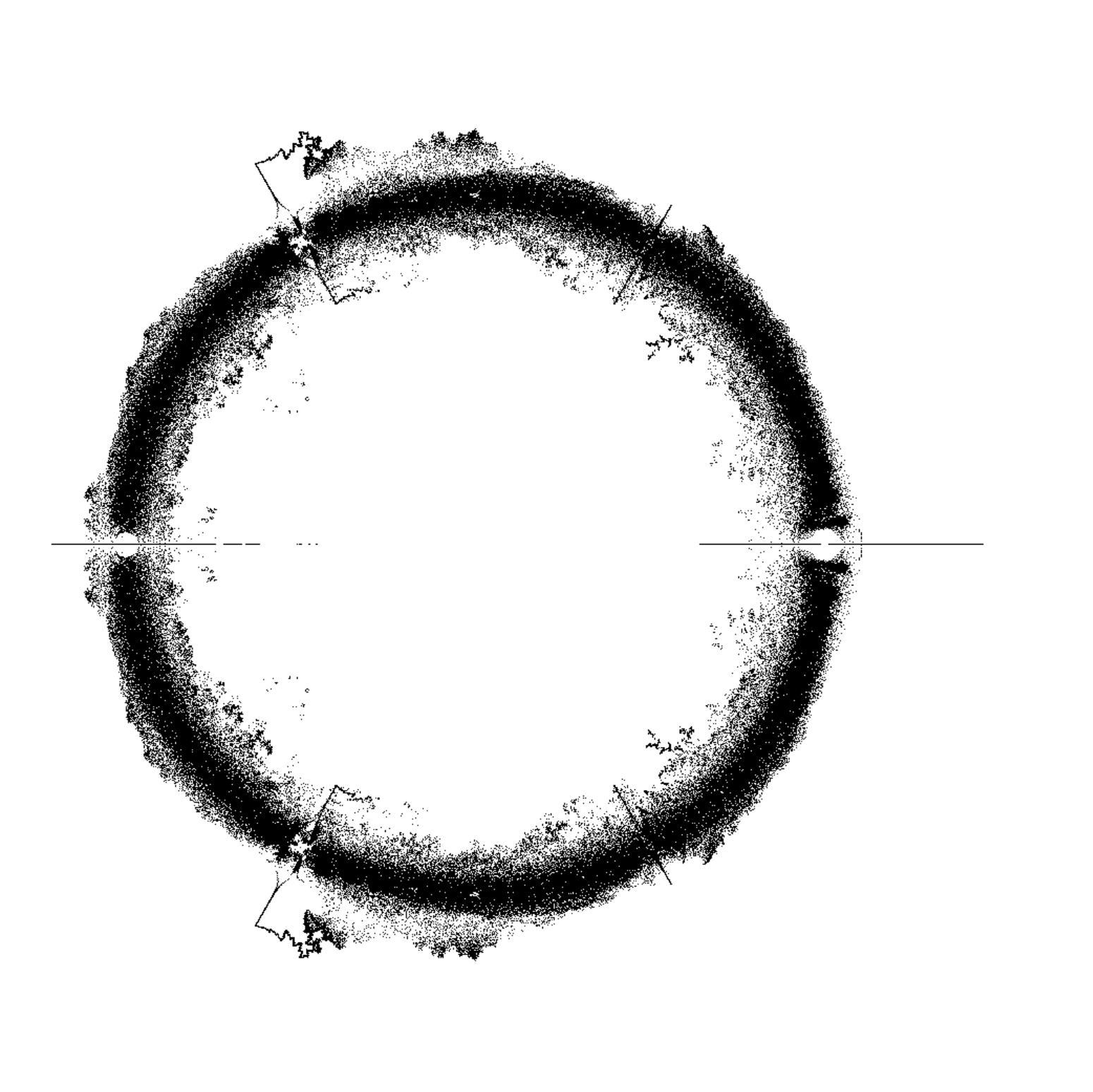}
\end{center}
\caption{The Thurston set $\Sigma_{1/3}$ associated to the $1/3$-vein.} 
\label{F:bagel3}
\end{figure}

\medskip
Note that there is a difference between the Galois conjugates of $\lambda$ and the eigenvalues of a matrix $M_\theta$ with $\lambda = e^{h(\theta)}$. In fact, the characteristic polynomial of $M_\theta$ need not be irreducible. Informed by the real case ($\tfrac{p}{q} = \tfrac{1}{2}$), we conjecture:

 \begin{conjecture}
$$\Upsilon_{p/q} = \overline{\{ (z, \lambda) \in \mathbb{C} \times \mathbb{R} \mid  \lambda = e^{h(\theta)}  \textrm{ for some } \theta \in \Theta^{per}_{p/q}, z \textrm{ is a Galois conjugate of }\lambda 
 \}}.$$
\end{conjecture}

\subsection*{Structure of the paper}
In Section \ref{S:background-M} we give some background on the Mandelbrot set, while in Section \ref{S:background-graph} 
we discuss some background on the graphs and combinatorial structures we use. 
In Section \ref{s:relatingThurstonAndMarkovPolys}, we relate the Thurston polynomial to the Markov polynomial, proving the first part of Theorem \ref{T:equalpolys}. 
In Sections \ref{sec:coversofthefinitemodel} and \ref{S:continuous-ext} we discuss the dependence of the eigenvalues on the external angle, proving Theorem \ref{t:continuousdiskextension}. 
Then, in Section \ref{S:kneading-veins} we develop our new kneading theory for principal veins, proving the second part of Theorem \ref{T:equalpolys}. 
In Section \ref{S:surgery} we discuss how to interpret the Branner-Douady surgery in terms of itineraries, and define the procedure of \emph{recoding} we need to compare itineraries on different veins. In Section \ref{S:renorm} we discuss how to describe renormalization (and tuning) 
in terms of our kneading polynomials. Using renormalization, we prove Theorem \ref{T:q-root}.
In Section \ref{s:persistence}, we prove the Persistence Theorem, Theorem \ref{t:persistence}. 
Finally, in Section \ref{ss:combinatorialveins} we apply these results to combinatorial veins and the Thurston set, proving Theorem \ref{T:bagel-connected}. 
In the Appendix, we show the useful fact (probably well-known, but we could not find a reference) that the Markov polynomial 
and the Milnor-Thurston kneading polynomial coincide for real critically periodic parameters. 

\subsection*{Acknowledgements} 
G. T. is partially supported by NSERC grant RGPIN-2017-06521 and an Ontario Early Researcher Award ``Entropy in dynamics, geometry, and probability". K.L. is partially supported by NSF grant \#1901247.

 %%%%%%%%%%%%%%%%%%%%%%%%%%%%%%%%%%%%%%
\section{Background on the Mandelbrot set and veins} \label{S:background-M}

\subsection{The Mandelbrot set}
 
\subsubsection{First definitions} 
Every quadratic polynomial on $\mathbb{C}$ is conformally equivalent to a unique polynomial of the form $f_c(z)=z^2+c$. The \emph{filled Julia set} for $f_c$, denoted $\mathcal{K}(f_c)$, consists of all points $z \in \mathbb{C}$ whose orbit under $f_c$ is bounded; the \emph{Julia set} $\mathcal{J}(f_c)$ is the boundary of $\mathcal{K}(f_c)$.  
The Mandelbrot set $\mathcal{M}$ is the set of parameters $c$ for which the filled Julia set for the map $f_c$ is connected.  A parameter $c \in \mathcal{M}$ is said to be \emph{postcritically finite} if $\{f_c^n(0) \mid n \in \mathbb{N}\}$ is a finite set.  A parameter $c \in \mathcal{M}$ is said to be \emph{critically periodic} if there exists $n \in \mathbb{N}$ such that $f_c^n(0) = 0$, and \emph{critically preperiodic} if it is not critically periodic but there exist $m > n > 0$ such that $f_c^n(0) = f_c^m(0)$.
  
\subsubsection{Hubbard trees} 
Let $f_c$ be a quadratic polynomial for which the Julia set is connected and locally connected (hence, also path connected). 
Then any two points $x, y$ in the filled Julia set are connected by a \emph{regulated arc}, i.e. a continuous arc which 
lies completely in $\mathcal{K}(f_c)$ and is canonically chosen (see e.g. \cite{DHOrsay}). We denote such regulated arc by $[x, y]$. 
Then we define the \emph{Hubbard tree} $T_{f_c}$ as the union 
$$T_{f_c} := \bigcup_{i, j \geq 0} [f_c^i(0), f_c^j(0)].$$
In particular, if $f_c$ is postcritically finite, the above hypotheses are satisfied, and the Hubbard tree $T_{f_c}$ is topologically a finite tree.
Moreover, one has $f_c(T_{f_c}) \subseteq T_{f_c}$.   
  
\subsubsection{B\"{o}ttcher coordinates}  \label{sss:BottcherCoords}
For $c \in \mathcal{M}$, \emph{B\"{o}ttcher coordinates} on $\hat{\mathbb{C}} \setminus \mathcal{K}(f_c)$ are the ``polar coordinates'' induced by the unique Riemann mapping $\Phi_c: \hat{ \mathbb{C}} \setminus \overline{\mathbb{D}} \to \hat{\mathbb{C}} \setminus \mathcal{K}(f_c)$ that satisfies $\Phi_c(\infty) = \infty$ and $\Phi_c'(\infty) = 1$.  The map $\Phi_c$ conjugates the dynamics outside of $\mathcal{K}(f_c)$ to the squaring map: $\Phi_c \circ f_c = (\Phi_c)^2$ on $\hat{\mathbb{C}} \setminus \mathcal{K}(f_c)$.  Similarly, B\"{o}ttcher coordinates on $\hat{\mathbb{C}} \setminus \mathcal{M}$ come from the unique Riemann mapping $\Phi_{\mathcal{M}} : \hat{ \mathbb{C}}  \setminus \overline{\mathbb{D}} \to \hat{\mathbb{C}} \setminus \mathcal{M}$ that satisfies $\Phi_{\mathcal{M}}(\infty) = \infty$ and $\Phi_{\mathcal{M}}'(\infty) = 1$.   By Carath\'{e}odory's Theorem, the maps $\Phi_c$ or $\Phi_{\mathcal{M}}$ extend continuously to the unit circle if and only if $\mathcal{K}(f_c)$ or $\mathcal{M}$, respectively, are locally connected.  The \emph{dynamical ray} of angle $\theta$ for $f_c$, which we denote by $R_c(\theta)$, is the image of the map $r \mapsto \Phi_c(r e^{2 \pi i \theta})$, and it is said to \emph{land} at $z \in \mathbb{C}$ if $\lim_{r \searrow 1} \Phi_c(re^{2 \pi i\theta}) = z$. 
Similarly, the \emph{parameter ray} of angle $\theta$, denoted by $R_{\mathcal{M}}(\theta)$, is the image of the  map $r \mapsto \Phi_{\mathcal{M}}(r e^{2 \pi i \theta})$, and
 it is said to \emph{land} at $z \in \mathbb{C}$ if $\lim_{r \searrow 1^+} \Phi_{\mathcal{M}}(re^{2 \pi i\theta}) = z$.
%the same holds for a parameter ray by considering $\Phi_{\mathcal{M}}$.  
 It is conjectured that $\mathcal{M}$ is locally connected, and it is well-known \cite[Theorem 13.1]{DHOrsay} that every rational parameter ray $R_{\mathcal{M}}(\theta)$ lands.  
 %We use $R_c(\theta)$ and $R_{\mathcal{M}}(\theta)$ to denote the dynamical or parameter rays of angle $\theta$.  
For any $c \in \mathcal{M}$, the landing point of $R_c(0)$ is a fixed point of $f_c$ and is called the \emph{$\beta$-fixed point}; the \emph{$\alpha$-fixed point} is the other fixed point of $f_c$ (see also Figure~\ref{f:Htree1/5}).  A point $c \in \partial \mathcal{M}$ is said to be \emph{biaccessible} (resp. \emph{$k$-accessible}) if it is the landing point of precisely $2$ (resp. $k$) parameter rays.  
  
\subsubsection{Hyperbolic and critically periodic parameters}
A parameter $c \in \mathcal{M}$ is said to be \emph{hyperbolic} if the critical point for $f_c$ tends  to the (necessarily unique) attracting cycle in $\mathbb{C}$.  
The hyperbolic parameters of $\mathcal{M}$ form an open set; connected components of this set are called \emph{hyperbolic components}.  Each hyperbolic component $H$ is conformally equivalent to $\mathbb{D}$ under the map $\lambda$ which assigns to each $c \in H$ the multiplier of its (unique) attracting cycle.  The \emph{center} of $H$ is the parameter $\lambda^{-1}(0)$, and (continuously extending $\lambda^{-1}$ to the unit circle) the \emph{root} of $H$ is $\lambda^{-1}(1)$.  Critically periodic parameters are precisely those parameters that are centers of hyperbolic components. The set of all real hyperbolic parameters is dense in $\mathcal{M} \cap \mathbb{R} = [-2,1/4]$; in particular, every component of the interior of $\mathcal{M}$ which meets the real line is hyperbolic \cite{Lyubich}.  Every critically periodic parameter $c$ is the center of a hyperbolic component of the Mandelbrot set. 
 
\subsubsection{Parabolic parameters}  \label{sss:parabolic}
A parameter $c \in \mathcal{M}$ is called \emph{parabolic} if $f_c$ has a periodic orbit with some root of unity as the multiplier (and such a point is called a parabolic periodic point).  The root of each hyperbolic component of $\mathcal{M}$ is a parabolic parameter.  Every parabolic periodic cycle of a polynomial attracts the forward orbit of a critical point, so quadratics have at most one parabolic periodic orbit.  For a parabolic parameter $c$, the unique Fatou component of $\mathcal{K}(f_c)$ containing the critical value of $f_c$ is called the \emph{characteristic Fatou component}, and it has a unique parabolic periodic point on its boundary; this parabolic point is called \emph{characteristic periodic point} of the parabolic orbit. The characteristic periodic point is the landing point of at least two dynamical rays, and the two rays closest to the critical value on either side are called \emph{characteristic rays} (\cite{Schleicher}).

\subsubsection{Rational angles and postcritically finite maps} \label{ss:anglestotrees}
A rational angle $\theta = a/b$, written in lowest terms, is periodic (resp. preperiodic) under the doubling map if and only if $b$ is odd (resp. even). 
If $\theta$ is periodic, the landing point of $R_{\mathcal{M}}(\theta)$ is a parabolic parameter, which is the root of some hyperbolic component.  We associate to $\theta$ the map, which we call $f_{c_\theta}$, that is the center of this hyperbolic component. 
 Note that topological entropy is constant on the closure of a hyperbolic component.  
If $\theta$ is preperiodic, the landing point of $R_{\mathcal{M}}(\theta)$ is a critically preperiodic parameter $c_\theta$, and we call the map associated to this parameter $f_{c_\theta}$. 

%%%%%%%%%%%%%%%%%%%%%%%%%%%%%%%%%%%%%%
  
  \subsection{Veins in parameter space} \label{ss:veins}
A \emph{vein} in the Mandelbrot $\mathcal{M}$ set is a continuous, injective arc in $\mathcal{M}$.  It is known (\cite[Corollary A]{BrannerDouady}) that there is a vein connecting the landing point in $\mathcal{M}$ of any external ray of angle $p/2^q$, for $p,q \in \mathbb{N}$, to the main cardioid.  

For any integers $p,q$ such that $0 < p < q$ and $p$ and $q$ are coprime, the \emph{$\frac{p}{q}$-limb} in the Mandelbrot set consists of the set of parameters $c \in \mathcal{M}$ such that $f_c$ has rotation number $\frac{p}{q}$ around the $\alpha$-fixed point of $f_c$. 
In each such limb, there exists a unique parameter $t_{p/q}$ such that the critical point, $0$, maps under $f_{t_{p/q}}$ to the $\beta$-fixed point (i.e. the landing point of the dynamical ray of angle $0$) of $f_{t_{p/q}}$ in precisely $q$ steps. 

The \emph{$\frac{p}{q}$-principal vein}, $0 < p < q$ coprime, which we denote by $\mathcal{V}_{p/q}$, is the vein joining $t_{p/q}$ to the main cardioid.  The Hubbard tree $T_f$ associated to any map $f$ in the $\frac{p}{q}$-principal vein is a $q$-pronged star (see e.g. \cite[Proposition 15.3]{TiozzoThesis}), whose center point $\alpha_f$ is the $\alpha$-fixed point.  Moreover, deleting $\alpha_f$ and $0$ from the Hubbard tree $T_f$ yields of a decomposition of $T_f$ into $q+1$ arcs: 
$$T_f \setminus \{\alpha_f, 0\} = I_0 \sqcup I_1 \sqcup \ldots \sqcup I_q$$
where the critical point, $0$, separates $I_0$ and $I_1$, the $\alpha$-fixed point separates $I_1, \ldots, I_q$, and the dynamics are:
\begin{itemize}
\item $f(I_0) \subseteq I_0 \cup I_1 \cup I_2$. 
\item $f: I_k \to I_{k+1}$ homeomorphically for $1 \leq k \leq q-1$,
\item $f:I_q \to I_0 \cup I_1$ homeomorphically.
\end{itemize}
  (See Figure \ref{f:Htree1/5}, which shows the combinatorial model of the Hubbard tree for angle $\theta=1/5$.)

\begin{figure} 
\begin{minipage}{0.49 \textwidth}
\includegraphics[width = 0.99 \textwidth]{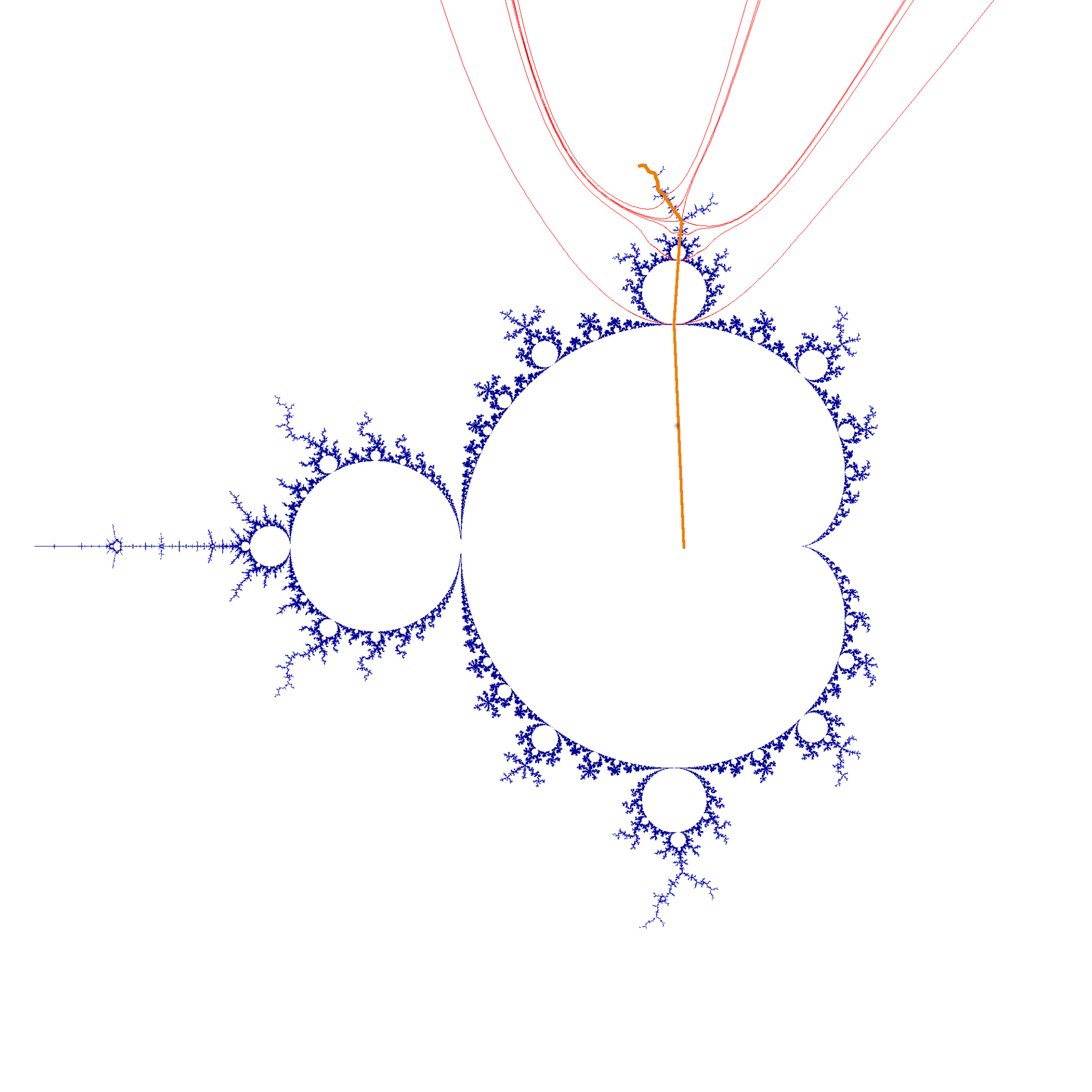}
\end{minipage}
\begin{minipage}{0.49 \textwidth}
\begin{tikzpicture}[scale=1.75]
\filldraw [black] (-.35,.6) circle (1pt);
\node at (-.5,.43)  {$\alpha$};
\filldraw [black] (0,0) circle (1pt);
\node  at (.2,.1)  {$x_0$};
\filldraw [black] (-.15,1) circle (1pt);
\node at (-.15,1.2)  {$x_1$};
\filldraw [black] (-1.2,.7) circle (1pt);
\node  at (-1.4,.7)  {$x_2$};
\filldraw [black] (.76,-.6) circle (1pt);
\node at (1.06,.-.6)  {$x_3$};
\node at (.5,-.2) {$I_0$};
\node at (-.25,.2) {$I_1$};
\node at (-.1,.75) {$I_2$};
\node at (-.8,.8) {$I_3$};
\draw (-.35,.6)-- (-.15,1);
\draw (-.35,.6)--(-1.2,.7);
\draw(-.35,.6)--(0,0);
\draw (0,0) -- (.76,-.6);
\end{tikzpicture}
\end{minipage}
\caption{Left: The $1/3$-principal vein $\mathcal{V}_{1/3}$, joining the center of the main cardioid to the parameter 
$c = t_{1/3} \approx -0.228 + 1.115 i$. For this parameter, the point $f_c^2(c) = \beta$ is fixed, so the characteristic angle 
$\theta$ must satisfy $2^2 \theta \equiv 0 \mod 1$, hence $\theta = \frac{1}{4}$ (the angle $\theta = \frac{3}{4}$ gives rise to the complex conjugate vein, with $\frac{p}{q} = \frac{2}{3}$). 
In red, some external rays landing on the vein. Right: The combinatorial model of the Hubbard tree for parameters on the $1/3$-principal vein. } 
\label{f:Htree1/5}
\end{figure}

In particular, for $f = f_{t_{p/q}}$, the map associated to the tip of the $\frac{p}{q}$-principal vein, the dynamics are given by 
 $$\begin{array}{ll}
 f(I_0) = I_0 \cup I_1 \cup I_2  & \\
 f(I_k) = I_{k+1} & \qquad \textup{for }k = 1, \dots, q-1 \\
 f(I_q) = I_0 \cup I_1 & 
 \end{array}$$
Hence, the entropy of $f_{t_{p/q}}$ equals $\log \lambda_{q}$, where 
 $\lambda_q$ is the largest root of polynomial $P(x) = x^q - x^{q-1} - 2$.
To see this, consider a piecewise linear model of slope $\lambda$, hence of entropy $\log \lambda$, that has the same dynamics as $f$ above. 
If we suppose that $I_1$ has length $|I_1| = 1$, then the length of $f^q(I_1)$ is  $|f^q(I_1)| = |I_0 \cup I_1| = \lambda^q$
 hence $|I_0| = \lambda^q -1$ and so, since $f$ is piecewise linear with slope $\lambda$, we get $|I_2| = |f(I_1)| = \lambda |I_1| = \lambda$ 
 and $|f(I_0)| = \lambda |I_0|  = \lambda ( \lambda^q - 1)$. On the other hand, we know that $f(I_0) = I_0 \cup I_1 \cup I_2$, so $|f(I_0)| =  |I_0|+ |I_1| + |I_2| = \lambda^q + \lambda$ which implies, by comparing the previous two equations and dividing by $\lambda \neq 0$, that $\lambda^q = \lambda^{q-1} + 2$. This polynomial is also computed in \cite{JungBiaccessibility}.

%%%%%%%%%%%%%%%%%%%%%%%%%%%%%%%%%%%%%%%%%

\section{Background on spectral determinant and growth rates} \label{S:background-graph}

\subsection{Directed graphs}

A \emph{directed graph} is an ordered pair $\Gamma = (\mathcal{V},\mathcal{E})$ where $\mathcal{V}$ is a set and $\mathcal{E}$ is a subset of $\mathcal{V} \times \mathcal{V}$.  Elements of $\mathcal{V}$ are called \emph{vertices} and elements of $\mathcal{E}$ are called \emph{edges}.   Given an edge $e=(v_1,v_2)$, the \emph{source} of $e$, denoted $s(e)$, is the vertex $v_1$, and the \emph{target} of $e$, denoted $t(e)$, is the vertex $v_2$; we say that such an edge ``goes from $v_1$ to $v_2$.'' The \emph{outgoing degree} (resp. \emph{incoming degree}) of a vertex $v$, denoted $\textrm{Out}(v)$ (resp. $\textrm{In}(v)$), is the cardinality of the set of edges whose source (resp. target) is $v$.  A directed graph such that $\textrm{Out}(v)$ and $\textrm{In}(v)$ are finite for every $v \in \mathcal{V}$ is said to be \emph{locally finite}. 
A directed graph for which there exists $n \in \mathbb{N}$ such that $\textrm{Out}(v) \leq n$ for all $v \in \mathcal{V}$ is said to have \emph{bounded outgoing degree}.  A directed graph is \emph{countable} if $\mathcal{V}$ is countable.

\subsubsection{Paths and cycles} \label{ss:pathsandcycles} A \emph{path} in a directed graph based at a vertex $v$ is a sequence $(e_1,\ldots,e_n)$ of edges such that $s(e_1) = v$ and $t(e_i) = s(e_{i+1})$ for $1 \leq i \leq n-1$.  Such a path is said to have \emph{length} $n$ and its \emph{vertex support} is the set $\{s(e_1),\ldots,s(e_n)\} \cup \{t(e_n)\}$.  A \emph{closed path} based at $v$ is a path $e_1,\ldots,e_n$ such that $v=s(e_1) = t(e_n)$.  Note that a closed path can intersect itself, and closed paths based at different vertices are considered different.   A \emph{simple cycle} is a closed path which does not self-intersect, modulo cyclical equivalence (meaning two such paths are considered the same simple cycle if the edges are cyclically permuted).  A \emph{multicycle} is the union of finitely many simple cycles with pairwise disjoint vertex supports. The length of a multicycle is the sum of the lengths of the simple cycles that comprise it. A directed countable graph is said to have \emph{bounded cycles} if it has bounded outgoing degree and for each positive integer $n$, it has at most finitely many simple cycles of length $n$. 
 
 \subsubsection{Graph maps and quotients}
 Let $\Gamma_1=(\mathcal{V}_1,\mathcal{E}_1)$ and $\Gamma_2=(\mathcal{V}_2,\mathcal{E}_2)$ be locally finite directed graphs. A \emph{graph map} from $\Gamma_1$ to $\Gamma_2$ is a map $\pi:\mathcal{V}_1 \to \mathcal{V}_2$ such that for every edge $(v,w)$ in $\mathcal{E}_1$,  $(\pi(v_1),\pi(v_2))$ is an edge in $\mathcal{E}_2$.  We will denote such a map $\pi:\Gamma_1 \to \Gamma_2$.  Such a graph map $\pi$ also induces maps, which (abusing notation) we also denote by $\pi$, $\pi:\mathcal{E}_1 \to \mathcal{E}_2$ and $\pi:\textrm{Out}(v) \to \textrm{Out}(\pi(v))$ for each $v \in \mathcal{V}_1$.  A \emph{weak graph map} $\pi:\Gamma_1 \to \Gamma_2$ is a graph map $\pi:\Gamma_1 \to \Gamma_2$ such that  the map $\pi:\mathcal{V}_1 \to \mathcal{V}_2$ between vertex sets is surjective, and the induced map $\pi:\textrm{Out}(v) \to \textrm{Out}(\pi(v))$ is a bijection or each $v \in \mathcal{V}_1$. 

For a directed, locally finite graph $G=(\mathcal{V},\mathcal{E})$, an equivalence relation $\sim$ on $\mathcal{V}$ is called \emph{edge-compatible} if whenever $v_1 \sim v_2$, for every vertex $w\in \mathcal{V}$ the total number of edges from $v_1$ to members of the equivalence class of $w$ equals the total number of edges from $v_2$ to members of the equivalence class of $w$.  For such a graph $\Gamma = (\mathcal{V},\mathcal{E})$ and edge-compatible equivalence relation $\sim$, the \emph{quotient graph} $\overline{\Gamma} = \Gamma / \sim$ is defined as follows.  Define the vertex set of $\overline{\Gamma}$ to be the set $\overline{\mathcal{V}} := \mathcal{V} / \sim$, and for each pair of vertices $[v]$ and $[w]$ in the quotient graph, define the number of edges from $[v]$ to $[w]$ in the quotient graph to be the total number of edges from the fixed vertex $v \in \mathcal{V}$ to all members of the equivalence class of $w$ in $\mathcal{G}$.

  \subsubsection{Adjacency operator and incidence matrix}
  Given a (directed) finite or countable graph $\Gamma$ with vertex set $\mathcal{V}$ that has bounded outgoing degree, we define the \emph{adjacency operator} $A:\ell^1(\mathcal{V}) \to \ell^1(\mathcal{V})$ to be the linear operator  on $\ell^1(\mathcal{V})$ such that, denoting by $e_i$ the sequence that has $1$ at position $i$ and $0$ otherwise, the $j^{\textrm{th}}$ component of $A(e_i)$ is $(A(e_i))_j := \#(i \to j)$, the number of edges from $i$ to $j$. Note that for each pair $i,j$ of vertices and each $n$, the coefficient $(A^n(e_i))_j$ equals the number of paths of length $n$ from $i$ to $j$.  
  
  When $\Gamma$ is a finite graph, $\ell^1(\mathcal{V}) \cong \mathbb{R}^{|\mathcal{V}|}$ is a finite-dimensional vector space, with a privileged choice of basis $\{e_i, i \in \mathcal{V}\}$, and $A$ is a linear map; in this case, $A$ can be represented by a (finite) square matrix, with one row/column for each $v \in V$, which we call the \emph{incidence matrix} associated to $\Gamma$.  (We only define ``the'' incidence matrix up to permutation of the rows/columns, which will be sufficient for our purposes, since eigenvalues are invariant under elementary row operations.)  For such a finite graph $\Gamma$, the characteristic polynomial for the action of $A$ is the polynomial $\chi_{\Gamma}(t) = \det(A - tI)$.  

\subsection{Spectral determinant} \label{ss:spectraldeterminant}

For a directed finite or countable graph $\Gamma$ with bounded cycles, 
define the \emph{spectral determinant} $P_{\Gamma}(t)$ of $\Gamma$ by 
\begin{equation} \label{eq:spectraldeterminant}
P_{\Gamma}(t) := \sum_{\gamma \textrm{ multicycle}} (-1)^{C(\gamma)}t^{\ell(\gamma)},
\end{equation}
where $\ell(\gamma)$ denotes the length of the multicycle $\gamma$, while $C(\gamma)$ is the number of connected components of $\gamma$.  Note that the empty cycle is considered a multicycle, so $P_{\Gamma}$ starts with the constant term $1$. 
When $\Gamma$ is a finite directed graph, it is well known (see e.g. \cite[Section 1.2.1]{BrouwerHaemers}) that
$$P_{\Gamma}(t) = \det(I - t A) = (-t)^d \chi_\Gamma(t^{-1}),$$ 
i.e. the spectral determinant coincides, up to a factor of $(-t)^d$, with the (reciprocal of the) characteristic polynomial of the adjacency matrix.

%%%%%%%%%%%%%%%%%%%%%%%%%%%%%%%%%%%%%%  

\subsection{Labeled wedges and associated graphs}

We now introduce the notion of  labeled wedges, which will be then used to encode the combinatorial dynamics and compute the core entropy. 
The \emph{unlabeled wedge} is the set $$\Sigma \coloneqq \{(i,j) \in \mathbb{N}^2: 1 \leq i  \leq  j\}.$$ A \emph{labeling of $\Sigma$} is a map from $\Sigma$ to the $3$-element set $\{N,S,  E  \}$, i.e. a map $\Phi:\Sigma \to \{N,S, E\}$.  ($N$ stands for `non-separated,' $S$ for `separated,'  and $E$ for `equivalent.'). A \emph{labeled wedge} $\mathcal{W}$ is
the set $\Sigma$ together with a labeling $\Phi$ of $\Sigma$; we will write $(i,j) \in \mathcal{W}$ to mean the point $(i,j) \in \Sigma$ together with the data of the value $\Phi$ assigns to $(i,j)$. 
 
 Associated to any labeled wedge $\mathcal{W}$, there is an associated directed graph $\Gamma_{\mathcal{W}} = (\Sigma,E_{\mathcal{W}})$. The vertex set of $\Gamma_{\mathcal{W}}$ is the unlabeled wedge $\Sigma$.  The edge set $E_{\mathcal{W}}$ is defined recursively as follows.  For each vertex $(i,j) \in \Sigma$, 
\begin{itemize}
\item if $(i,j)$ is equivalent,  there is no edge in $E_{\mathcal{W}}$ with $(i,j)$ as its source;
\item if $(i,j)$ is non-separated, we add the edge $\left((i,j) , (i+1,j+1) \right)$ to $E_{\mathcal{W}}$;
\item if $(i,j)$ is separated, we add the edges $\left((i,j) , (1,j+1) \right)$ and $\left( (i,j) , (1,i+1) \right)$ to $E_{\mathcal{W}}$.
\end{itemize}

  We say that a sequence $(\mathcal{W}_n)_{n \in \mathbb{N}}$ of labeled wedges \emph{converges} to a labeled wedge $\mathcal{W}$ if for each finite set of vertices $V \subset \Sigma$ there exists $N \in \mathbb{N}$ such that for all $n \geq N$  the labels of the elements of $V$ for $\mathcal{W}_n$ and $\mathcal{W}$ are the same. 
  
 \begin{theorem}[\cite{TiozzoContinuity}, Theorem 4.3] \label{t:SpectralDeterminantGrowthRate}
 Let $\mathcal{W}$ be a labeled wedge. Then its associated graph $\Gamma_{\mathcal{W}}$ has bounded cycles, and its spectral determinant $P(t)$ defines a holomorphic function in the unit disk. Moreover, the growth rate $r$ of the graph $\Gamma_{\mathcal{W}}$ equals the inverse of the smallest real positive root of $P(z)$, in the following sense: $P(z) \neq 0$ for $|z|<r^{-1}$ and, if $r>1$, then $P(r^{-1})=0$.
 \end{theorem}
 
\begin{proposition}[\cite{TiozzoContinuity}, Proposition 4.2, part 5] \label{p:boundedcoeffs} \label{p:boundingnumberofcycles} 
For every labeled wedge $\mathcal{W}$ and $n \in \mathbb{N}$, the associated graph $\Gamma_{\mathcal{W}}$ has at most $(2n)^{\sqrt{2n}}$ multicycles of length $n$.  \end{proposition}
  
  \subsubsection{Periodic labeled wedges and finite models} \label{ss:periodiclabeledwedgesfinitemodels}
  Given integers $p \geq 1$ (the \emph{period}) and $q \geq 0$ (the \emph{preperiod}), let $\equiv_{p,q}$ be the equivalence relation on $\mathbb{N}$ defined by $i \equiv_{p,q} j$ if and only if either
  \begin{enumerate}
    \item  $i = j$, or 
 \item  $\min\{i,j\} \geq q+1$ and $i \equiv j \mod p$. 
  \end{enumerate}
  The equivalence relation $\equiv_{p,q}$ on $\mathbb{N}$  induces an equivalence relation on $\mathbb{N} \times \mathbb{N}$, which we also denote $\equiv_{p,q}$, by setting
  $(i,j) \equiv_{p,q} (k,l)$ if and only if $i \equiv_{p,q} k$ and  $j \equiv_{p,q} l$.  It also induces an equivalence relation on the set of unordered pairs of natural numbers by declaring that the unordered pair $\{i,j\}$ is equivalent to to $\{k,l\}$ if either $(i,j) \equiv_{p,q} (k,l)$ or $(i,j) \equiv_{p,q} (l,k)$. 
  
  \begin{definition} \label{D:periodiclabeledwedge}
  A labeled wedge is \emph{periodic of period $p$ and preperiod $q$} if and only if 
  \begin{enumerate} 
  \item \label{i:correspondingverticessamelabel} any two pairs $(i,j)$ and $(k,l)$ such that $\{i,j\} \equiv_{p,q} \{k,l\}$ have the same label, 
  \item a point $(i,j)$ is labeled $E$ if and only if $i \equiv_{p,q} j$, and 
\item if $i \equiv_{p,q} j$, then the pair $(i,j)$ is non-separated. 
\end{enumerate}
We say that a labeled wedge is \emph{purely periodic} if its preperiod is zero. 
\end{definition}
 
  The \emph{finite model} of a countably infinite, directed graph $\Gamma_{\mathcal{W}}$ associated to a periodic labeled wedge $\mathcal{W}$ of period $p$ is and preperiod $q$ is the quotient graph of $\Gamma_{\mathcal{W}} / \equiv_{p,q}$. (Figure \ref{f:1/5} shows the finite graph model for angle $\theta=1/5$.)

 \subsection{The Thurston entropy algorithm}

\subsubsection{The labeled wedge associated to a rational angle}
  
 For any angle $\theta \in \mathbb{R}/\mathbb{Z}$, define $P_{\theta}$ to be the partition of $\mathbb{R}/\mathbb{Z}$ into 
$$\left[ \frac{\theta}{2}, \frac{\theta+1}{2} \right) \sqcup \left[ \frac{\theta+1}{2}, \frac{\theta}{2} \right),$$
and for each $i \in \mathbb{N}$, set $x_i(\theta) :=2^{i-1}\theta \mod 1$.  Define $\mathcal{W}_{\theta}$ to be the labeled wedge defined by labeling each pair $(i,j) \in \Sigma$ as  \emph{equivalent} if $x_i(\theta) = x_j(\theta)$, as \emph{separated} if $x_i(\theta)$ and $x_j(\theta)$ are in the interiors of different elements of the partition $P_{\theta}$, and labeling the pair as \emph{non-separated} otherwise.  Denote by $\Gamma_{\theta}$ the associated infinite directed graph.  (Figure  \ref{f:1/9} shows a portion of the infinite graph $\Gamma_{1/9}$.)

\begin{figure}[h!]
\begin{tikzpicture}[node distance = 1.3cm]

\tikzstyle{non-sep} = [rectangle, rounded corners, minimum width=.5cm, minimum height=.5cm,text centered, draw=black, fill=none]
\tikzstyle{sep} = [rectangle, dashed, rounded corners, minimum width=.5cm, minimum height=.5cm, text centered, draw=black, fill=none]
\tikzstyle{process} = [rectangle, minimum width=3cm, minimum height=1cm, text centered, draw=black, fill=orange!30]
\tikzstyle{decision} = [diamond, minimum width=3cm, minimum height=1cm, text centered, draw=black, fill=green!30]
\tikzstyle{thickarrow} = [thick,->,>=stealth]
\tikzstyle{arrow} = [->,>=stealth]

\node (11) [non-sep] {1,1};
\node (12) [non-sep, right of = 11] {1,2};
\node (13) [non-sep, right of=12] {1,3};
\node (14) [sep, right of=13] {1,4};
\node (15) [sep, right of=14] {1,5};
\node (16) [non-sep, right of=15] {1,6};

\node (17) [non-sep, right of=16] {1,7};
\node(rdots) [right of = 17]{$\ldots$};

\node (22) [non-sep, above of = 12] {2,2};
\node (23) [non-sep, above of=13] {2,3};
\node (24) [sep, above of=14] {2,4};
\node (25) [sep, above of=15] {2,5};
\node (26) [non-sep, above of=16] {2,6};

\node (27) [non-sep, above of=17] {2,7};

\node (33) [non-sep, above of = 23] {3,3};
\node (34) [sep, above of=24] {3,4};
\node (35) [sep, above of=25] {3,5};
\node (36) [non-sep, above of=26] {3,6};

\node (37) [non-sep, above of=27] {3,7};

\node (44) [non-sep, above of = 34] {4,4};
\node (45) [non-sep, above of=35] {4,5};
\node (46) [non-sep, above of=36] {4,6};

\node (55) [non-sep, above of = 45] {5,5};
\node (56) [non-sep, above of=46] {5,6};

\node (66) [non-sep, above of=56] {6,6};

\node (27) [non-sep, above of=17] {2,7};
\node (37) [non-sep, above of=27] {3,7};
\node (47) [non-sep, above of=37] {4,7};
\node (57) [non-sep, above of=47] {5,7};
\node (67) [non-sep, above of=57] {6,7};
\node (77) [non-sep, above of=67] {7,7};\
\node (upperdots) [above right of=77] {$\iddots$};
\node (upperdots) [ right of=77] {$\ldots$};

\draw [arrow] (12) to (23);
\draw [arrow] (23) -- (34);
\draw [arrow] (34) -- (15);
\draw [arrow] (13) to [bend left] (24);
\draw [arrow] (24) to [bend left]  (13);
\draw [arrow] (24) to (15);
\draw [arrow] (34) to [out=-70,in=70, looseness=1.8]  (14);
\draw [arrow] (14) to [out = 210, in=-30, looseness=1] (12);
\draw [arrow] (14) to (15);
\draw [arrow] (15) to (16);
\draw [arrow] (25) to (16);
\draw [arrow] (35) to (16);
\draw [arrow] (35) to (14);
\draw [arrow] (45) to (56);
\draw [arrow] (15) to [out=210, in=-40] (12);
\draw [arrow] (25) to (13);

\draw [arrow] (16) to (27);
\draw [arrow] (26) to (37);
\draw [arrow] (36) to (47);
\draw [arrow] (46) to (57);
\draw [arrow] (56) to (67);

\draw[dotted] ([xshift=.75cm]16.south) -- ([xshift=.75cm]66.north);

 %%%%%%%%%%%%%%%%%%%%%%%%%%%%%%%%%%%%%%

\begin{scope}[yshift=6cm, scale=1.25]
\draw (0,0) circle (1);
\filldraw [black] (360/9:1) circle (1pt);
\node at (360/9:1.5)  {$x_1 {=} \tfrac{1}{9}$};
\filldraw [black] (360*2/9:1) circle (1pt);
\node at (360*2/9:1.3)  {$x_2 {=} \tfrac{2}{9}$};
\filldraw [black] (360*4/9:1) circle (1pt);
\node at (360*4/9:1.5)  {$x_3 {=} \tfrac{4}{9}$};
\filldraw [black] (360*8/9:1) circle (1pt);
\node at (360*8/9:1.3)  {$x_4 {=} \tfrac{8}{9}$};
\filldraw [black] (360*7/9:1) circle (1pt);
\node at (360*7/9:1.3)  {$x_5 {=} \tfrac{7}{9}$};
\filldraw [black] (360*5/9:1) circle (1pt);
\node at (360*5/9:1.5)  {$x_6 {=} \tfrac{5}{9}$};

\draw[dashed] (360/18:1.5) to (360/18+180:1.5);
\node at (360/18:1.5)  {$\tfrac{\theta}{2} {=} \tfrac{1}{18}$};
\end{scope}

 %%%%%%%%%%%%%%%%%%%%%%%%%%%%%%%%%%%%%%

\end{tikzpicture}
\caption{The infinite graph $\Gamma_{1/9}$. Angle $1/9$ is periodic with period $6$ for the doubling map; the vertical dotted line indicates the edge of a ``fundamental domain'' for $\equiv_{6,0}$. Vertices that are separated are indicated with a dashed boundary; non-separated and equivalent vertices have a solid boundary.  The angle diagram in the upper left is helpful for determining which vertices of $\Gamma$ are separated. }
\label{f:1/9}
\end{figure}
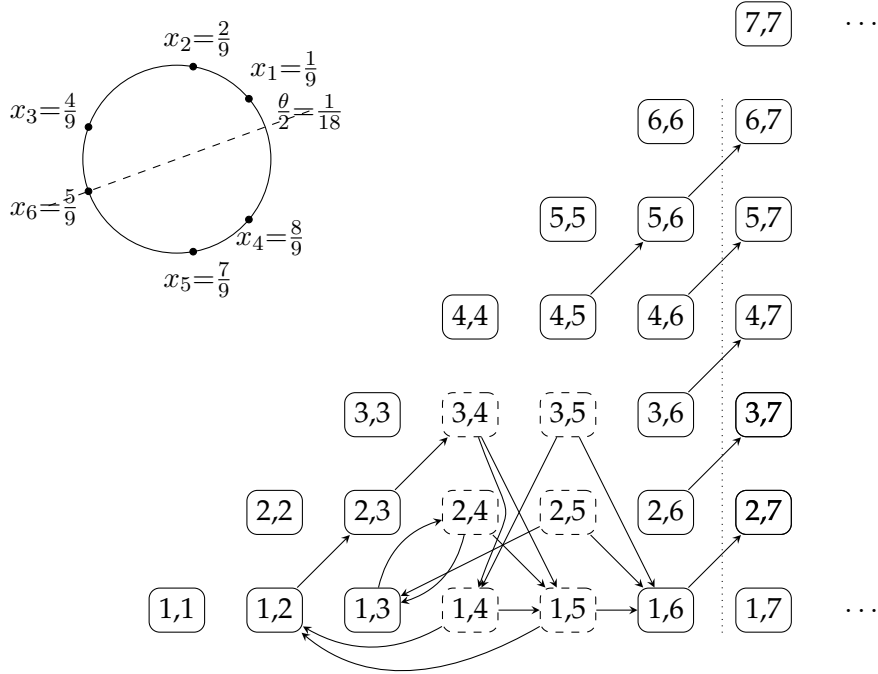

\subsubsection{Growth rate and core entropy} \label{sss:growthratecoreentropy}
 
Given a finite or infinite graph $\Gamma$ with bounded cycles, we define its \emph{growth rate} as 
$$r := \limsup_{n \to \infty} \sqrt[n]{C(\Gamma, n)}$$
where $C(\Gamma, n)$ is the number of closed paths in $\Gamma$ of length $n$.  It is straightforward to prove that the growth rate of a finite graph is the leading eigenvalue of its incidence matrix. 

\begin{proposition}[\cite{TiozzoContinuity}, Proposition 5.2] \label{p:GrowthRateFiniteModel}
Let $\mathcal{W}$ be a periodic labeled wedge, with associated (infinite) graph $\Gamma$.  Then the growth rate of $\Gamma$ equals the growth rate of its finite model. 
\end{proposition} 
 
When the infinite graph is a wedge coming from a rational angle, the logarithm of the growth rate yields precisely the core entropy 
of the corresponding PCF quadratic polynomial.

\begin{theorem}[\cite{TiozzoContinuity}, Theorem 6.4]
Let $\theta$ be a rational angle.  Then the logarithm of the growth rate $r(\theta)$ of the infinite graph $\Gamma_{\theta}$ coincides with core entropy: $h(\theta) = \log r(\theta)$. 
\end{theorem} 

\noindent Note that here $h(\theta)$ denotes the core entropy of the critically periodic polynomial associated to $\theta$, as described in \S \ref{ss:anglestotrees}.

For a postcritically finite polynomial $f$, we define the \emph{Thurston polynomial} $P_{Th}(t)$ to be the characteristic polynomial of the incidence matrix for the finite model graph associated to $f$. 

%%%%%%%%%%%%%%%%%%%%%%%%%%%%%%%%%%%%%%%
\section{Relating the Thurston polynomial and the Markov polynomial} \label{s:relatingThurstonAndMarkovPolys}

The goal of this section is to prove the following comparison between the Thurston algorithm polynomial and the Markov polynomial, 
establishing the first part of Theorem \ref{T:equalpolys} from the introduction:

\begin{theorem} \label{T:Th-Mar} 
For a postcritically finite parameter in the Mandelbrot set, the polynomial $P_{Th}(t)$ that we get from Thurston's algorithm 
and the polynomial $P_{Mar}(t)$ that we get from the Markov partition satisfy the following relation: 
$$P_{Th}(t) = P_{Mar}(t) Q(t)$$
where $Q(t)$ is a polynomial whose roots lie in $\{ 0 \} \cup S^1$. 
\end{theorem}

As defined in \S  \ref{sss:growthratecoreentropy}, the Thurston polynomial $P_{Th}(t)$ associated to a postcritically finite polynomial $f$ is the characteristic polynomial of the incidence matrix for the finite graph model associated to $f$.   The Markov polynomial $P_{Mar}(t)$ associated to the postcritically finite polynomial $f$ is the characteristic polynomial of the incidence matrix for the Markov partition of the Hubbard tree $T_f$ formed by cutting $T_f$ at its branch points and at its postcritical set (including the critical point).  

\bigskip
\begin{example} We compute the Thurston polynomial $P_{Th}(t)$ and Markov polynomial $P_{Mar}(t)$ for angle $\theta=1/5$.  \medskip

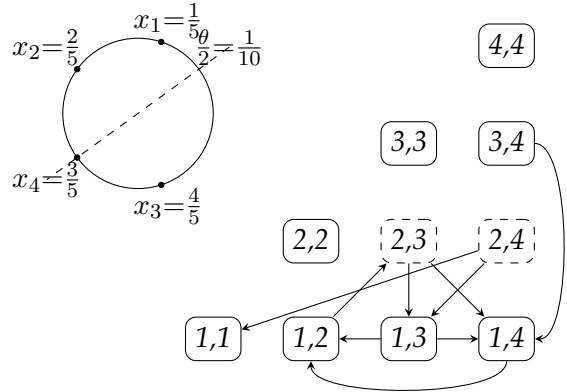
\begin{figure}[h!]
\begin{tikzpicture}[node distance = 1.3cm]

\tikzstyle{non-sep} = [rectangle, rounded corners, minimum width=.5cm, minimum height=.5cm,text centered, draw=black, fill=none]
\tikzstyle{sep} = [rectangle, dashed, rounded corners, minimum width=.5cm, minimum height=.5cm, text centered, draw=black, fill=none]
\tikzstyle{process} = [rectangle, minimum width=3cm, minimum height=1cm, text centered, draw=black, fill=orange!30]
\tikzstyle{decision} = [diamond, minimum width=3cm, minimum height=1cm, text centered, draw=black, fill=green!30]
\tikzstyle{thickarrow} = [thick,->,>=stealth]
\tikzstyle{arrow} = [->,>=stealth]

\node (11) [non-sep] {1,1};
\node (12) [non-sep, right of = 11] {1,2};
\node (13) [sep, right of=12] {1,3};
\node (14) [non-sep, right of=13] {1,4};
\node (22) [non-sep, above of = 12] {2,2};
\node (23) [sep, above of=13] {2,3};
\node (24) [non-sep, above of=14] {2,4};
\node (33) [non-sep, above of = 23] {3,3};
\node (34) [non-sep, above of=24] {3,4};
\node (44) [non-sep, above of = 34] {4,4};

\draw [arrow] (12) to (23);
\draw [arrow] (23) -- (14);
\draw [arrow] (23) -- (13);
\draw [arrow] (13) -- (12);
\draw [arrow] (13) -- (14);
\draw [arrow] (14) to [out=-90,in=-90, looseness=.5]  (12);
\draw [arrow] (34) to [out=-0,in=0, looseness=.5]  (14);
\draw [arrow] (24) to (13);
%\draw [arrow] (24) to (11);

\begin{scope}[yshift=3cm, xshift=-1cm]
\draw (0,0) circle (1);
\filldraw [black] (360/5:1) circle (1pt);
\node at (360/5:1.3)  {$x_1 {=} \tfrac{1}{5}$};
\filldraw [black] (360*2/5:1) circle (1pt);
\node at (360*2/5:1.5)  {$x_2 {=} \tfrac{2}{5}$};
\filldraw [black] (360*4/5:1) circle (1pt);
\node at (360*4/5:1.3)  {$x_3 {=} \tfrac{4}{5}$};
\filldraw [black] (360*3/5:1) circle (1pt);
\node at (360*3/5:1.5)  {$x_4 {=} \tfrac{3}{5}$};
\draw[dashed] (360/10:1.5) to (360/10+180:1.5);
\node at (360/10:1.5)  {$\tfrac{\theta}{2} {=} \tfrac{1}{10}$};
\end{scope}
\end{tikzpicture}
\caption{The finite model graph for angle $\theta = 1/5$.}
\label{f:1/5}
\end{figure}

\noindent \textbf{Thurston polynomial $P_{Th}(t)$:} Figure \ref{f:1/5} shows the finite graph model for angle $\theta=1/5$.  The adjacency matrix for this finite graph (omitting rows/columns for vertices of the form $(i,i)$ for $i=1,2,3,4$, which have no incident edges) is:

\begin{center}
\begin{tabular}{ c|cccccc } 
& $(1,2)$ & $(1,3)$ & $(1,4)$ & $(2,3)$ & $(2,4)$ & $(3,4)$  \\ 
 \hline
$(1,2)$ & $0$ & $0$ & $0$ & $1$ & $0$ & $0$ \\ 
$(1,3)$ &$1$& $0$ & $1$ & $0$&$0$ & $0$ \\ 
$(1,4)$ & $1$ & $0$ & $0$ &  $0$ &$0$  & $0$ \\ 
$(2,3)$ & $0$ & $1$ & $1$ &  $0$ & $0$  & $0$ \\ 
$(2,4)$ & $0$ & $1$ & $0$ &  $0$ &$0$  & $0$ \\ 
$(3,4)$ & $0$ & $0$ & $1$ &  $0$ & $0$  & $0$ \\ 
\end{tabular}
\end{center}

The Thurston polynomial is the characteristic polynomial of the matrix above (padded with $0$s to have 4 more rows and columns, representing the vertices $(1,1),\ldots,(4,4)$):
$$P_{Th}(t) = t^6 + 2 t^7 - t^{10}.$$

\medskip 

\noindent \textbf{Markov polynomial $P_{Mar}(t)$:} Figure \ref{f:Htree1/5} shows the combinatorial model of the Hubbard tree for angle $1/5$.   The incidence matrix for the dynamics on this combinatorial Hubbard tree is:

\begin{center}
\begin{tabular}{ c|cccc } 
& $I_1 $ & $I_2$ & $I_3$ & $I_0$  \\ 
 \hline
$I_1$ & $0$ & $1$ & $0$ & $0$ \\ 
$I_2$ &$0$& $0$ & $1$ & $0$ \\ 
$I_3$ & $1$ & $0$ & $0$ &  $1$  \\ 
$I_0$ & $1$ & $1$ & $0$ &  $0$  \\ 
\end{tabular}
\end{center}
The Markov polynomial for angle $1/5$ is the characteristic polynomial of the incidence matrix above: 
$$P_{Mar}(t) = -1 - 2 t + t^4.$$
\end{example}
 
\medskip

\subsection{Setup} We will use the notation established below in the remainder of this section. 

Let $f$ be a postcritically finite polynomial with Hubbard tree $T_f$. Let $P$ be the union of the postcritical set, the critical point and the branch points of the Hubbard tree.
Define $\sigma(\theta) := 2 \theta \mod 1$. 
Let $\hat{\theta}$ be the characteristic angle of $f$, whose corresponding 
external dynamical ray $R_c(\hat{\theta})$  lands at the critical value (or whose external parameter ray $R_{\mathcal{M}}(\hat{\theta})$ lands at the root of the hyperbolic component containing the critical value). 

Let us consider $\mathbb{T} = \mathbb{R}/\mathbb{Z} \to J(f)$ the Carath\'eodory loop, 
sending an angle to the landing point of the corresponding external ray. Given an angle $\alpha \in \mathbb{T}$, 
we denote as $\overline{\alpha}$ the landing point of the ray of angle $\alpha$. 
Elements of $\mathbb{T}$ will be denoted by Greek letters $\xi, \eta, \alpha \dots$ and points in the Hubbard tree by lower-case Latin letters $x, y, z \dots$. Given two points  $x, y$, we denote by $[x, y]$ the closed arc in $T_f$ joining $x$ and $y$, 
and by $(x, y)$ the open arc. Finally, an unordered pair of angles will be denoted $\{\xi, \eta \}$, while we denote $[\xi, \eta]$
the arc in the Hubbard tree joining the landing points $\overline{\xi}$ and $\overline{\eta}$. 

\smallskip

Let $\mathcal{E}_1$ be the set of connected components of $T_f \setminus P$, and $E_1 := \mathbb{R}^{\mathcal{E}_1}$. 
Let $T_1 : E_1 \to E_1$ be the linear map induced by the action of the dynamics on the set $\mathcal{E}_1$, 
and $A_1$ be the matrix representing $T_1$ in the basis $\mathcal{E}_1$. 
By construction, this is the matrix obtained by the Markov partition. 
Hence the associated polynomial is 
$$P_{Mar}(t) := \det( I - t A_1).$$

Let $\mathcal{E}_2$ denote the set of unordered pairs of postcritical angles,
$E_2 := \mathbb{R}^{\mathcal{E}_2}$. 
An element of $E_2$ will be called an \emph{angle pair configuration} and can be written as 
$$v = \sum_i a_i \{ \xi_i, \eta_i \}$$
with $a_i \in \mathbb{R}$. 
The \emph{formal support} of $v$ is the set of pairs $\{\xi_i, \eta_i \}$ for which $a_i \neq 0$. 
The \emph{geometric support} of $v$ is the union of the arcs $[\overline{\xi_i}, \overline{\eta_i}]$ in the tree. 

Let $T_2 : E_2 \to E_2$ be the linear map 
induced by the action of the dynamics on the set $\mathcal{E}_2$, 
%i.e. $T_2(\{\xi, \eta\}) = \{ \sigma\xi, \sigma \eta \}$ for any $\xi, \eta$.
and $A_2$ be the matrix representing $T_2$ in the basis $\mathcal{E}_2$. 
Hence the polynomial given by Thurston's algorithm is 
$$P_{Th}(t) := \det( I - t A_2).$$

\subsection{Semiconjugacy of $T_1$ and $T_2$ via elementary decomposition} 

\begin{definition} [elementary decomposition map] \ 
\begin{enumerate}
\item We first define a map $\pi: \mathcal{E}_2 \to \mathcal{E}_1$ as follows.  
For any pair $\{\xi_1, \xi_2 \} \in \mathcal{E}_2$ of  rational   postcritical angles, write $x = \overline{\xi_1}$, $y = \overline{\xi_2}$ for the landing points.  
If $x = y$, define $\pi(\{ \xi_1, \xi_2 \}) = 0$. 
If $x \neq y$, let 
$$(x, y) \setminus P = (p_1, p_2) \cup \dots \cup (p_{k-1}, p_k)$$
be the decomposition of $(x, y) \setminus P$ in its connected components, where each $p_i \in P$, and let
$$\pi(\{ \xi_1, \xi_2\}) := \sum_{i = 1}^{k-1} [p_i, p_{i+1}].$$
\item Define the \emph{elementary decomposition} map $\pi : E_2 \to E_1$ to be the linear extension to $E_2$ of the action of $\pi$ acting on $\mathcal{E}_2$. 
\end{enumerate}
\end{definition}

Note that, by construction, the following diagram commutes:
\[\begin{tikzcd}
E_2 \arrow{r}{T_2} \arrow[swap]{d}{\pi} & E_2  \arrow{d}{\pi} \\
E_1 \arrow{r}{T_1}  & E_1
\end{tikzcd}\]
\noindent Therefore,
$$T_2(\ker \pi) \subseteq \ker \pi.$$
As a consequence, if 
$$Q(t) \coloneqq \det(I - t A_2\vert_{\ker \pi})$$ is the characteristic polynomial of the action of $T_2$ on $\ker \pi$, we have the identity
$$P_{Th}(t) = P_{Mar}(t) Q(t).$$
Thus, to prove Theorem \ref{T:Th-Mar}, it suffices to establish the following: 

\begin{proposition} \label{L:Q-cycl}
All non-zero roots of $Q(t)$ lie on the unit circle. 
\end{proposition}

We will use the following lemma from linear algebra:

\begin{lemma} \label{L:finite-cyclo}
Let $T : E \to E$ be a linear map of a finite-dimensional vector space, and suppose that there exists a finite set 
$S \subseteq E$ such that $T(S) \subseteq S$ and so that the span of $S$ equals $E$. 
Then the characteristic polynomial of $T$ is of the form 
$$P(t) = t^k \chi(t),$$ 
where $k \geq 0$ and $\chi(t)$ is the product of cyclotomic polynomials. 
\end{lemma}

\begin{proof}
Since $T(S) \subseteq S$, then for any $v \in S$ the set $\{ T^n(v) \ : \ n \geq 0\} \subseteq S$ is finite. 
Moreover, since $S$ spans $E$, for any $v \in E$
we can write $v = \sum_i a_i v_i$ with $a_i \in \mathbb{C}$ and $v_i \in E$. 
Hence also 
$$T^n(v) = \sum_i a_i T^n(v_i) \subseteq \sum_i a_i S,$$
which is a finite set. 
Now, suppose that $v$ is an eigenvector of $T$, with eigenvalue $\lambda$. Then note that the set 
$\{ T^n(v) = \lambda^n v \ : \ n \geq 0\}$ is finite only if either $\lambda = 0$ or if $\lambda$ is a root of unity. 
This proves the claim. 
\end{proof}

To prove Proposition \ref{L:Q-cycl}, we apply Lemma \ref{L:finite-cyclo} to the action of $T_2$ on $\ker \pi$ with of ``elementary triples'' and ``elementary stars'' serving as the set $S \subseteq \ker \pi$; we define these objects and prove that they have the requisite properties in the next subsection, \S \ref{ss:elementarytriplesandstars}.

\begin{proof}[Proof of Proposition \ref{L:Q-cycl}]
By Lemma \ref{L:triples-stars}, %\ref{L:kernel-combo}, 
the space $\ker \pi$ is the span of the union of the set $\mathcal{T}$ of elementary triples 
and the set $\mathcal{S}$ of elementary stars of norm at most $M$.
Note that the set $\mathcal{T} \cup \mathcal{S}$ is finite, since $f$ is postcritically finite. 
Since by Lemmas \ref{L:permute} and \ref{L:S-forward-inv}, $T_2$ maps each element of $\mathcal{T} \cup \mathcal{S}$ to either $0$ or an element of $\mathcal{T} \cup \mathcal{S}$, 
then by Lemma \ref{L:finite-cyclo} the characteristic polynomial $Q(t)$ of the restriction of $T_2$ to $\ker \pi = \mathbb{R}^{\mathcal{T} \cup \mathcal{S}}$ 
is the product of $t^k$ for some $k \geq 0$ and cyclotomic polynomials. 
\end{proof}

 \subsection{Elementary triples and elementary stars} \label{ss:elementarytriplesandstars} 
 
\subsubsection{Elementary triples}

\begin{definition} 
Define an \emph{elementary triple} of angles to be a linear combination $v \in E_2$ of $3$ angle pairs of the form 
$$v = p_1+p_2-p_3$$ 
with $p_1= \{ \alpha, \beta \}, p_2= \{ \beta, \gamma \} $ and $p_3 = \{ \alpha, \gamma \}$ are elements of $\mathcal{E}_2$ 
and so that
$\overline{\beta}$ lies on the arc $[\overline{\alpha}, \overline{\gamma}]$. 
Note that, as a special case, if the rays at angle $\alpha, \alpha'$ land at the same point, setting $\alpha = \beta$ and $\gamma = \alpha'$ we obtain that the pair $\{ \alpha, \alpha' \}$ is also an elementary triple. We call such triple \emph{degenerate}. 
We denote as $\mathcal{T}$ the set of elementary triples.
\end{definition}

It is immediate to check that every elementary arc lies in the kernel of $\pi$. We now start with the following: 

\begin{lemma} \label{L:kernel-combo}
Every angle pair configuration which lies in $\ker \pi$ and whose support is contained in a segment 
is a linear combination of elementary triples. 
\end{lemma}

\begin{proof}
Since the coefficients of $T_2$ are integers, if there exists a non-zero element of $\ker \pi$ there exists an 
element of $\ker \pi$ which is a linear combination of angle pairs $\{\xi_i, \eta_i\}$ with rational coefficients. 
By multiplying all coefficients by a suitable integer, we can assume that there exists a linear combination with integer 
coefficients. 
Suppose that there exists a linear combination of the form
$$v = \sum_{i = 1}^k a_i  \{ \xi_i, \eta_i \}$$
with $a_i \in \mathbb{Z}$ and $ \{ \xi_i, \eta_i \} \in \mathcal{E}_2$, so that $v$ lies in $\ker \pi$ but not in the span $\textrm{span}(\mathcal{T})$ of the elementary triples. 

First, note that, if $\alpha, \alpha'$ land at the same point, then for any $\beta$ 
$$\{ \alpha, \beta \} - \{ \alpha', \beta \} = (\{ \alpha, \beta \} - \{ \alpha', \beta \} + \{ \alpha, \alpha' \} ) - (\{ \alpha, \alpha' \})$$ 
is a sum of elementary triples, hence, by subtracting elementary triples, we can assume that at most one angle
in the formal support of $v$ lands at each point in the geometric support of $v$. 

Now, let us choose a configuration $v$ for which the weight $\Vert v \Vert := \sum_{i = 1}^k |a_i|$ is minimal. 

Let $a$ be an end of the geometric support of $v$. Since $v$ lies in the kernel, there exists two 
elements $\{ \alpha, \beta \}$, $\{ \alpha, \gamma \}$ in the formal support of $v$ with coefficients of opposite signs 
and so that $\overline{\alpha} = \overline{\alpha} = a$. 
Suppose by symmetry that $b = \overline{\beta}$ lies in $[\overline{\alpha}, \overline{\gamma}]$. 
Then 
$$v = a_1 \{ \alpha, \beta \} + a_2 \{ \alpha, \gamma\} + \sum_{i = 3}^k a_i \{ \xi_i, \eta_i \}$$
and, up to changing $v$ with $-v$, we can assume that $a_1 > 0$, $a_2 < 0$. 
Then we can write $v$ as 
\begin{equation} \label{E:sum}
v = v_1 + v_2
\end{equation}
where 
\begin{align*}
v_1 & =\{\alpha, \beta\} + \{\beta, \gamma\} - \{\alpha, \gamma\} \\
v_2 & =  (a_1-1) \{\alpha, \beta\} + (a_2 + 1) \{\alpha, \gamma\} + \sum_{i = 3}^k a_i \{ \xi_i, \eta_i \} - \{\beta, \gamma\}.
\end{align*}
Now, by \eqref{E:sum}, $v_2$ also lies in $\ker \pi$; moreover, its weight satisfies 
$$\Vert v_2 \Vert \leq |a_1| -1  + |a_2| - 1 + \sum_{i = 3}^k |a_i| + 1 < \Vert v \Vert$$
hence it has lower weight than $v$; thus, by minimality, $v_2$ must belong to $\textrm{span}(\mathcal{T})$. However, 
since $v_1$ also belongs to $\textrm{span}(\mathcal{T})$, by \eqref{E:sum} we also have that $v$ belongs to $\textrm{span}(\mathcal{T})$, contradicting our 
hypothesis. 
\end{proof}

\begin{lemma} \label{L:permute}
The map $T_2$ sends every elementary triple to either $0$ or an elementary triple. 
\end{lemma}

\begin{proof}
Let $v = \{ \alpha, \beta \} +  \{ \beta, \gamma \} -  \{ \alpha, \gamma\}$ be an elementary triple, so that  
$b:= \overline{\beta}$ lies in $[\overline{\alpha}, \overline{\gamma}]$. 

If $\{ \alpha, \gamma\}$ is non-separated, then 
$$T_2(v) = \{ \sigma(\alpha), \sigma(\beta) \} + \{ \sigma(\beta), \sigma(\gamma) \} -  \{ \sigma(\alpha), \sigma(\gamma) \}$$
is clearly an elementary triple. 

If $\{ \alpha, \gamma\}$ is separated, then either $b$ is the critical point, $\{ \alpha, \beta\}$ is separated, or $\{\beta, \gamma\}$ is separated. 
If $\{ \alpha, \beta\}$ is  separated, then 
\begin{align*}
T_2(\{ \alpha, \beta\}) & = \{ \sigma(\alpha), \hat{\theta} \} + \{ \hat{\theta}, \sigma(\beta) \} \\
T_2(\{\beta, \gamma\}) & = \{ \sigma(\beta), \sigma(\gamma) \} \\
T_2(\{ \alpha, \gamma\}) & = \{ \sigma(\alpha), \hat{\theta} \} + \{ \hat{\theta}, \sigma(\gamma) \}
\end{align*}
where $v$ is the critical value; hence
\begin{align*}
T_2(v) & =  \{ \sigma(\alpha), \hat{\theta} \} + \{ \hat{\theta}, \sigma(\beta) \} +\{ \sigma(\beta), \sigma(\gamma) \} - 
(\{ \sigma(\alpha), \hat{\theta} \} + \{ \hat{\theta}, \sigma(\gamma) \}) \\
& =  \{ \hat{\theta}, \sigma(\beta) \} + \{ \sigma(\beta), \sigma(\gamma) \} - \{ \hat{\theta}, \sigma(\gamma) \}
\end{align*}
which is also an elementary triple. The case of $\{\beta, \gamma\}$ is symmetric. Finally, if $b$ is the critical point, then 
$\{\alpha, \beta\}$ and $\{\beta, \gamma\}$ are not separated, hence 
$$T_2(v) = \{\sigma(\alpha), \sigma(\beta)\} + \{\sigma(\beta), \sigma(\gamma)\} - (\{\sigma(\alpha), \sigma(\beta)\} + \{\sigma(\beta), \sigma(\gamma)\}) = 0.$$
\end{proof}

\subsubsection{Elementary stars}

We now need to take care of branch points in the Hubbard tree $T_f$. 
Recall that the \emph{valence} of a point $x \in T_f$ is the number of connected components of $T_f \setminus \{ x \}$. 
A \emph{branch point} is a point of valence larger than $2$. 

\begin{definition}
Let $x$ be a branch point of $T$. A set of angles $\Theta := \{ \theta_1, \dots, \theta_k\} \subseteq \mathbb{T}$ 
form a \emph{star centered at }$x$ if for any pair of distinct elements of $\Theta$, the corresponding external rays 
lie in different connected components of $T_f \setminus \{ x \}$.
A formal sum 
$$S = \sum_{i < j} a_{\{i,j\}} \{ \theta_i, \theta_j \}, \qquad a_{\{i,j\}} \in \mathbb{Z}$$
is an \emph{elementary star} if there exists a branch point $x$ and a star $\Theta$ centered at $x$
so that each $\theta_i$ lies in $\Theta$. 
The \emph{norm} of a star $S$ is $\max_{i,j} |a_{\{i,j\}}|$. A star has \emph{zero geometric weight} if $\sum_j a_{\{i,j\}} = 0$ 
for each $i$. 
\end{definition}
 
\begin{lemma} \label{L:elem-star}
There exists $M \geq 1$ such that any elementary star with geometric weight $0$ is a linear combination 
of elementary stars of geometric weight zero with norm at most $M$. 
\end{lemma}

\begin{proof}
Let $x$ be a branch point of $T_f$, and $m$ its valence. The map $A: \mathbb{Q}^{{m \choose 2}} \to \mathbb{Q}^m$ 
defined by $A(e_{\{i,j\}}) := e_i + e_j$
has a finite dimensional kernel. Let $v_1, \dots, v_h$ be a basis for the kernel $K$. By clearing denominators, we obtain a basis $v_1', \dots, v_h'$ of $K$ with integer coefficients. 
An elementary star supported on the neighborhood of $x$ yields an element of $K$. 
Let $M$ be the 
largest norm of all $v_i'$. Since this bound depends only on the valence of the branch point, and there are finitely many 
branch points in $T_f$, the claim follows. 
\end{proof}

\begin{lemma} \label{L:S-forward-inv}
Let $\mathcal{S}$ be the set of all elementary stars of norm at most $M$. Then $T_2(\mathcal{S}) \subseteq \mathcal{S}$. 
\end{lemma}

\begin{proof}
Consider a star $S = \sum_{i < j} a_{\{i,j\}} \{\theta_i, \theta_j\}$, centered at a branch point $x$. There are two cases.

If the critical point does not lie in the interior of the support of $S$, then every arc in $S$ is non-separated, hence 
the image of $S$ is an elementary star with the same norm. 

Otherwise, if the critical point $c$ lies in the interior of the star, let us say, up to relabeling, 
that $c$ lies on the segment $[x, \overline{\theta_1}]$. Then all pairs $\{ \theta_i, \theta_j \}$ of $S$ 
which contain the index $1$ are separated. Hence 
\begin{align*}
T_2(S) & = \sum_{i, j \neq 1} a_{\{i,j\}} \{ \sigma(\theta_i), \sigma(\theta_j) \} + \sum_j a_{\{1, j\}} \{ \sigma(\theta_j), \hat{\theta} \} + \left(  \sum_j a_{\{1, j\}} \right) \{ \sigma(\theta_1), \hat{\theta} \} \\ 
& = \sum_{i, j \neq 1} a_{\{i,j\}} \{ \sigma(\theta_i), \sigma(\theta_j)\} + \sum_j a_{\{1, j\}} \{ \sigma(\theta_j), \hat{\theta} \}
\end{align*}
since the geometric weight is zero. Since $\{\hat{\theta}, \sigma(\theta_2), \dots, \sigma(\theta_k) \}$ is a star centered at $f(x)$, $T_2(S)$ 
is also an elementary star, of the same norm as $S$. 
\end{proof}

\begin{lemma} \label{L:triples-stars}
Every element in the kernel of $\pi$ is the linear combination of elementary triples and elementary stars with norm at most $M$. 
\end{lemma}

\begin{proof}
We denote as $B_0$ the branch points of the Hubbard tree $T_f$ which lie in the set $\bigcup_{n \geq 0} f^n(0)$,
and as $B_1$ the other branch points. 

For each branch point $\alpha$ in $B_1$, define its $1$-neighborhood $N_1(\alpha)$ as the set of postcritical points which are closest to $\alpha$, meaning that there is no other postcritical point between them and $\alpha$. The complement 
$$T_f \setminus \left(  \bigcup_{\alpha \in B_0} N_1(\alpha) \cup \bigcup_{\alpha \in B_1} \{\alpha \} \right)$$
is the union of segments, which we will call \emph{edges}. 

Let $S$ be an element in the kernel of $\pi$. 
If an angle pair $\{ \xi,  \eta\}$ is in the support of $S$, then we can write its associated segment as a union of segments
$$[\overline{\xi}, \overline{\eta}] = [x_0, x_1] \cup [x_1, x_2] \cup \dots \cup [x_{r-1}, x_r],$$
each of them lying 
either in a neighborhood of a branch point or in an edge. For each $1 \leq i \leq r-1$, let $\eta_i$ be a postcritical angle 
whose ray lands at $x_i$. Moreover, set $\eta_0 = \xi$, $\eta_r = \eta$.
Then we can write 
$$\{ \xi,  \eta \} = \sum_{i = 0}^{r-1} \{ \eta_i, \eta_{i+1} \} + \sum_{i = 0}^{r-1} ( \{ \eta_i, \eta_r\} - \{ \eta_i, \eta_{i+1}\} - \{ \eta_{i+1}, \eta_r \})$$
where all terms in the last sum are elementary triples. 
Thus, any configuration with zero geometric weight can be written, up to adding elementary triples, as the sum of configurations 
with zero geometric weight which are supported either in the $1$-neighborhood of a branch point or in an edge. 

By Lemma \ref{L:kernel-combo}, configurations with zero geometric weight supported in an edge can be written as linear combinations of elementary triples. 
Moreover, for each branch point, the configuration restricted to the $1$-neighborhood of each point is an elementary star 
with zero geometric weight. By Lemma \ref{L:elem-star}, this configuration is a linear combination of elementary stars with
norm at most $M$. This completes the proof of the claim.
\end{proof}

%%%%%%%%%%%%%%%%%%%%%%%%%%%%%%%%%%%%%%%%%%%

\section{Roots of the spectral determinant for periodic angles} \label{sec:coversofthefinitemodel}
  
In \cite{TiozzoContinuity}, Tiozzo shows (by combining Theorem \ref{t:SpectralDeterminantGrowthRate} and Proposition \ref{p:GrowthRateFiniteModel}) that for a 
rational angle $\theta$, the inverse of the smallest root of the spectral determinant of the graph $\Gamma_{\theta}$ associated to the labeled wedge $\mathcal{W}_{\theta}$ equals the growth rate (largest eigenvalue) of the finite model graph ($\Gamma_1$ in the notation below),  which Thurston's entropy algorithm shows is the core entropy of the quadratic polynomial of external angle $\theta$. 
In this section, we investigate \emph{all} the roots of the spectral determinant, not only the smallest root.

\bigskip
\noindent \textbf{Setup.} Throughout this section, we will use the following notation:
\begin{itemize}
\item  $\mathcal{W}$ is a periodic labeled wedge of period $p$ and preperiod $q$,
\item $\Gamma$ is the (periodic) directed graph associated to $\mathcal{W}$,
\item For each $k \in \mathbb{N}$, $$\Gamma_k := \Gamma \big/ \equiv_{kp,q}$$ is the quotient of the graph $\Gamma$ by the equivalence relation $\equiv_{kp,q}$ (defined in \S  \ref{ss:periodiclabeledwedgesfinitemodels}).  $\Gamma_k = (\mathcal{V}_k, \mathcal{E}_k)$ denotes the vertex set and edge set of $\Gamma_k$. Since the labeling on $\mathcal{W}$ is constant on $\equiv_{kp,q}$-equivalence classes, the labeling of $\mathcal{W}$ induces a labeling on vertices of $\Gamma_k$. 
\item For each $k \in \mathbb{N}$, we use the canonical basis $\{e_v \ : \ v \in \mathcal{V}_k \}$ for $\mathbb{R}^{\mathcal{V}_k}$, where we denote as $e_v$ the element of $\mathbb{R}^{\mathcal{V}_k}$ that has a $1$ in the position corresponding to $v$ and $0$ in the other positions.  Then, $M_k$ is the incidence matrix associated to $\Gamma_k$, and $A_k$ the associated linear map corresponding to the matrix $M_k$ in the canonical basis.
As vertices of $\Gamma_k$ are in bijection with the set $\{(i,j) : 1 \leq i  \leq j  \leq kp + q \}$, 
$M_k$ is a square matrix of dimension ${kp + q + 1}\choose{2}$.  

\item For each $k \in \mathbb{N}$, we consider the linear map $\pi_k:\mathbb{R}^{\mathcal{V}_k} \to \mathbb{R}^{\mathcal{V}_1}$ 
defined on its basis elements by 
$$\pi_k(e_v) := e_{[v]}$$
where $[v]$ denotes the class in $\Gamma_1$ of a vertex $v$ in $\Gamma_k$. 
\end{itemize}

\subsection{Characteristic polynomials of finite covers} \label{ss:charpolysoffinitecovers}

The goal of this subsection is to prove the following theorem: 

  \begin{theorem} \label{t:cyclotomicfactors}
  For any $k \in \mathbb{N}$, the characteristic polynomial for the action of $A_k$ on $\mathbb{R}^{\mathcal{V}_k}$ is the product of the characteristic polynomial for the action of $A_1$ on $\mathbb{R}^{\mathcal{V}_1}$, cyclotomic factors, and $x^d$ for some integer $d \geq 0$. 
  \end{theorem}
   
\begin{proof}
This follows from Lemma \ref{l:commutativediagram} and Proposition \ref{p:charpolyrestricted}, which we state and prove below.
\end{proof}

\begin{lemma} \label{l:commutativediagram}
For each $k$, $\pi_k$ is a linear map which satisfies $A_1 \circ \pi_k = \pi_k \circ A_k$. That is, the following diagram commutes:
\[\begin{tikzcd}
\mathbb{R}^{\mathcal{V}_k} \arrow{r}{A_k} \arrow[swap]{d}{\pi_k} & \mathbb{R}^{\mathcal{V}_k} \arrow{d}{\pi_k} \\
\mathbb{R}^{\mathcal{V}_1} \arrow{r}{A_1} & \mathbb{R}^{\mathcal{V}_1}
\end{tikzcd}
\]
\end{lemma}

\begin{proof} 
By linearity, it suffices to verify commutativity on the set of basis vectors $\{e_v, v \in \mathcal{V}_k\}$ for $\mathbb{R}^{\mathcal{V}_k}$.  So consider a fixed vector $e_v$.  By condition \eqref{i:correspondingverticessamelabel} from Definition~\ref{D:periodiclabeledwedge}, the label of the vertex $v$ corresponding to $e_v$ in $\Gamma_k$ is the same as the label of the corresponding vertex (the image under the projection map from $\Gamma_k$ to $\Gamma_1$) in $\Gamma_1$, call it $w$.  Also by the periodicity of the labeling, an edge leaves $v$ if and only if a corresponding edge leaves $w$, and their targets belong to the same $\equiv_{p,q}$ equivalence class.  Since $\pi_k$ is constant on $\equiv_{p,q}$ equivalence classes, it follows that  $A_1 \circ \pi_k(e_v) = \pi_k \circ A_k(e_v)$.
\end{proof}

As a consequence of Lemma \ref{l:commutativediagram}, $A_k$ preserves $\textrm{ker}(\pi_k)$. It remains to prove: 

\begin{proposition}\label{p:charpolyrestricted}
The characteristic polynomial of the restriction of the linear map $A_k$ to the vector space $\textrm{ker}(\pi_k)$ is a product of cyclotomic polynomials and the polynomial $x^d$, $d \in \mathbb{N}$. 
\end{proposition}
  
In order to prove Proposition \ref{p:charpolyrestricted}, we first define and investigate a related vector space $H_k$ and a linear map $L_k$ on $H_k$. 
Define the set 
$$ \mathcal{H}_k \coloneqq \left \{(v, w) \in \mathcal{V}_k \times \mathcal{V}_k: v \equiv_{p,q} w, \ v \neq w \right \}$$
and let $H_k \coloneqq \mathbb{R}^{\mathcal{H}_k}$ be the vector space over $\mathbb{R}$ for which $\mathcal{H}_k$ is a basis. 

Note that an element $(v,w)$ of $\mathcal{H}_k$ is an \emph{ordered} pair, while each of $v$ and $w$ denotes an element of the unlabeled wedge $\Sigma$, and elements of $\Sigma$ are \emph{unordered pairs} of natural numbers: for this reason, we will use the notation $v = \{ a, b \}$ with $a, b \in \mathbb{N}$ to denote elements of $\mathcal{V}_k$, and $(v, w)$ to denote elements of $\mathcal{H}_k$. 

Note also that, by periodicity,  for any element $(v,w)$ of $\mathcal{H}_k$, the vertices $v$ and $w$ have the same label.

We use the canonical basis $\{e_{(v, w)}  :  (v, w) \in \mathcal{H}_k \}$ for $H_k$, where we denote as $e_{(v, w)}$ the element of 
$H_k = \mathbb{R}^{\mathcal{H}_k}$ that has a $1$ in the position corresponding to $(v, w)$ and $0$ in the other positions.
Moreover, if $v = w \in \mathcal{V}_k$, we define $e_{(v,w)} = e_{(v,v)}$ as $0$. 

Let $(v, w) \in \mathcal{H}_k$. Since $v \equiv_{p, q} w$, we can reorder the elements so as to write $v = \{ a_1, b_1 \}$ and $w = \{ a_2, b_2 \}$ such that 
$$a_1 \equiv_{p, q} a_2, \qquad b_1 \equiv_{p, q} b_2.$$
We now define $L_k(e_{(v, w)})$ as follows. 

\begin{enumerate}
\item 
If $v$ (hence also $w$) is equivalent, set
$$L_k(e_{(v, w)}) := 0.$$
\item
If $v$ (hence also $w$) is non-separated, then set
$$L_k(e_{(v, w)}) := e_{( \{ a_1 +1, b_1+1 \}, \{ a_2 + 1, b_2 + 1\})}.$$
\item
If $v$ (hence also $w$) is separated, then set
$$L_k(e_{(v, w)}) := e_{(\{1, a_1 +1\}, \{1, a_2+1\})} + e_{(\{1, b_1+1\}, \{1, b_2+1\})}.$$
\end{enumerate}
Then let $L_k$ be the unique linear extension to $H_k$. 
Note that in the above definition, we set $e_{(v, w)} = 0$ if $v = w$.  Since $a_1 \equiv_{p,q} a_2$ and $b_1 \equiv_{p,q} b_2$ together imply $a_1 +1  \equiv_{p,q} a_2 +1$ and $b_1 + 1 \equiv_{p,q} b_2 +1 $, 
in all cases the image under $L_k$ of $e_{(v, w)}$ belongs to $H_k$.  

\medskip
Recall that for topological dynamical systems $f:X \to X$ and $g:Y \to Y$, $f$ is said to be \emph{semiconjugate} to $g$ if there exists a continuous surjection $\phi:X \to Y$ such that $g \circ \phi = \phi \circ f$. 

\begin{lemma} \label{l:Lkconjonkernel}
The action of $L_k$ on $H_k$ is linearly semiconjugate to the action of $A_k$ on $\ker(\pi_k)$.  That is, there exists a surjective linear map $\phi:H_k \to \ker(\pi_k)$ such that the following diagram commutes:
\[\begin{tikzcd}
H_k \arrow{r}{L_k} \arrow[swap]{d}{\phi} & H_k \arrow{d}{\phi} \\
\textup{ker}(\pi_k) \arrow{r}{A_k}  & \textup{ker}(\pi_k).
\end{tikzcd}\]
\end{lemma}

\begin{proof}
Define $\phi:H_k \to \ker(\pi_k)$ to be the linear map whose action on the canonical basis vectors is given by 
$$\phi(e_{(v,w)}) := e_v - e_w.$$ 
For any elements $(v,w) \in \mathcal{H}_k$, by definition $v \equiv_{p,q} w$, which implies  $\pi_k(e_v - e_w)$ = 0.  Hence the codomain of $\phi$ is $\ker(\pi_k)$, as desired. 

Next we show that $A_k \circ \phi = \phi \circ L_k$.  By linearity, it suffices to verify that 
$$A_k \circ \phi (e_{(v,w)}) = \phi \circ L_k (e_{(v,w)})$$ for each element $(v,w) \in \mathcal{H}_k$.   

Let us consider $(v, w) = (\{a_1, b_1\}, \{a_2, b_2\})$ an element of $\mathcal{H}_k$ as above. 
If $v$ and $w$ are separated, then we compute
\begin{align} 
\phi \circ L_k(e_{(v, w)}) &  = \phi(e_{ (\{1, a_1 +1\}, \{1, a_2+1\})}) +  \phi(e_{ (\{1, b_1+1\}, \{1, b_2+1\}) })\\
& = e_{\{1, a_1 +1\}} - e_{\{1, a_2+1\}} + e_{\{1, b_1+1\}}  - e_{\{1, b_2+1\}}. \label{E:comm}
\end{align}
On the other hand, 
\begin{align}
A_k \circ \phi (e_{(v,w)})  & =  A_k \circ \phi \left( e_{(\{ a_1, b_1 \}, \{ a_2, b_2 \})}   \right) \\
 & =  A_k \left(e_{\{ a_1, b_1 \}} - e_{\{ a_2, b_2 \}}  \right) \\
 & =  e_{\{1, a_1 +1\}} + e_{\{1, b_1+1\}} - e_{\{1, a_2+1\}} - e_{\{1, b_2+1\}}
 \end{align}
which coincides with \eqref{E:comm}, verifying commutativity. 

The cases of $v, w$ equivalent or non-separated are more straightforward, so we do not write out the details. 
\end{proof}

\begin{lemma}\label{l:charpolyLk}
The characteristic polynomial for the action of $L_k$ on $H_k$ is a product of cyclotomic polynomials and the polynomial $x^d$ for some $d \in \mathbb{N}$. 
\end{lemma}

\begin{proof}
First, we will define a subspace $J_k$ of $H_k$ and investigate the action of $L_k$ restricted to this subspace. 

Define the set 
$$\mathcal{J}_k := \{(\{a_i,b_i\}, \{a_j, b_j\}) \in \mathcal{H}_k  \mid \textrm{at least one of } a_i = a_j \textrm{ and } b_i = b_j \textrm{ holds} \}.$$
Define $J_k$ to be the $\mathbb{R}$-vector space for which $\mathcal{J}_k$ is a basis. 

\smallskip
We claim that $L_k$ sends each element of $\mathcal{J}_k$ to either $0$ or an element of $\mathcal{J}_k$. 
Consider an element $(v, w) = (\{a_i, b_i \}, \{ a_j, b_j\} ) \in \mathcal{J}_k$. By construction, $\Phi(v) = \Phi(w)$. 

\begin{enumerate}

\item If $v$ is equivalent, then 
$$L_k(e_{(v,w)}) = 0.$$ 

\item If $v$ is non-separated, then the fact that $\{ a_i,b_i \}$ and $\{ a_j,b_j \}$ have a common coordinate implies that the two pairs 
$\{ a_1 +1, b_1+1 \}$, $\{ a_2 + 1, b_2 + 1\}$ that comprise $L_k(e_{(v,w)}) $ do too, hence it is an element of $\mathcal{J}_k$. 

\item If $v$ is separated, then, by definition of $\mathcal{J}_k$, one of the two targets 
$$e_{(\{1, a_1 + 1 \}, \{1, a_2 + 1 \})}, \qquad e_{(\{1, b_1 + 1 \}, \{1, b_2 + 1 \})}$$
is equal to zero, while the other belongs to $\mathcal{J}_k$, since two of its entries are equal to $1$. 

\end{enumerate}

Since $\mathcal{J}_k$ is a finite set, by Lemma~\ref{L:finite-cyclo} it follows that 
%there exist natural numbers $n_1 > n_2$ such that $(L_k \vert_{J_k} )^{n_1} = (L_k \vert_{J_k})^{n_2}$.  Consequently, 
the characteristic polynomial for the restriction of $L_k \vert_{J_k}$ is a product of a cyclotomic polynomial and a factor of $x^d$ for some integer $d \geq 0$. 

\smallskip
We now claim that $L_k$ sends in finitely many iterations every basis element $h \in \mathcal{H}_k \setminus \mathcal{J}_k$ to either $J_k$ 
or to a cycle of elements of $\mathcal{H}_k \setminus \mathcal{J}_k$. 

\begin{enumerate}

\item
If a pair $(v, w) \in \mathcal{H}_k \setminus \mathcal{J}_k$ is equivalent, then $L_k$ sends it to $0$, which is in $J_k$.  

\item
If $(v, w) \in \mathcal{H}_k$ is separated, then its two pairs of targets under $L_k$ both have a coordinate equal to $1$, so both target pairs are in $\mathcal{J}_k$.  
 
\item
If $(v, w) \in \mathcal{H}_k \setminus \mathcal{J}_k$  is non-separated, $L_k$ sends it to another pair $(v, w) \in \mathcal{H}_k \setminus \mathcal{J}_k$.   Since $\mathcal{H}_k$ is a finite set, either the orbit of this non-separated pair $(v, w)$ enters a cycle of non-separated pairs in $\mathcal{H}_k \setminus \mathcal{J}_k$, or the orbit eventually enters $J_k$.  
\end{enumerate}
 
 Therefore, there exists an integer $p$ such that $L_k^p(\mathcal{H}_k)$ is contained in the union of $J_k$ and finitely many (possibly zero) cyclic orbits in $\mathcal{H}_k \setminus \mathcal{J}_k$.  

Therefore the characteristic polynomial of $L_k$ acting on $H_k$ is the product of the characteristic polynomial of $L_k$ acting on $J_k$ and finitely many (possibly zero) cyclotomic polynomials and $x^d$.  Thus, the characteristic polynomial of $L_k$ acting on $H_k$ has the desired form. 
\end{proof}

Equipped with Lemmas  \ref{l:Lkconjonkernel} and \ref{l:charpolyLk}, we are now ready to prove Proposition \ref{p:charpolyrestricted}. 
  
\begin{proof}[Proof of Proposition \ref{p:charpolyrestricted}.]
By Lemma \ref{l:Lkconjonkernel}, there is a linear map  $\phi:H_k \to \ker(\pi_k)$ that semiconjugates the action of $L_k$ on $H_k$  to the action of $A_k$ on $\ker(\pi_k)$. 
Hence, the characteristic polynomial of $A_k$ divides the characteristic polynomial of $L_k$. 
By Lemma \ref{l:charpolyLk}, the characteristic polynomial for the action of $L_k$ on $H_k$ is a product of cyclotomic polynomials and the polynomial $x^d$ for some $d \in \mathbb{N}$.  
Thus, the same is true for all its divisors; in particular, for the characteristic polynomial of $A_k$. 
\end{proof}

\subsection{Relating characteristic polynomials of finite covers to the spectral determinant}
  
The main goal of this subsection is to prove the following theorem, which builds on Theorem \ref{t:cyclotomicfactors}.
  
\begin{theorem} \label{t:approxroots}
The set of roots  in $\mathbb{D}$ of the spectral determinant for the infinite graph, $P_{\Gamma}$, equals the set of roots in $\mathbb{D}$ of  the spectral determinant of the finite graph model, $P_{\Gamma_1}$.  
\end{theorem}

\noindent The remainder of this subsection builds up to the statement and proof of Theorem  \ref{t:approximations}, which will be then used to prove Theorem \ref{t:approxroots} above.
 
 We begin by proving two lemmas (\ref{l:cyclesshowupfinitetime} and \ref{l:shortfinitecyclesarereal}) which relate the multicycles of finite and infinite graph models.  
First, multicycles in the infinite graph also show up in big finite graphs: 
\begin{lemma} \label{l:cyclesshowupfinitetime}
 For each $n \in \mathbb{N}$ there exists $M \in \mathbb{N}$ such that whenever a multicycle $\gamma$ of length at most $n$ is in $\Gamma$, then $\gamma$ is also a multicycle in the finite graph $\Gamma_m$ for all $m \geq M$.  
\end{lemma}

\begin{proof} 
By \cite[Proposition 6.2]{TiozzoContinuity}, every vertex in a closed path of length $n$ has width at most $2n$. 
Hence, it is enough to choose $M$ with $M p + q\geq 2 n$, where $p$ is the period of the critical orbit.
\end{proof}

\noindent Second, short multicycles in big finite graphs also show up in the infinite graph:

\begin{lemma}  \label{l:shortfinitecyclesarereal}
For each $n \in \mathbb{N}$, and $m > \frac{2n + 2}{p}$, %there exists $m \in \mathbb{N}$ such that 
if $\gamma$ is a multicycle of length $n$ in $\Gamma_m$, then $\gamma$ is also a multicycle in $\Gamma$.
 \end{lemma}

\begin{proof}
Suppose that there exists a cycle of length $n$ in $\Gamma_m$ but not in $\Gamma$. This implies that it must pass through 
a vertex of the form $(h, pm +q)$. 
If you are at a vertex $(h, pm+q)$ and the vertex is separated, then you go to $(1,h+1)$ and $(1, q+1)$. 

In the first case, note that along our path starting at $(1,h+1)$ one needs to travel vertically at least $h - 1$ and horizontally at least $pm +q - h - 1$.  (Here, directions like ``vertical,'' ``horizontal,'' etc. refer to the layout used in, for example, Figure \ref{f:1/9}.)
Since each step goes up or to the right by at most one, the length $n$ satisfies 
$$n \geq \max \{ h -1, pm +q - h - 1\} \geq \frac{pm +q - 2}{2}.$$ 

In the second case, one notes that one needs to travel at least $pm - 1$ horizontally, implying $n \geq pm-1$. 

On the other hand, if the vertex $(h, pm + q)$ is non-separated, its outgoing edge goes to $(q+1, h+1)$ if $q \leq h$, or to $(h+1, q+1)$ if $h < q$.
In the first case, the vertical displacement is at least $h - q - 1$ and the horizontal displacement is at least $pm + q - h -1$, yielding
$$n \geq \max \{ h - q - 1, pm +q - h - 1\} \geq \frac{pm - 2}{2}.$$ 
In the second case, the horizontal displacement is at least $pm-1$, so
$$n \geq pm -1.$$

Thus, in every case we have $n \geq \frac{pm - 2}{2}$, and hence taking 
$$m > \frac{2n + 2}{p}$$ 
is sufficient to exclude the existence of such additional cycles of length $n$. 
\end{proof}

Combining Lemmas \ref{l:cyclesshowupfinitetime} and \ref{l:shortfinitecyclesarereal} immediately implies that the coefficients of the spectral determinants $P_{\Gamma_m}$ are asymptotically stable in the following sense: 
 
\begin{corollary} \label{c:truncationscoincide} 
For any degree $n$, there exists $M \in \mathbb{N}$ such that the terms of degree at most $n$ in the spectral determinant $P_{\Gamma}$ coincide with the terms of degree at most $n$ in the spectral determinant $P_{\Gamma_m}$ for all $m \geq M$. 
\end{corollary}

Now, in order to bound the coefficients of the spectral determinant, we need a uniform bound 
on the number of simple multicycles of given length. For the infinite graph $\Gamma$, this is given by Proposition \ref{p:boundedcoeffs}; 
we now obtain a similar bound for the finite graph $\Gamma_m$. 

\begin{lemma} \label{l:cycle-upper-bound}
For each $n$, the number $S_n^{(m)}$ of simple multicycles of length $n$ in $\Gamma_m$ is at most 
$$S_n^{(m)} \leq 2^q (pm + q)^{3 \sqrt{2 n} + 1}.$$
\end{lemma}

\begin{proof}
Recall that in $\Gamma_m$ there are the following types of edges: 
\begin{enumerate}
\item If $(i,j)$ is non-separated with $j < pm + q$, one has the edge $(i, j) \to (i+1, j+1)$, which we call of type $U$ (\emph{upwards}); 

\item  if $(i, j)$ is non-separated with $j = p m+ q$, then there is one edge coming out of it, and it may be of one of the two types: 
\begin{itemize}
\item
if $i > q$, $(i, p m + q) \to (q+1, i+1)$, which we call of type $D'$; 
\item
if $i < q$, $(i, p m + q) \to (i+1, q+1)$, which we call of type $U'$.
\end{itemize}
Note that if $i = q$, the edge goes to $(q+1, q+1)$, from which there 
is no further edge, so no multicycle is supported there.
\item If $(i,j)$ is separated with $j < pm + q$, one has two edges: 
\begin{itemize}
\item[(a)]
$(i, j) \to (1, i+1)$, which we call of type $B$ (\emph{backwards}), 
\item[(b)]
$(i, j) \to (1, j+1)$, which we call of type $D$ (\emph{downwards});
\end{itemize}

\item If $(i,j)$ is separated with $j  = pm + q$, one has two edges: 

\begin{itemize}
\item[(a)]
$(i, j) \to (1, i+1)$, which we call of type $B$ (\emph{backwards}), 
\item[(b)]
if $j = pm + q$, we have $(i, p m + q) \to (1, q+1)$, which we call of type $Q$. 
\end{itemize}
\end{enumerate}

Now, fix a simple multicycle $\gamma$ of length $n$. 

Consider the set of edges of type $B$ along $\gamma$, and let $h_1, \dots, h_r$ be the heights of the sources of 
all such edges. 
First, we claim that all $h_i$ are distinct: this is because their targets are $(1, h_i +1)$, and these vertices have to be disjoint 
by the definition of simple multicycle. 
Moreover, we claim that 
$$\sum_{i = 1}^r h_i \leq n;$$ 
this is because to get from $(1, h_i)$ to the vertex at height $h_{i+1}$ 
you need to increase the height by $h_{i+1} -1$, and each move raises the height by at most one. 
Note that the $h_i$'s are distinct positive integers, hence  their sum is larger or equal than the sum of first $r$ positive integers. 
As a consequence, 
$$\frac{r(r+1)}{2} = \sum_{i =1}^r i \leq \sum_{i = 1}^r h_i \leq n, $$
and hence $r \leq \sqrt{2 n}$. Hence, since $\Gamma_m$ has $\leq (pm+q)^2$ vertices, there are at most $(pm+q)^{2 \sqrt{2n}}$ 
choices for the set of sources of edges of type $B$. 

Similarly, consider the heights $h_1', \dots, h_s'$ of the sources of the edges of type $D'$ along $\gamma$. Since each of them has target
$(q+1, h_i'+1)$, all these heights are different. 
Moreover, we claim that 
$$\sum_{i= 1}^s (h_i' -q) \leq n; $$ 
this is because to get from $(q +1, h_i +1) $ to the next vertex of type $D'$ on the cycle, at height $h_{i+1}'$, 
you need to increase the height by at least $h_{i+1}' - q -1$, and each move raises the height by at most one. 
Thus, similarly as before, since all $(h_i - q)$ are distinct, we obtain $s \leq \sqrt{2 n}$. 
Since the possible sources for an edge of type $D'$ are $(i, pm+q)$ with $q+1 \leq i \leq pm + q$, 
there are at most $(pm)^{\sqrt{2n}}$ possible choices. 
 
We also claim that there is at most one edge labeled $Q$: this is because their target is $(1, q+1)$, hence by disjointness
this can only happen once along the multicycle. Since the number of possible sources for edges of type $Q$ is at most $pm + q$, there are at most $pm  + q$ such choices. 

Finally, since the edges $U'$ have $q$ possible targets, there are at most $q$ of them in any multicycle, and their source 
has height between $1$ and $q$, so there are at most $2^q$ possible choices. 

Now, we claim that the positions of the sources of the $B, Q, D',$ and $U'$ edges 
determines the multicycle. This is because any simple cycle in $\gamma$ needs to contain at least one of them (otherwise the 
path keeps going to the right), and, once these vertices are specified, the other vertices of the path are determined uniquely since 
there is only one edge coming out of non-separated vertices. 

Hence there are at most $((pm + q)^2)^{\sqrt{2 n}}$ choices for the sources of the $B$ edges, $(pm + q)^{\sqrt{2 n}}$ choices for the sources
 of the $D'$ edges, $pm+q$ choices for the sources of the $Q$ edge, and $2^q$ for the sources of the $U'$ edges.
Altogether, this leads to at most 
$$ (pm + q)^{2 \sqrt{2 n}}  \cdot (pm + q)^{\sqrt{2 n}} \cdot (pm+q) \cdot 2^q \leq 2^q (pm + q)^{3 \sqrt{2 n} + 1}$$
multicycles of length $n$, as required.
\end{proof}
  
We now obtain asymptotic stability of roots in $\mathbb{D}$ of the spectral determinants $P_{\Gamma_m}$:

\begin{theorem} \label{t:approximations}
The sequence of functions $(P_{\Gamma_k}(t))_{k \geq 0}$ converges uniformly to $P_\Gamma(t)$
on compact subsets of $\mathbb{D}$. 
As a consequence, roots of $P_{\Gamma}$ in $\mathbb{D}$ are approximated arbitrarily well by roots of $P_{\Gamma_k}$.
\end{theorem}
  
\begin{proof}
We claim that for each $r< 1$, the sequence $(P_{\Gamma_k}(t))_{k \geq 0}$ converges uniformly to $P_\Gamma(t)$ 
on the disk $\mathbb{D}_r = \{ |t| < r \}$.
Let
$$P_\Gamma(t) = \sum_{n = 0}^\infty a_n t^n, \qquad P_{\Gamma_k}(t) = \sum_{n = 0}^\infty a_n^{(k)} t^n$$
be the power series expansion of $P_\Gamma$ and  $P_{\Gamma_k}$.

First of all, by Proposition~\ref{p:boundingnumberofcycles}, the number $S_n$ of simple multicycles of length $n$
in the graph $\Gamma$ is bounded above by $(2n)^{\sqrt{2n}}$, hence 
%\begin{equation} \label{E:multi}
$$
a_n \leq (2n)^{\sqrt{2n}}.
$$
%\end{equation}
As for $\Gamma_k$, note that by Lemma \ref{l:shortfinitecyclesarereal}, if $k > \frac{2n + 2}{p}$ then 
%\begin{equation} \label{E:multi2}
$$a_n^{(k)} \leq S_n \leq (2n)^{\sqrt{2 n}}.$$
%\end{equation}
On the other hand, if $k \leq \frac{2n + 2}{p}$, by Lemma \ref{l:cycle-upper-bound} we obtain
%\begin{equation} \label{E:multi3}
$$a_n^{(k)} \leq S_n^{(k)} \leq  2^q (pk + q)^{3 \sqrt{2 n} + 1} \leq  2^q (2n + q + 2)^{3 \sqrt{2 n} + 1} .$$
%%\end{equation}
 
Now fix $0 < r < 1$ and $\epsilon > 0$, and pick $c$ with $1 < c < \frac{1}{r}$. There exists $n_0$ such that 
%$$\max\left\{ \frac{\sqrt{2n} \log(2n)}{n},  
$$\frac{(3 \sqrt{2n} + 1) \log(2n + q + 2) + q \log 2}{n} \leq \log c \qquad \textup{for all } n \geq n_0$$ 
so that 
\begin{equation} \label{E:multi3}
a_n \leq c^n \qquad 
\textup{ and } a_n^{(k)} \leq c^n  \qquad  \qquad \textup{for any }  n \geq n_0, k \geq 0.
\end{equation}

Hence, if we write 
$$P_\Gamma(t) = \sum_{n = 0}^\infty a_n t^n, \qquad P_{\Gamma_k}(t) = \sum_{n = 0}^\infty a_n^{(k)} t^n$$
the power series expansion of $P_\Gamma$ and  $P_{\Gamma_k}$, we have by equation \eqref{E:multi3} that
%the previous estimate on multicycles
$$|a_n t^n| \leq (2n)^{\sqrt{2n}} r^n \leq (rc)^n \qquad \textup{for }|t| \leq r, n \geq n_0.$$
Thus, if $n_1 \geq n_0$ is sufficiently large, we have  
$$ \left|\sum_{n = n_1}^\infty a_n t^n \right| \leq \sum_{n = n_1}^\infty (rc)^n = \frac{(rc)^{n_1}}{1 - rc} < \epsilon$$
%for $n_1$ sufficiently large, 
and exactly the same estimate holds for $P_{\Gamma_k}$ for any $k \geq 1$. 

Now, by Corollary \ref{c:truncationscoincide} there exists $k_0$ such that for $k \geq k_0$ the first $n_1$ coefficients of $P_\Gamma$ and $P_{\Gamma_k}$ are the same. 
Hence, for $k \geq k_0$, 
$$ \left| P_\Gamma(t) - P_{\Gamma_k}(t) \right| \leq \left|  \sum_{n = 0}^{n_1} a_n t^n  -  \sum_{n = 0}^{n_1} a_n^{(k)} t^n \right| + \left| \sum_{n = n_1 + 1}^{\infty} a_n t^n \right| 
+  \left| \sum_{n = n_1 + 1}^{\infty} a_n^{(k)} t^n \right|  \leq 0 + \epsilon + \epsilon = 2 \epsilon,$$
which proves the uniform convergence on the disk of radius $r$. 
\end{proof}

\begin{proof}[Proof of Theorem \ref{t:approxroots}] 
By Theorem \ref{t:approximations} and Rouch\'e's theorem, the set of roots of $P_{\Gamma_k}$ in $\mathbb{D}$ converges in the Hausdorff topology to the set of roots 
of $P_\Gamma$. By Theorem \ref{t:cyclotomicfactors}, for any $k$ the roots of $P_{\Gamma_k}$ in $\mathbb{D}$ coincide with the roots of $P_{\Gamma_1}$. Hence, the roots of $P_{\Gamma}$ in $\mathbb{D}$ coincide with the roots of $P_{\Gamma_1}$. 
\end{proof}

%%%%%%%%%%%%%%%%%%%%%%%%%%%%%%%%%%%%%%
\section{Continuous extension of $Z^+$ to $\mathbb{R}/\mathbb{Z}$} \label{S:continuous-ext}

In this section, we use Theorem \ref{t:approxroots} to prove Theorem \ref{t:continuousdiskextension}, reproduced here:

\begin{theoremcontinuousdiskextension}
The map $Z^+ : \mathbb{Q} / \mathbb{Z} \to Com^+(\mathbb{C})$ admits a continuous extension from $\mathbb{R}/\mathbb{Z} \to Com^+(\mathbb{C})$.  
\end{theoremcontinuousdiskextension}
   
Recall that we say that a sequence $(\mathcal{W}_n)$ of labeled wedges converges to a labeled wedge $\mathcal{W}$ if for any finite subgraph of $\mathcal{W}$, 
there exist $n_0$ such that, for $n \geq n_0$, the labels of all corresponding vertices of $\mathcal{W}_n$ agree with the labels of $\mathcal{W}$. 

By \cite[Proof of Proposition 8.5]{TiozzoContinuity}, for any angle $\theta$ the limits 
$$\mathcal{W}_{\theta^+} := \lim_{\theta' \to \theta^+}\mathcal{W}_{\theta}, \quad \mathcal{W}_{\theta^-} := \lim_{\theta' \to \theta^-} \mathcal{W}_{\theta'}$$
exist. The properties of these limit graphs are described in the following lemmas.
    
\begin{lemma} \cite[Lemma 8.6]{TiozzoContinuity} \label{l:limitonlydiffers}
 Let $\theta$ be periodic of period $p$ under the doubling map.  Then $\mathcal{W}_{\theta}$, $\mathcal{W}_{\theta^+}$ and $\mathcal{W}_{\theta^-}$ are periodic of period $p$, and differ only in the labelings of pairs $(i,j)$ with either $i \equiv 0 \mod{p}$ or $j \equiv 0 \mod{p}$. 
\end{lemma} 
           
\begin{lemma} \cite[Lemma 7.3]{TiozzoContinuity} \label{l:isomorphicfinitemodels}
Let $\mathcal{W}^a$ and $\mathcal{W}^b$ be two labeled wedges which are purely periodic of period $p$.  Suppose moreover that for every pair $(i,j)$ with $i,j \not \equiv 0 \mod{p}$, the label of $(i,j)$ in $\mathcal{W}^a$ equals the label in $\mathcal{W}^b$.  Then the finite models $\Gamma_1^a$ and $\Gamma_1^b$ are isomorphic graphs.  
\end{lemma}
   
\begin{remark} The subscript $1$ in the notation $\Gamma_1^a$ in Lemma \ref{l:isomorphicfinitemodels} is consistent with the notation defined at the beginning of \S  \ref{sec:coversofthefinitemodel} to denote the finite ``1-cover'' model associated to the labeled wedge $\mathcal{W}^a$. It differs slightly from the notation in \cite{TiozzoContinuity}. 
\end{remark}
   
\begin{lemma} \label{l:Hausdorffconvergence}
Let $(\mathcal{W}_n)_{n \in \mathbb{N}}$ be a sequence of labeled wedges (associated to angles $(\theta_n)_{n \in \mathbb{N}}$) that converges to a labeled wedge $\mathcal{W}$.   Denote by $(P_n)_{n \in \mathbb{N}}$ and $P$ the associated spectral determinants.  Then 
$$ \lim_{n \to \infty} S^1 \cup \{z \in \mathbb{D} \mid P_n(z) = 0\} = S^1 \cup \{z \in \mathbb{D} \mid P(z) = 0\} $$  
in the Hausdorff topology.  
\end{lemma}

\begin{proof}
The proof of \cite[Lemma 6.4]{TiozzoContinuity} shows $P_n$ converges to $P$ uniformly on compact subsets contained in the open disk $\mathbb{D}$.  The result then immediately follows by applying Rouch\'e's Theorem. 
\end{proof}

\begin{proof}[Proof of Theorem \ref{t:continuousdiskextension}]
For a periodic angle $\theta$, denote the finite graphs associated to the labeled wedges $\mathcal{W}_{\theta}$, $\mathcal{W}_{\theta^+}$ and $\mathcal{W}_{\theta^-}$  by $\Gamma_1^{\theta}$, $\Gamma_1^{\theta^+}$, and $\Gamma_1^{\theta^-}$, respectively.  
Combining Lemmas \ref{l:limitonlydiffers} and \ref{l:isomorphicfinitemodels} immediately gives that for any periodic angle $\theta$, the finite models $\Gamma_1^{\theta}$, $\Gamma_1^{\theta^+}$ and $\Gamma_1^{\theta^-}$ are isomorphic graphs.   Hence the characteristic polynomials  associated to the finite models, $P_{\Gamma_1^{\theta}}$, $P_{\Gamma_1^{\theta^+}}$ and  $P_{\Gamma_1^{\theta^- }}$, coincide.  Consequently,  Theorem \ref{t:approxroots} implies that the sets of roots in $\mathbb{D}$ of each of the spectral determinants (of the infinite models)
 $P_{\Gamma^{\theta}}$, $P_{\Gamma^{\theta^+}}$ and  $P_{\Gamma^{\theta^- }}$ coincide.
If $\theta$ is not periodic, $\mathcal{W}_{\theta^+} = \mathcal{W}_{\theta} = \mathcal{W}_{\theta^-}$ by \cite[Proof of Proposition 8.5]{TiozzoContinuity}.
In both cases, Lemma \ref{l:Hausdorffconvergence} then gives the result. 
\end{proof}

%%%%%%%%%%%%%%%%%%%%%%%%%%%%%%%%%%%%%%
 \section{Kneading theory for principal veins} \label{S:kneading-veins}

The purpose of this section is to define a new ``kneading polynomial'' for quadratic polynomials in principal veins. 
Although approaches to kneading theory for tree maps already exist (e.g. \cite{kneadingtheoryfortreemaps}), they do not 
satisfy the continuity properties we need later (specifically, in \S  \ref{s:persistence}), hence we cannot apply them directly. 
To this end, we formulate a new kneading determinant which is uniform along each principal vein, using the first 
return map to a certain subinterval.

\subsection{Itineraries} \label{ss:itineraries}

Fix integers $0 < p < q$, with $p$ and $q$ coprime, and let $\mathcal{V}_{p/q}$ denote the $\tfrac{p}{q}$-principal vein.  As described in \S \ref{ss:veins}, there is a  fixed topological/combinatorial model that describes the dynamics of  the restriction of any  polynomial $f_c$ in $\mathcal{V}_{p/q}$ to its Hubbard tree $T_{f_c}$.  

Namely, $T_{f_c}$ is, topologically, a star-shaped tree with $q$ branches, and the central vertex of the star is the $\alpha$-fixed point of $f_c$.  One of the branches contains the critical point, $0$ in its interior (unless the map $f$ is conjugate to a rotation, which happens, e.g. for the Douady rabbit map).  
We cut this branch at $0$ to form two topological intervals; we label the interval that contains the central vertex of the star $I_1^c$, and we label the other interval $I_0^c$.  We label the other branches $I_2^c,\ldots,I_q^c$, so that $f(I_k^c) = I_{k+1}^c$ 
for any $2 \leq k \leq q-1$.  (See Figure \ref{f:Htree1/5}.)
 We take these intervals to be closed, so that both the $\alpha$-fixed point and $0$ belong to more than one interval.  Then the dynamics of $f_c$ restricted to $T_{f_c}$ is as follows:
\begin{itemize}
\item $I_k^c$ is sent to $I_{k+1}^c$ homeomorphically, for $1\leq k\leq q-1$.
\item $I_q^c$ is sent to $I_0^c\cup I_1^c$.
\item $I_0^c$ is sent to a subset of $I_0^c\cup I_1^c\cup I_2^c$.
\end{itemize}
To lighten notation, we shall sometimes drop the superscript $c$ in $I_k^c$ when it is clear which 
parameter we are referring to.

\begin{definition}[Itinerary]
Let $f_c$ be a map in the principal vein $\mathcal{V}_{p/q}$. For any point $x \in T_{f_c}$  such that $x \not \in \bigcup_{k=0}^{\infty} f_c^{-k}(\{0,\alpha_c\})$,  where $\alpha_c$ is the $\alpha$-fixed point of $f_c$, define the \emph{itinerary} of $x$ under $f_c$ to be
the sequence 
\begin{equation} \label{E:itin1}
\textup{It}_{c}(x) \coloneqq  w_0 w_1 \ldots \in \{ 0,1,\ldots,q \}^{\mathbb{N}}
\end{equation}
such that, for all $j \geq 0$, $I_{w_j}^c$ is the interval in the Hubbard tree $T_{f_c}$ that contains $f_c^j(x)$. 
Additionally, if $x \neq 0$ but there exists $k \geq 1$ such that $f_c^k(x) = 0$, we define 
\begin{equation} \label{E:itin2}
\textup{It}_{c}(x) := \lim_{y \to 0} \textup{It}_{c}(y).
\end{equation}
\end{definition}

\begin{figure} 
\includegraphics[width=0.8 \textwidth]{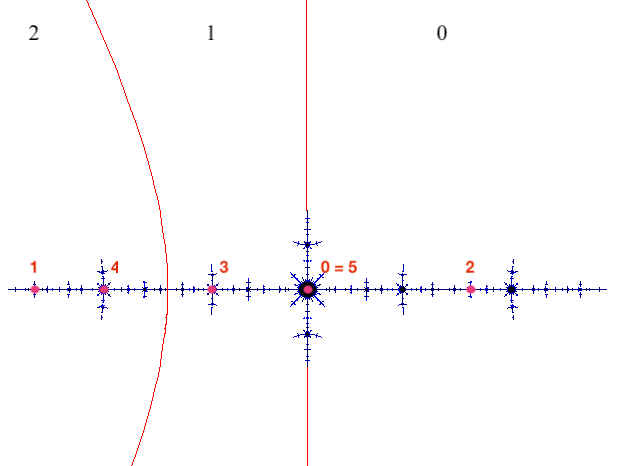}
\caption{The itinerary for the critical orbit of the real map of characteristic angle $\theta = \frac{13}{31} = .\overline{01101}$, 
corresponding to $c \approx -1.6254$.
With respect to the highlighted partition, the itinerary $\textup{It}(c)$ of the critical value is $\overline{\texttt{20121}}$.}
\label{F:real5}
\end{figure}

Note that the latter definition is well posed, as points on either side of $0$ map to a one-sided neighborhood of $c = f_c(0)$. 
In particular, Eq. \eqref{E:itin2} applies to define the itinerary $\textup{It}_c(c)$ of the critical value when the map $f_c$ is critically periodic of period $p > 1$. This will be the most important case in our paper. Thus, we can give the following:

\begin{definition} \label{def:kneadingseq}
We define the \emph{itinerary} associated to the map $f_c$, where $c$ belongs to any principal vein, as 
$$\textup{It}(c) := \textup{It}_c(c)$$
the itinerary of the critical value. 
\end{definition}

It will turn out to be simpler to consider a version of the itinerary that uses only the symbols $\{0,1, 2\}$, independently of $q$. 
In order to do so, we consider the first return map to the interval $I_0 \cup I_1 \cup I_2$, which is in fact unimodal:

\begin{definition}[First return map itinerary \emph{or }simplified itinerary]  \label{def:firstreturnitinerary} 
Let $f_c$ be a  map in the principal vein $\mathcal{V}_{p/q}$.  Let $\widehat{f}_c$ denote the first return map of $f_c$ to $I_0^c \cup I_1^c \cup I_2^c$.  Define the \emph{first return map itinerary}, or \emph{simplified itinerary}, of $f_c$, denoted $\widehat{\textup{It}}(x)$, to be the itinerary of $x$ under the map $\widehat{f}_c$. 
That is, $$\widehat{\textup{It}}_c(x) = w_0w_1w_2\ldots$$ 
is the sequence in  $\{0,1,2\}^{\mathbb{N}}$ where $(\hat{f}_c)^j(x) \in I_{w_j}^c$ for all $j$.  
Moreover, we define the simplified itinerary for the map $f_c$ as 
$$\widehat{\textup{It}}(c) := \widehat{\textup{It}}_c(c).$$
\end{definition}

It is easy to see that to go from $\widehat{\textup{It}}_c(x)$ to $\textup{It}_c(x)$, one simply deletes all letters greater than $2$ from $\widehat{\textup{It}}_c(x)$. To go from $\textup{It}_c(x)$ to $\widehat{\textup{It}}_c(x)$, one replaces every $2$ in the sequence $\textup{It}_c(x)$ with $234\dots q$. 

\subsection{Kneading polynomial and kneading determinant}
We are now ready to define a new analogue to the kneading determinant for polynomials along a principal vein. 
Let us fix $p, q$ coprime, with $0 < p < q$. 

We now define maps $F_{j, q, \lambda}$ which are ``candidate" piecewise linear models for 
the first return map $\widehat{f_c}$ for any parameter $c$ on the $p/q$-vein (see Remark \ref{l:PLsemiconjugacy}). 

\begin{definition} \label{d:Fmaps}
Let us define
\begin{align*}
F_{0,q,\lambda}(x) & := \lambda x + \lambda + 1\\
F_{1,q,\lambda}(x) & :=  -\lambda x + \lambda + 1\\
F_{2,q,\lambda}(x) & := -\lambda^{q-1} x + \lambda^{q-1} + 1
\end{align*}
Let polynomials $a_j$, $b_j$ be such that 
$$F_{j, q, \lambda}(x) := a_j(\lambda) x + b_j(\lambda)$$
Then for each $j \in \{0,1,2\}$, there exist unique choices of $\epsilon_j \in \{ +1, -1 \}$ and $q_j \in \mathbb{N}^+$, and polynomial $B_j$ such that 
$$a_j(\tfrac{1}{t}) = \frac{\epsilon_j}{t^{q_j}},  \qquad b_j(\tfrac{1}{t}) = \frac{B_j(t)}{t^{q_j}}$$ for all $t \in \hat{\mathbb{C}}$.
Let $w= \widehat{\textup{It}}(c) \in \{0, 1, 2 \}^\mathbb{N}$.  
For each integer $k \geq 1$, define
\begin{align*} 
\eta_k & := \epsilon_{w_1} \dots \epsilon_{w_k} \\
d_k & := q_{w_0} + \dots + q_{w_{k-1}}
\end{align*}
while $\eta_0 = 1$, $d_0 = 0$.
We define the \emph{$q$-principal vein kneading determinant} of $f_c$  as 
$$D(t) := \sum_{k =0}^\infty \eta_k B_{w_k} t^{d_k}.$$
This is a power series in the formal variable $t$.
\end{definition}

We will sometimes suppress the $q$ in the subscript $F_{j,q,z}$ and write just $F_{j,z}$ when the $q$ is clear from the context.  

The coefficients of the principal vein kneading determinant are all bounded, hence we have:

\begin{lemma}
For any natural number $q$, the $q$-principal vein kneading determinant $D(t)$ converges in the unit disk to a holomorphic function. The roots of $D(t)$ 
inside the unit circle, in particular the smallest root, change continuously with $w  \in \{0, 1, 2 \}^\mathbb{N}$.
\end{lemma}

In the case when the first return map itinerary is periodic, i.e. $f$ is critically periodic, the principal vein kneading determinant is rational. In this case, we can define more directly the following polynomial. 

\begin{definition}[principal vein kneading polynomial] \label{D:princ-kneading}
Let $f$ be a critically periodic quadratic polynomial of period $p$, with simplified itinerary $(w_1, w_2, \dots, w_{p-1}, w_p)^{\infty}$. 
Then the \emph{$q$-principal vein  kneading polynomial} $P_f$ is defined as 
\begin{equation} \label{E:pvkp}
P_f(z) := F_{w_{p-1},q, z} \circ \ldots \circ F_{w_1, q, z}(1+z).
\end{equation}
\end{definition}

The principal vein kneading polynomial and the kneading determinant are closely related:

\begin{lemma} \label{L:periodicP}
Suppose that the sequence $w = (w_j)_{j \geq 1}$ is periodic of period $p$; then we have 
$$D(t) =  P_f(\tfrac{1}{t}) \cdot \frac{  \eta_{p-1} t^{d_{p}}}{1 - \eta_p t^{d_p}}.$$ As a corollary, $D(t)$ and $P_f(\tfrac{1}{t})$ have the same 
roots inside the unit disk.
\end{lemma}

\begin{proof}
If $w = (w_j)_{j \geq 1}$ is periodic with period $p$, then 
$$\eta_{pn + k} = (\eta_p)^n \eta_k, \qquad B_{pn+k} = B_k, \qquad d_{pn+k} = n d_p + d_k$$
hence
$$D(t) = \sum_{k = 0}^\infty \eta_k B_k t^{d_k} = \frac{ \sum_{k = 0}^{p-1} \eta_k B_k t^{d_k}}{1 - \eta_p t^{d_p}}.$$
Moreover, we can compute by induction for any $n$
$$F_{w_n, 1/t} \circ \dots \circ F_{w_1,1/t}(x) = \frac{\epsilon_{w_1} \epsilon_{w_2} \dots \epsilon_{w_n} x + \sum_{k = 1}^n B_{w_k} \epsilon_{w_{k+1}} \dots \epsilon_{w_n} t^{q_{w_1} + \dots + q_{w_{k-1}}}}{t^{q_{w_1} + \dots + q_{w_n}}}$$
hence also
$$F_{w_n,  1/t} \circ \ldots \circ F_{w_1, 1/t}(1 + \tfrac{1}{t}) \eta_n t^{d_{n+1}} =  \sum_{k = 0}^n \eta_k B_k t^{d_k}.$$
Setting $n = p -1$ yields
$$P_f(\tfrac{1}{t})  \eta_{p-1} t^{d_{p}} = \sum_{k = 0}^{p-1} \eta_k B_k t^{d_k}$$
which implies 
$$ D(t) =  P_f(\tfrac{1}{t}) \cdot \frac{ \eta_{p-1} t^{d_{p}}}{1 - \eta_p t^{d_p}},$$
as required.
\end{proof}

\begin{remark}
Note that there is some ambiguity in the definition of the first letter $w_0$ of the itinerary of the critical point, since $0$ lies 
on the boundary of both $I_0$ and $I_1$; however, one has $1 + z = F_{0,z}(0) = F_{1, z}(0)$, hence we can interpret the formula \eqref{E:pvkp} for $P_f(z)$ as 
$$P_f(z) = F_{w_{p-1}, z} \circ \ldots \circ F_{w_1, z} \circ F_{w_0,z}(0)$$
where $w_0$ may be indifferently $0$ or $1$. 
\end{remark}

\subsection{Relationship with Markov matrix}

Now we will prove the main result of this section --  relating the roots of the kneading polynomial and the roots of the Markov polynomial.

Note that this completes the proof of Theorem \ref{T:equalpolys} from the introduction.

\begin{theorem} 
When $f$ is critically periodic and on the $\tfrac{p}{q}$-principal vein, the roots of the $q$-principal vein kneading polynomial $P_f(z)$ are the eigenvalues of the incidence matrix $A$ of the Markov decomposition defined using postcritical points as well as the $\alpha$-fixed point.
\end{theorem}

\begin{proof}
Let $$P := \bigcup_{n \geq 0} f^n(c) \cup \{ \alpha_f \}$$ be the ``postcritical" set of $f$, where $\alpha_f$ is the $\alpha$-fixed point, and let us suppose that $f$ is critically periodic, with critical point $c_0$. 
Denote by $\widehat{f}$ the first return map, and let $p$ such that $\widehat{f}^p(c_0) = c_0$. 
Let us denote $c_j := \widehat{f}^j(c_0)$, let $\widehat{\textup{It}}_c(c_1) = w_1 w_2 \dots w_{p-1} \dots$ the itinerary of the critical value, and let us define the map 
$\pi : P \cap I \to \mathbb{C}$ as 
$$\begin{array}{ll}
\pi(c_0) & = 0 \\
\pi(\alpha_f) & = 1 \\
\pi(c_k) & = F_{w_{k-1},q, z}\dots F_{w_1,q, z}(1+z) \qquad \textup{for }1 \leq k \leq p-1.
\end{array}$$
If $z$ is a root of the $q$-principal vein kneading polynomial, then
$$F_{w_{p-1}, q,z} \circ \ldots \circ F_{w_1, q,z}(1+z) = 0,$$
which implies that $\pi$ satisfies 
$$\pi(c_{j+1}) = F_{w_j,q, z}(\pi(c_j)) \qquad \textup{for } 0 \leq j \leq p-1.$$

Moreover, let us denote $$T_f \setminus P := \bigcup_{\alpha \in \mathcal{A}} J_\alpha.$$ Now, fix an orientation on each branch of $T$
(for instance, we can set the orientation on $I_0$ as increasing towards $I_1$, on $I_1$ as increasing towards $\alpha_f$, and on any $I_k$ with $k \geq 2$ as increasing away from $\alpha_f$); for each interval $J_\alpha = [a, b] \subseteq I_0 \cup I_1$, with $a < b$, define 
$$v_\alpha := \pi(b) - \pi(a).$$
If $J_\alpha \subseteq I_k$ with $k \geq 2$, then $J_\beta = f^{q+1 - k}(J_\alpha) \subseteq I_0 \cup I_1$, and we define
$$v_\alpha := z^{k-q-1} v_{\beta}.$$

We claim that $(v_\alpha)$ is an eigenvector for $A$, of eigenvalue $z$. 
In order to prove this, let $J_\alpha = [c_j, c_k] \subseteq I_0 \cup I_1$.
Then, since $J_\alpha$ is contained in one monotonic piece, $f(J_\alpha) = [c_{j+1}, c_{k+1}]$, and also 
there exists $w \in \{0, 1, 2 \}$ for which
$$\pi(c_{j+1}) = F_{w, q,z}(\pi(c_j)) \qquad \pi(c_{k+1}) = F_{w, q,z}(\pi(c_k)).$$
Now, write the decomposition of $f(J_\alpha)$ as $f(J_\alpha) = \bigsqcup_{i = 1}^{r} J_{\beta_i}$, 
and also denote $J_{\beta_i} = [p_i, p_{i+1}]$ with $p_i < p_{i+1}$. Note that $p_1 = c_{j+1}$, $p_{r+1} = c_{k+1}$ if $w = 0$, and 
 $p_1 = c_{k+1}$, $p_{r+1} = c_{j+1}$ if $w = 1$.
Then we compute
\begin{align*}
(Av)_\alpha & = \sum_{i = 1}^r v_{\beta_i}  = \sum_{i = 1}^r (\pi(p_{i+1}) - \pi(p_{i})) \\
 & =  \pi(p_{r+1}) - \pi(p_{1}) \\
 & = \epsilon_w \left( \pi(c_{k+1}) - \pi(c_{j+1}) \right)\\
 & = \epsilon_w \left( F_{w,q, z}(\pi(c_k)) - F_{w, q,z}(\pi(c_j)) \right) \\
 & = z ( \pi(c_k) - \pi(c_j)) = z v_\alpha,
 \end{align*}
 %which can be rewritten as $(Av)_\alpha = z v_\alpha$, 
 showing that $z$ is an eigenvalue for $A$. 

Conversely, suppose that $z$ is an eigenvalue of $A$, with eigenvector $(v_\alpha)$, and let us normalize it so that 
$\sum_{J_\beta \subseteq [c_0, \alpha]} v_\beta = 1$. Then define 
$m : I_0 \cup I_1 \cup I_2 \to \mathbb{C}$ as 
$$m(x) := - \epsilon(x) \sum_{J_\beta \subseteq [c_0, x]} v_\beta$$
where $\epsilon(x) = +1$ if $x \in I_0$ and $\epsilon(x) = -1$ if $x \in I_1 \cup I_2$  (note that, whatever choice we make about $\epsilon(c_0)$, we always have $m(c_0) = 0$). 
Then, since $(v_\alpha)$ is an eigenvector,  
\begin{align*}
(Av)_\alpha & = \sum_{J_\beta \subseteq f(J_\alpha)} v_{\beta}  = z v_\alpha.
\end{align*}
Now, suppose that $c_j \in I_0 \cup I_1$. By summing over $J_\alpha \subseteq [c_0, c_j]$, we have
\begin{align*}
\sum_{J_\beta \subseteq f([c_0, c_j])} v_{\beta}  & = z  \sum_{J_\alpha \subseteq [c_0, c_j]} v_\alpha = - \epsilon(c_j) z \cdot m(c_j) \\
\intertext{and, noting that $f([c_0, c_j]) = [c_{j+1}, c_0] \cup [c_0, c_1]$ or $f([c_0, c_j]) = [c_0, c_1] \setminus [c_0, c_{j+1})$ depending on $\epsilon(c_{j+1})$, we have}
 %\end{align*}
 %so 
 %\begin{align*}
\sum_{J_\beta \subseteq f([c_0, c_j])} v_{\beta}  & =
\sum_{J_\beta \subseteq [c_0, c_1]} v_{\beta} + \epsilon(c_{j+1})  \sum_{J_\beta \subseteq [c_0, c_{j+1}]} v_{\beta}    = (1 + z) - m(c_{j+1}) 
\end{align*}
hence, by comparing the previous two equations we obtain
$$m(\widehat{f}(c_j)) = F_{w_j,q, z}(m(c_j)).$$
%\qquad \textup{for any }0 \leq j \leq p-1.$$
An analogous computation holds if $c_j \in I_2$, using the first return time of the map $\widehat{f}$, yielding
$$m(\widehat{f}(c_j)) = F_{w_j,q, z}(m(c_j))\qquad \textup{for any }0 \leq j \leq p-1.$$
Thus, 
$$m(\widehat{f}^p(c_0)) = F_{w_{p-1, q, z}} \circ \ldots \circ F_{w_1,q, z}(1+z)$$
and, since $\widehat{f}^p(c_0) = c_0$, we obtain 
$$F_{w_{p-1,q, z}} \circ \ldots \circ F_{w_1,q, z}(1+z) = m(c_0) = 0.$$
\end{proof}

\begin{remark} 
This proof shows that the eigenvalues are the same. However, that does not quite prove that the polynomials are the same, 
as there could be Jordan blocks of size $> 1$.
\end{remark}

With the previous method we can also get the following semiconjugacy to the linear model. 
Since a similar result is obtained e.g. by \cite{BaillifCarvalho}, we do not provide the proof. 

\begin{remark}[Semiconjugacy to the piecewise linear model] \label{l:PLsemiconjugacy}
Let $f_c$ be in a principal vein $\mathcal{V}_{p/q}$, and let $\lambda = e^{h(f_c)}$. 

Set $I_{0,q,\lambda} \coloneqq [1-\lambda^q,0 ), I_{1,q,\lambda}\coloneqq [ 0, 1), I_{2,q,\lambda} \coloneqq [1,1+\lambda]$ and 
write $I_{q,\lambda} \coloneqq [1 -\lambda^q,1+\lambda]$.   Now, the maps $F_{j, q, \lambda}$ can be ``glued" to define the continuous, piecewise linear 
map $F_{q, \lambda} : I_{q, \lambda} \to I_{q, \lambda}$ as 
$$F_{q, \lambda}(x) := \sum_{j = 0}^2 F_{j, q, \lambda}(x) \boldsymbol{1}_{I_{j, q, \lambda}}(x).$$ 
Then the first return map $\widehat{f}_c$ of $f_c$ acting on $I_0^c \cup I_1^c \cup I_2^c$ is semi-conjugate to %the standard model 
$F_{q,\lambda}$ acting on $I_{q,\lambda}$:
\[ \begin{tikzcd}
I_0^c \cup I_1^c \cup I_2^c \arrow{r}{\widehat{f}_c} \arrow[swap]{d}{h} & I_0^c \cup I_1^c \cup I_2^c \arrow{d}{h} \\%
I_{q,\lambda} \arrow{r}{F_{q,\lambda}}& I_{q,\lambda}
\end{tikzcd}
\]
The semi-conjugacy $h$ sends the critical point of $f_c$ to $0 \in I_{q,\lambda}$ and the $\alpha$-fixed point of $f_c$ to $1 \in I_{q,\lambda}$; $h(I^c_j) = I_{j,q,\lambda}$ for $j =0,1,2$.  
If $f_c$ is critically periodic, then $F_{q,\lambda}$ is too, meaning $F_{q,\lambda}^n(0) = 0$ for some $n \geq 1$.
\end{remark}

 %%%%%%%%%%%%%%%%%%%%%%%%%%%%%%%%%%%%%%
 \section{Surgery and recoding} \label{S:surgery} 
 
\subsection{The Branner-Douady surgery} 

Recall that postcritically finite parameters in the Mandelbrot set are partially ordered:

\begin{definition}
Given two postcritically finite parameters $c_1, c_2$, we 
denote $c_1 <_{\mathcal{M}} c_2$ if $c_1$ lies on the vein $[0, c_2]$. That is, $c_1$ lies closer to the main cardioid than $c_2$. 
\end{definition}

Following Branner-Douady \cite{BrannerDouady}, there is a  \emph{$\tfrac{p}{q}$-surgery map} $\Psi_{p/q}:\mathcal{V}_{1/2} \to \mathcal{V}_{p/q}$ between the real vein in the Mandelbrot set and the $\tfrac{p}{q}$-principal vein. 
The construction was extended by Riedl \cite{Riedl} to arbitrary veins. 

\begin{theorem}[Branner-Douady \cite{BrannerDouady}, Riedl \cite{Riedl}] \label{T:BD}
The surgery map $\Psi_{p/q}$ is a homeomorphism between the real vein $\mathcal{V}_{1/2}$ and the $\tfrac{p}{q}$-principal vein $\mathcal{V}_{p/q}$. 
Moreover:
\begin{enumerate}
\item
 $\Psi_{p/q}$ preserves the order $<_\mathcal{M}$; 
 \item 
 the parameter $c' = \Psi_{p/q}(c)$ is critically periodic if and only $c$ is; 
 \item
 For any real critically periodic parameter $c$, the parameter $c' = \Psi_{p/q}(c)$ is the only critically periodic parameter 
in the $\tfrac{p}{q}$-principal vein satisfying
$$\widehat{\textup{It}}(c') = \widehat{\textup{It}}(c).$$
 \end{enumerate}
 
\end{theorem}

The map $\Psi_{p/q}$ is constructed as follows, at least for critically periodic parameters (for details, see \cite{BrannerDouady}):
\begin{itemize}
\item $\Psi_{p/q}:\mathcal{V}^{per}_{1/2} \to \mathcal{V}^{per}_{p/q}$: given a critically periodic map $f: I \to I$ on the interval, partition its domain $I$ in three 
subintervals $I_0, I_1, I_2$, where the critical point separates $I_0$ and $I_1$, and the $\alpha$-fixed point, denoted $\alpha$, separates $I_1$ and $I_2$. 
Now, make $q-2$ additional branches starting at $\alpha$, denoted as $I_3$, $I_4$, $\dots$, $I_q$. Now we modify the dynamics as follows: instead of sending the interval $I_2$ from $\alpha$ to the critical value to the original interval, we send it to another branch $I_3$, send $I_3$ to $I_4$, etc, and then send $I_q$ back to the image of $I_2$. The resulting map is conjugate to a map in  $\mathcal{V}_{p/q}^{per}$. 
\item $\Psi^{-1}_{p/q}:\mathcal{V}_{p/q}^{per} \to \mathcal{V}_{1/2}^{per}$:  For any map corresponding to a parameter  in $\mathcal{V}_{p/q}^{per}$, 
take the first return map on $I_0\cup I_1\cup I_2$. This first return map is an interval map, 
which is topologically conjugate to a map  in $\mathcal{V}_{1/2}^{per}$. 
\end{itemize}
There is some subtlety in defining $\Psi_{p/q}$ as above if the critical point eventually maps to the $\alpha$-fixed point, but we will from now on 
only focus on the critically periodic case, so we do not need that case.

Tiozzo \cite{TiozzoTopologicalEntropy} provided a description of the Branner-Douady surgery in terms of external angles called {\bf combinatorial surgery}. Because we are doing kneading theory we are more interested in the description of Branner-Douady surgery in terms of itineraries, which we call {\bf recoding} and will be described below.

\subsection{Binary itineraries}

Let us first recall the classical setup of Milnor-Thurston \cite{MilnorThurston} for unimodal interval maps. This gives rise to a symbolic coding with two symbols, $0$ and $1$. 

Let $f_c : I \to I$ be a real quadratic map. Decompose the interval $I$ into two subintervals $J_0$ and $J_1$, separated by the critical point $0$, where $J_1$ contains the $\alpha$ fixed point.

For any point $x \in I$  such that $x \not \in \bigcup_{k=0}^{\infty} f_c^{-k}(\{0\})$,  define the \emph{binary itinerary} of $x$ under $f_c$ to be
the sequence 
\begin{equation} \label{E:itin1-bin}
\textup{it}_{c}(x) \coloneqq  w_0 w_1 \ldots \in \{ 0,1 \}^{\mathbb{N}}
\end{equation}
such that, for all $j \geq 0$, $J_{w_j}$ is the interval that contains $f_c^j(x)$. 
Additionally, if $x \neq 0$ but there exists $k \geq 1$ such that $f_c^k(x) = 0$, we define 
\begin{equation} \label{E:itin2-bin}
\textup{it}_{c}(x) := \lim_{y \to 0} \textup{it}_{c}(y).
\end{equation}
Finally, the \emph{kneading sequence}, or \emph{binary itinerary}, of $f_c$ is defined as 
$$\textup{it}(c) := \textup{it}_{c}(c).$$
Moreover, the \emph{twisted lexicographic order} on $\{0, 1\}^\mathbb{N}$ is defined as $w<_{lex} w'$ iff there is some $i\in\mathbb{N}$, such that $w_j=w'_j$ for all $j<i$, and $(-1)^{\sum_{j<i}w_j}(w_i-w'_i)<0$.

A key property of the twisted lexicographical order is  the following:

\begin{lemma}[Milnor-Thurston \cite{MilnorThurston}] \label{L:twisted}
If $c, c'$ are real parameters in the Mandelbrot set, then $c <_{\mathcal{M}} c'$ if and only if 
$$\textup{it}(c) <_{lex} \textup{it}(c').$$
\end{lemma}

\subsection{Recoding}

The principal vein $\mathcal{V}_{1/2}$ consists of the real polynomials in $\mathcal{M}$.  For parameters $c \in \mathcal{V}_{1/2}$, the Hubbard tree $T_{f_c}$ is a real interval whose endpoints are $c$ and $f_c(c)$.  
This real interval may be thought of as a 2-pronged star emanating from the $\alpha$-fixed point of $f_c$; one prong contains the critical point $0$ in its interior, and so $0$ divides the prong into two topological intervals, which we label $I_0$ and $I_1$.  The other prong we label $I_2$. 

On the other hand, the classical setting of kneading theory uses two intervals, say $J_0$ and $J_1$, separated by the critical point. 
The relationship is $J_0 = I_0$ and $J_1 = I_1 \cup I_2$.  

In order to compare the $\{0, 1\}$-coding to the $\{0, 1, 2\}$-coding, we define the following \emph{recoding map}. 
Let $w\in\{0, 1\}^{\mathbb{N}}$ be the periodic kneading sequence of a critically periodic real quadratic map $f_c$.
Whenever there are $m$ consecutive $1$s in $w$, replace them with $\dots 21212$. More precisely, 
we replace any maximal block $1^{2m}$ of consecutive $1$'s of even length by $(12)^m$ and any maximal block $1^{2m+1}$ of odd length by $2 (12)^m$. In formulas:

\begin{definition} \label{def:recodingmap}
Let $\Sigma_1 \subseteq \{0, 1\}^\mathbb{N}$ be the set of binary sequences $(b_n)$ for which there exists $N$ such that $b_n = 1$ for all $n \geq N$. 
For any $k \geq 1$, define 
$$\begin{array}{l} 
\sigma(0^k) := 0^k \\
\sigma(1^{2k}) := (12)^k\\
\sigma(1^{2k-1})  := 2(12)^{k-1}.
\end{array}$$
We define the \emph{recoding map} 
$R : \{ 0, 1 \}^\mathbb{N} \setminus \Sigma_1 \to \{0, 1, 2\}^\mathbb{N}$
as follows:  
if $$\mathbf{a} = 1^{a_0} 0^{a_1} 1^{a_2} \dots 1^{a_{2n}} 0^{a_{2n + 1}}\dots$$
with $a_0 \geq 0$, $a_i \geq 1$ for $i \geq 1$, then 
$$R(\mathbf{a}) := \sigma(1^{a_0}) \sigma(0^{a_1}) \sigma(1^{a_2}) \dots \sigma(1^{a_{2n}}) \sigma(0^{a_{2n + 1}}) \dots$$
\end{definition}

The key property of the recoding map is the following:

\begin{lemma} \label{L:recode}
The recoding map establishes a bijective correspondence between the set of binary itineraries and the
set of simplified itineraries of critically periodic real quadratic polynomials. 
In particular, for any critically periodic real parameter $c$, we have 
$$R(\textup{it}(c)) = \widehat{\textup{It}}(c).$$ 
\end{lemma}

\begin{proof}
 This is because the dynamics on any $1/2$-vein Hubbard tree described above implies that $I_1$ must be sent to $I_2$, and $I_2$ may be sent to either $I_0$ or $I_1$. This also implies that the first $1$ will be replaced with $2$. 
 To prove that $R$ is bijective, note that $R^{-1} : \{0, 1, 2\}^\mathbb{N} \to \{0, 1 \}^\mathbb{N}$ is simply defined by replacing each character $2$ with $1$. 
 \end{proof}
 
\begin{remark}
Recoding is not well defined when the real quadratic itinerary ends with $1^\infty$, or, equivalently, the $\tfrac{p}{q}$-vein itinerary of the critical value hits the $\alpha$-fixed point. However, since we are only focusing on the critically periodic case, this do not happen in the situations we consider.
\end{remark}

The recoding map can also be defined on finite words $w$ where $w^\infty$ is some critical itinerary. In the classical  real quadratic map context, such words must start with $10$; hence if $w$ ends with $1$ the $1$ should not be changed into $2$ in the first step. We denote the resulting simplified itinerary as $R(w)$.
In symbols, if $w = a_1 \dots a_p \in \{0, 1\}^p$ is a finite word (with $w \neq 1^p$), then we compute
$$R(w^\infty) = b_1 \dots b_n \dots$$
and set 
$$R(w) := b_1 \dots b_p.$$
 More concretely, if $w$ is of the form $w = 1 0^{a_1} 1^{b_1} \dots 0^{a_r} 1^{b_r},$
with $r \geq 1, a_1 \geq 1, b_r \geq 0$, then one obtains
$$R(w) = \sigma(1) \sigma(0^{a_1}) \dots \sigma(0^{a_r}) \widehat{\sigma}(1^{b_r}),$$
where we set 
$$\widehat{\sigma}(1^{2k}) \coloneqq (21)^k,  \qquad \widehat{\sigma}(1^{2k+1}) \coloneqq  1 (21)^k,$$
so that 
$\sigma(1^{b + 1}) = \widehat{\sigma}(1^{b}) \sigma(1)$ for any $b \geq 0.$
 
It will be useful to note that $R$ respects concatenation, as follows: 

\begin{lemma} \label{L:R-concat}
If $w_0, w_1$ are finite words in the alphabet $\{0, 1\}$ and both start with $10$, then 
$$R(w_0 \cdot w_1) = R(w_0) \cdot R(w_1).$$
\end{lemma}

\begin{proof}
  Consider a maximal block of consecutive $1$s in $w_0\cdot w_1$; such a block must be can only arise via three cases:
  \begin{itemize}
  \item Consecutive $1$s in the end of $w_1$: they are turned into $\dots 2121$ by $R$.
  \item Consecutive $1$s entirely located in $w_0$ or $w_1$, but not at the end: they are turned into $\dots 21212$ by $R$.
  \item Consecutive $1$s at the end of $w_0$, followed by the first $1$ of $w_1$: they are turned into $\dots 21212$ by $R$. In particular, the first $1$ in $w_1$ is turned into a $2$ and the consecutive $1$s at the end of $w_0$ are turned into $\dots 2121$.
  \end{itemize}
In symbols, if we write 
$$w_0 = 1 0^{a_1} 1^{b_1} \dots 0^{a_r} 1^{b_r},\qquad w_1 = 1 0^{c_1} 1^{d_1} \dots 0^{c_s} 1^{d_s}$$ 
with $r, s \geq 1, a_1, c_1 \geq 1, b_r, d_s \geq 0$, then we have
$$w_0  \cdot w_1 = 1 0^{a_1} \dots 0^{a_r} 1^{b_r} 1 0^{c_1} \dots 0^{c_s} 1^{d_s}$$
hence by definition of $\widehat{\sigma}$ above 
\begin{align*}
R(w_0 \cdot w_1) & = \sigma(1) \sigma(0^{a_1}) \dots \sigma(0^{a_r}) \sigma(1^{b_r+1}) \sigma(0^{c_1}) \dots \sigma(0^{c_s}) \widehat{\sigma}(1^{d_s}) \\
 & = \sigma(1) \sigma(0^{a_1}) \dots \sigma(0^{a_r}) \widehat{\sigma}(1^{b_r}) \sigma(1) \sigma(0^{c_1}) \dots \sigma(0^{c_s}) \widehat{\sigma}(1^{d_s}) = R(w_0) \cdot R(w_1)
\end{align*}
showing that $R(w_0\cdot w_1)$ is the concatenation of $R(w_0)$ and $R(w_1)$.
\end{proof}

\subsubsection{The $q$-recoding map} 
We also consider the \emph{$q$-recoding map} $R_q : \{ 0, 1, 2 \}^\mathbb{N} \to \{ 0, 1, \dots, q \}^\mathbb{N}$ 
given by substituting each occurrence of $2$ with the word $23\dots q$. 
The $q$-recoding map turns simplified itineraries into ``full itineraries,''' i.e. has the property 
$$R_q(\widehat{\textup{It}}(c')) = \textup{It}(c')$$
for any critically periodic parameter $c'$ on the $\frac{p}{q}$-principal vein.
The map $R_q$ is a bijection between the simplified itineraries and the itineraries of critically periodic parameters on the $p/q$-vein.
The inverse map $R_q^{-1}: \{ 0, 1, \dots, q \}^\mathbb{N} \to \{ 0, 1, 2 \}^\mathbb{N}$ acts by deleting all characters larger than $2$. 

Finally, we call the map $$R^{-1} \circ {R_q}^{-1} : \{0, 1, 2, \dots, q \}^\mathbb{N} \to \{0, 1\}^\mathbb{N}$$ the \emph{binary recoding map}.

\medskip

Using Lemma \ref{L:recode}, Theorem \ref{T:BD}, and Lemma \ref{L:twisted}, we can now summarize our discussion in the following theorem:

\begin{theorem} \label{T:recode-summary}
Recoding provides a $1-1$ correspondence between all the following sets:
 \begin{enumerate}
    \item The itineraries of critically periodic parameters on the $\tfrac{p}{q}$-principal vein.
    \item The simplified itineraries of critically periodic parameters on the $\tfrac{p}{q}$-principal vein.
    \item The binary itineraries of critically periodic real quadratic maps on an interval.
 \end{enumerate}
In greater detail, we have: 
\begin{itemize}
\item[(a)]
Let $c$ be a critically periodic, real parameter. Then the recoding map yields
$$R(\textup{it}(c)) = \widehat{\textup{It}}(c).$$

\item[(b)]
If $c' = \Psi_{p/q}(c)$ is the surgery map, then 
$$\widehat{\textup{It}}(c) = \widehat{\textup{It}}(c').$$

\item[(c)]
The $q$-recoding map satisfies, for any critically periodic parameter $c'$ on the $p/q$-vein,
$$R_q(\widehat{\textup{It}}(c')) = \textup{It}(c').$$

\item[(d)]
Let $c$, $c'$ be two critically periodic parameters on the $p/q$-vein. Then $c <_{\mathcal{M}} c'$, i.e. $c$ is closer to the main cardioid than $c'$, iff the 
binary recoding of the critical itinerary of $c$ is smaller than the binary recoding of the critical itinerary of $c'$ under the twisted lexicographic order. 
\end{itemize}
\end{theorem}

The discussion is summarized in the diagram:
$$\begin{array}{lllll}
\{0,1\}^\mathbb{N} & \overset{R}{\longrightarrow}  & \{0,1, 2 \}^{\mathbb{N}}  & \overset{R_q}{\longrightarrow} &  \{0, 1, \dots, q\}^{\mathbb{N}} \\
\textup{it}(c) & & \widehat{\textup{It}}(c) =  \widehat{\textup{It}}(c') & &  \textup{It}(c')
\end{array}$$
where $c' = \Psi_{p/q}(c)$. 

 %%%%%%%%%%%%%%%%%%%%%%%%%%%%%%%%%%%%%%
\section{Renormalization} \label{S:renorm}

One notes that the teapot behaves nicely under taking roots. 
This is closely related to renormalization (and its inverse, \emph{tuning}) in the quadratic family. 

Recall that a \emph{polynomial-like} map is a proper holomorphic map $f : U \to V$ , where $U$ and $V$ are open, simply connected subsets of $\mathbb{C}$ and $\overline{U}$ is a compact subset of $V$.  A quadratic polynomial $f$ whose Julia set is connected is  \emph{$n$-renormalizable}, for $n \geq 2$, if there exists a neighborhood $U$ of the critical value such that $f^n:U \to f^n(U)$ is polynomial-like of degree $2$ with connected Julia set. A \emph{tuning map} of period $n \geq 2$ is a continuous injection $\tau:\mathcal{M} \to \mathcal{M}$ such that for every $c \in \mathcal{M}$ 
(except for possibly the root $c = 1/4$), the map $f_{\tau(c)}$ is $n$-renormalizable and the corresponding polynomial-like map is \emph{hybrid equivalent} to $f_c$ (i.e. is conjugate via a quasiconformal map to $f_c$ restricted to a suitable domain) \cite{douady1984etude}.
% disjoint open arcs $I_k$ for $k = 0, \dots, n-1$ in the Hubbard tree of $f$ such that $I_0$ contains the critical point, $f(I_k) = I_{k+1}$ for $k = 0, \dots, n-2$ and $f(I_{n-1}) \subseteq I_0$.

Let $f_{c_2}$ be a critically periodic quadratic polynomial and let $C$ be the hyperbolic component to which $c_2$ belongs. 
Then, there exists a tuning map $\tau_{c_2} : \mathcal{M} \to \mathcal{M}$ that sends the main cardioid to $C$.  
For any parameter $c_1 \in \mathcal{M}$, the parameter $\tau_{c_2}(c_1)$ is called the \emph{tuning} of $c_1$ by $c_2$, and 
the Julia set of $f_{\tau_{c_2}(c_1)}$ is obtained by inserting copies of the Julia set of $f_{c_1}$ at the locations of the critical orbit of $f_{c_2}$. 
For details, see e.g. \cite{McMullenRenormalization}. 

We now show that, as in the classical kneading theory for real maps, the principal vein kneading polynomial 
behaves well under tuning. 

\begin{lemma} \label{L:tuned}
Let $c_1$ be a critically periodic, real quadratic parameter, and let $c_2$ be a 
critically periodic parameter in the $\tfrac{p}{q}$-principal vein. 
Then the parameter $c = \tau_{c_2}(c_1)$ that is the tuning of $c_1$ by $c_2$ belongs 
to the $\tfrac{p}{q}$-principal vein and has $q$-principal vein kneading polynomial
$$P_{f_c}(t) = P_{f_{c_2}}(t)\cdot \frac{P_{f_{c_1}}(t^\ell)}{1+t^\ell},$$
where $\ell$ is the period of $f_{c_2}$. 
\end{lemma}

\begin{proof}
Recall that, by Definition~\ref{D:princ-kneading}, if we set $w = (w_1, \dots, w_{n})$ the simplified itinerary of a critically periodic $f_c$, and the piecewise linear model of $f_c$ is 
$$F_{w_j}(x) = A_j x + B_j = \epsilon_j t^{d_j} \cdot x  + t^{d_j} + 1, \qquad \epsilon_j \in \{\pm 1\}, d_j \geq 1$$
then we have the formula
$$P_{f_c}(t) = \sum_{k = 0}^{n-1} B_k \prod_{j = k+1}^{n-1} A_j = \sum_{k = 0}^{n-1} (1 + t^{d_k}) \prod_{j = k+1}^{n-1} \epsilon_{j} t^{d_j} .$$
where we set $w_0  = w_n$. 

Now, let $(w_1, \dots, w_{q})$ be the simplified itinerary of $c_2$, and let $(v_1, \dots, v_{p})$ be the simplified itinerary of $c_1$. 
Then, by looking at the combinatorics of tuning (see e.g. Figure \ref{F:tuned}), the simplified itinerary of $c$ is 
$$w_1 w_2 \dots w_{q-1} \widehat{v_1} w_1 w_2 \dots w_{q-1} \dots \widehat{v_{p-1}} w_1 w_2 \dots w_{q-1}\widehat{v_p} $$
where $\widehat{0} = 0$, $\widehat{1} = \widehat{2} = 1$ if $\epsilon_0 = +1$, and $\widehat{0} = 1$, $\widehat{1} = \widehat{2} = 0$ if $\epsilon_0 = - 1$.
Moreover, denote the local models of $f_{c_1}$ as 
$$F_{v_j, 2, t}(x) = \eta_j t \cdot x  + t+ 1 \qquad \eta_j \in \{\pm 1\}.$$
Now, note that, since $w_0$ is either $0$ or $1$, we have 
$$B_{\widehat{v_i}} = B_{w_0}$$
for any $i$ (recall that $B_0 = B_1 = t+1$ from Definition~\ref{d:Fmaps} ). 
Hence, we compute
\begin{align*}
P_{f_c}(t) & = \sum_{j = 0}^{p-1} \sum_{i = 0}^{q-1}  B_{w_i} A_{w_{i+1}} \dots A_{w_{q-1}} \left( \prod_{h = j +1}^{p-1} A_{\widehat{v_h}} A_{w_1} \dots A_{w_{q-1}} \right) \\
& = \left( \sum_{i = 0}^{q-1}  B_{w_i} A_{w_{i+1}} \dots A_{w_{q-1}}\right) \cdot \sum_{j = 0}^{p-1}  \left( \prod_{h = j+1 }^{p-1} A_{\widehat{v_h}} A_{w_1} \dots A_{w_{q-1}} \right) \\
& = P_{f_{c_1}}(t) \cdot  \sum_{j = 0}^{p-1} \left( \prod_{h = j +1}^{p-1} A_{\widehat{v_h}} A_{w_1} \dots A_{w_{q-1}} \right).
\end{align*}
Note now that 
$$A_{\widehat{v_h}} = \epsilon_q \eta_h t, \qquad A_{w_i} = \epsilon_i t^{d_i}$$
hence, since $\epsilon_1 \dots \epsilon_{q-1} \epsilon_q = 1$, we have for any $h$
$$A_{\widehat{v_h}} A_{w_1} \dots A_{w_{q-1}} =  \eta_h t^\ell$$ 
where $\ell$ is the period of the large scale dynamics, i.e. $f_{c_2}$. 
Hence
\begin{align*}
 \sum_{j = 0}^{p-1}  \left( \prod_{h = j +1}^{p-1} A_{\widehat{v_h}} A_{w_1} \dots A_{w_{q-1}} \right) & = \sum_{j = 0}^{p-1} \eta_{j+1} \dots \eta_{p-1} t^{\ell (p-1-j)} \\
 & = \frac{P_{f_{c_1}}(t^\ell)}{1 + t^\ell}
\end{align*}
which proves the formula. 
\end{proof}

\begin{corollary} \label{C:tuned}
Given $c_1, c_2$ as above and $c = \tau_{c_2}(c_1)$, we have 
$$h(f_c) = \max \left \{ h(f_{c_2}), \frac{1}{\ell} h(f_{c_1}) \right \}.$$
Moreover, by the same proof as in \cite[Theorem 6]{Douady-entropy}, if $h(f_{c_2}) > 0$ we have $h(f_c) = h(f_{c_2})$. 
\end{corollary}

\begin{figure}
\includegraphics[width=0.7 \textwidth]{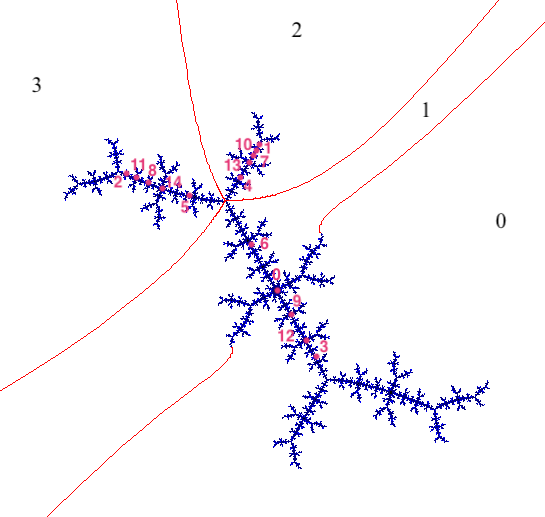}
\caption{The itinerary for the critical orbit of the complex map $f_{c'}$ in the $\frac{1}{3}$-principal vein given by tuning the map of Figure \ref{F:real5} 
by the rabbit. With respect to the highlighted partition, the itinerary of the critical value is $\textup{It}(c') = \overline{\texttt{230231230230230}}$.
The first return itinerary is $\widehat{\textup{It}}(c') = \overline{\texttt{2021202020}}$. The first return itinerary of the critical value for the map $f_c$ of Figure \ref{F:real5} is $\widehat{\textup{It}}(c) = \overline{\texttt{20121}}$; note that $\widehat{\textup{It}}(c')$ can be obtained by applying to $\widehat{\textup{It}}(c)$ the substitution $\texttt{0} \to \texttt{21}, \texttt{1} \to \texttt{20}, \texttt{2} \to \texttt{20}$ .} \label{F:tuned}
\end{figure}

\subsection{Characterization of minimal parameters} 

To distinguish between parameters on a given vein with the same entropy, we give the following

\begin{definition}[minimal parameter] \label{d:minimalparam}
We define a parameter $c$ on the $\tfrac{p}{q}$-principal vein to be \emph{minimal} if there is no parameter $c'$ with the same core entropy as 
$c$ and with $c' <_\mathcal{M} c$. 
\end{definition}

%A \emph{rooted tree} is a pair  $(T, \beta)$ where $T$ is a finite topological tree, and $\beta$ is an endpoint of $T$, which we will call its \emph{root}. 
%To the rooted tree $(T, \beta)$ one associates a partial order on $T$ by defining $x \leq_T y$ if $x$ belongs to the arc $[\beta, y]$. 
%A \emph{weakly monotone}  map $\varphi : (T, \beta) \to (T, \beta')$ between rooted trees is a map $\varphi : T \to T'$ with 
% $\varphi(\beta) = \beta'$ and such that $x \leq_T y$ implies $\varphi(x) \leq_{T'} \varphi(y)$.

Given a continuous map $f : T \to T$ of a finite topological tree $T$, a point $x \in T$ is a \emph{critical point} for $f$ if %of a continuous map $f : T \to T$ if 
$f$ is not a local homeomorphism in a neighbourhood of $x$. A map $f : T \to T$ is \emph{unimodal} if it has a unique critical point, which we denote $c$. 
Then $T \setminus \{c\} = T_0 \sqcup T_1$ is the union of two connected subtrees.
The \emph{kneading sequence} of a unimodal tree map $f : T \to T$ is the sequence $\epsilon = (\epsilon_i) \in \{0, 1, * \}^{\mathbb{N}}$ where $\epsilon_i = k$ if $f^i(c) \in T_k$ and $\epsilon_i = *$ if $f^i(c) = c$. 
We call a quadratic polynomial \emph{topologically finite} if its Julia set is connected and locally connected and its Hubbard tree is a finite topological tree. 

A \emph{small Mandelbrot set} is the image of $\mathcal{M}$ under a tuning map $\tau$; the \emph{root} of such a small Mandelbrot set is the root of the hyperbolic component onto which $\tau$ maps the main cardioid. 
%Recall that a parameter $c$ is renormalizable of period $p$ if and only if it belongs to a small Mandelbrot set of period $p$. 
We call a parameter \emph{tuned} if it lies in some small Mandelbrot set, and \emph{untuned} otherwise. 
Note that a topologically finite parameter $c$ belongs to a small Mandelbrot set of period $n$ if and only if $f_c$ has a so-called \emph{restrictive subtree}, i.e. a subtree $J$ in the Hubbard tree of $f_c$ such that $J$ contains the critical point of $f_c$ in its interior, $f_c^n(J) \subseteq J$, and the interiors of $f_c^k(J)$ for $k = 0, \dots, n-1$ are disjoint. 

\begin{remark}
Note that being untuned is almost the same as being non-renormalizable, but with the following caveat: If $c$ is the root of a satellite component in the Mandelbrot set, then $f_c$ is not renormalizable but clearly $c$ belongs to a small Mandelbrot set. In fact, the Hubbard tree of $f_c$ contains a restrictive subtree, but the first return map to this subtree does not extend to a polynomial-like map on a neighborhood of the subtree in the complex plane. 
\end{remark}

We now can characterize minimal parameters using the following lemmas. 

\begin{lemma} \label{L:semic-nr}
Let $f$ be a topologically finite, untuned quadratic polynomial, with Hubbard tree $T_f$ and $h(f) > 0$. 
%Let $f : T_f \to T_f$ be the restriction of a non-renormalizable, quadratic polynomial to its Hubbard tree. 
Then $f : T_f \to T_f$ is semiconjugate to a piecewise linear tree map $g : T_g \to T_g$ 
with constant slope (i.e. expansion factor) $s = e^{h(f)}$, where $T_g$ is homeomorphic to $T_f$ and $g$ has the same kneading sequence as $f$.
\end{lemma}

\begin{proof}
By \cite[Theorem 4.3]{BaillifCarvalho} there exists a %weakly monotone 
semiconjugacy $p : T_f \to T_g$ to a piecewise linear tree map $g : T_g \to T_g$ with constant slope $s = e^{h(f)}$. 
The semiconjugacy is the composition of a map that collapses at most countably many intervals to points and a homeomorphism; 
hence, the preimage of a point is a subtree of $T_f$. Moreover, the number of critical points of $g$ is at most the number of critical points of $f$, 
hence $g$ has a unique critical point. 

Let $c$ be the critical point of $f$, and $c' = p(c)$ the critical point of $g$. Set also $J = p^{-1}(\{ c' \})$, which is a subtree of $T_f$. 
%Since the semiconjugacy is weakly monotone, $J$ is a closed interval. 
If there exists $n$ such that $f^n(c)$ belongs to $J \setminus \{c \}$, then $f$ is tuned. %renormalizable. 
Indeed, let $n$ be the smallest such number. 
We have $$g^n(c') = g^n(p(c)) = p(f^n(c)) = p(J) = c'.$$ 
Moreover, for any $x \in J$ we have 
$$p(f^n(x))=g^n(p(x)) = g^n(c') = c',$$ 
hence $f^n(J) \subseteq J$.
Hence, if we set $J_i := p^{-1}(g^i(c'))$ for $i = 0, 1, \dots, n-1$, we have that the $J_i$ are disjoint, $f(J_i) \subseteq J_{i+1}$, 
and $f(J_{n-1}) \subseteq J$. Moreover, $J$ contains the critical point of $f$, hence $J$ is a restrictive subtree for $f$, so $f$ lies in a small Mandelbrot set, 
%is renormalizable, 
contradicting our assumption.
Hence, $f^n(c)$ does not belong to $J \setminus \{c \}$ for any $n$, which implies that for any $n \geq 1$ the points $f^n(c)$ and $g^n(c')$ 
belong to the corresponding connected components of the sets $T_f \setminus \{c\}$ and $T_g \setminus \{c'\}$, respectively, 
hence the kneading sequence of $f$ and $g$ are the same.
\end{proof}

\begin{lemma} \label{L:strict-mono}
Let $c_1, c_2$ be untuned,  %non-renormalizable, 
PCF quadratic polynomials on the $\tfrac{p}{q}$-principal vein, with $c_1 <_\mathcal{M} c_2$. 
Then $h(f_{c_1}) < h(f_{c_2})$. 
\end{lemma}

\begin{proof}
By monotonicity of core entropy (\cite{TaoLi}, \cite{Zeng-landing}), $h(f_{c_1}) \leq h(f_{c_2})$. 
Note that the unique untuned PCF quadratic polynomial with zero core entropy is $f_0(z) = z^2$, hence we can assume $h(f_{c_1}) > 0$. 
Then, let $g_1, g_2$ be the piecewise linear tree maps given by Lemma \ref{L:semic-nr}. 
If $h(f_{c_1}) = h(f_{c_2})$, then $h(g_1) = h(g_2)$.
Therefore, since the growth rate of a piecewise linear map equals its slope and the underlying trees are homeomorphic,
we have also $g_1 = g_2$, which implies the kneading sequence of $g_1$ is the same as the kneading sequence of $g_2$, hence by the Lemma \ref{L:semic-nr} also $f_{c_1}$ and $f_{c_2}$ have the same kneading sequence, hence $c_1 = c_2$ since they are both PCF.
\end{proof}

\begin{lemma} \label{L:plateau}
A parameter $c \in \mathcal{V}_{p/q}$ with $h(f_c) > 0$ is minimal if and only if it does not lie in the interior of a small Mandelbrot set whose root has positive core entropy.
\end{lemma}

\begin{proof}
By Corollary \ref{C:tuned}, core entropy is constant on small Mandelbrot sets whose roots have positive entropy. 
On the other hand, suppose that $c$ does not lie in the interior of any small Mandelbrot set whose root has positive entropy; there are two cases.
First, suppose that $c$ does not lie in the interior of any small Mandelbrot set. 
%,  i.e. it is non-renormalizable. 
Then, there exists two untuned, %non-renormalizable, 
PCF parameters $c', c''$ arbitrarily close to $c$ and with $c' <_\mathcal{M} c'' <_\mathcal{M} c$.
% arbitrarily close to $c$ also non-renormalizable. 
 By Lemma \ref{L:strict-mono}, $h(f_{c'}) < h(f_{c''}) \leq h(f_c)$, 
so entropy is not constant in any neighborhood of $c$. 
If instead $c$ lies inside a small Mandelbrot set whose root has zero entropy, since $c$ belongs to the vein $\mathcal{V}_{p/q}$, 
there exists $n \geq 0$ such that $c =  \tau_{p/q} \circ \tau_{1/2}^n(c')$ where $\tau_{p/q}$ is the tuning operator by the rabbit in the $p/q$-limb and $c'$ is 
untuned, %non-renormalizable, 
and $h(f_{c'}) = 2^n q \cdot h(f_c)$, and we can apply the previous argument to $c'$. 
\end{proof}

The strict monotonicity of core entropy except on small Mandelbrot sets was also shown in \cite{JungBiaccessibility}. 

 \begin{lemma} \label{l:closestrepresentative}
 Fix integers $0<p<q$ coprime.  Let $\lambda > 1$ be the growth rate of a parameter in $\mathcal{V}_{p/q}$.  Of all the critically periodic parameters in $\mathcal{V}_{p/q}^{per}$ with growth rate $\lambda$, let $c(\lambda)$ be the one that is closest to the main cardioid along the vein. 
 Then $\mathcal{Z}(\lambda)$ coincides with the set of all eigenvalues of $M_{c(\lambda)}$.  
 \end{lemma}
 
 \begin{proof} 
 By \cite{TaoLi}, \cite{Zeng-landing} core entropy is weakly increasing along veins, and by Lemma \ref{L:plateau} is constant precisely on small Mandelbrot sets with roots of positive entropy. 
 Thus, the set of parameters along a vein with given growth rate $\lambda > 1$ is the closure of a small Mandelbrot set. Let $c(\lambda)$ be the root of such small Mandelbrot set. 
 By Lemma \ref{L:tuned}, for all parameters $c'$ in the small Mandelbrot set with root $c$, the polynomial $P_{f_{c'}}$ is a multiple of $P_{f_c}$, 
hence the eigenvalues of $M_{c'}$ contain the eigenvalues of $M_c$. 
Thus, the intersection $\mathcal{Z}(\lambda)$ of all eigenvalues equals the eigenvalues of $M_{c}$. 
\end{proof}

\subsection{The teapot is closed under taking roots}
We can now prove Theorem \ref{T:q-root}, reproduced here:

\begin{renormalizationtheorem} 
If $(z,\lambda) \in \Upsilon_{1/2}$ with $|z| \neq 1$, then for any $q$, if $w^q = z$ then the point 
$(w,+\sqrt[q]{\lambda})$ belongs to $\Upsilon_{p/q}.$
\end{renormalizationtheorem}

\begin{proof}
Let $c$ be a critically periodic, real parameter, with $h(f_c) = \log \lambda$, and consider a fixed $z \in \mathbb{C}$ such that 
the $2$-principal vein kneading polynomial satisfies $P_{f_c}(z) = 0$. 
Now, let $c_2$ be the root of the $p/q$-principal vein, which has $q$-principal vein kneading polynomial $P_{f_{c_2}}(t) = 1 - t^q$. 
Then, by Lemma \ref{L:tuned}, the parameter $c' = \tau_{c_2}(c)$ lies on the $p/q$-vein and its kneading polynomial satisfies 
\begin{equation} \label{E:q-root}
P_{f_{c'}}(t) =  (1 - t^q) \frac{P_{f_c}(t^q)}{1+ t^q}.
\end{equation}
Hence, the core entropy satisfies $h(f_{c'}) = \frac{1}{q} \log \lambda$, and moreover, if $w^q = z$, by plugging $w = t$ into Eq. 
\eqref{E:q-root} we obtain 
$P_{f_{c'}}(w) = 0$.
It remains to show that $c'$ is minimal. Since $\sqrt[q]{\lambda} \leq \sqrt[q]{2}$, every critically periodic map in $\mathcal{V}^{per}_{p/q}$ with growth rate $\leq \sqrt[q]{\lambda}$ is of the form $\tau_{c_2}(c)$, where $c$ is a real critically periodic map. 
Since $\tau_{c_2}$ respects the ordering and scales the entropy by $\frac{1}{q}$, it maps minimal parameters to 
minimal parameters. 
 This shows that $(w, \sqrt[q]{\lambda}) \in \Upsilon_{p/q}$. 
The claim follows by taking the closure. 
\end{proof}

%%%%%%%%%%%%%%%%%%%%%%%%%%%%%%%%%%%%%%%%%
\section{Persistence for Thurston teapots for principal veins} \label{s:persistence}

The goal of this section is to prove the Persistence Theorem (Theorem  \ref{t:persistence}) for teapots associated to principal veins.  

%%%%%%%%%%%%%%%%%%%%%%%%%%%%%%%%%%%%%%%%%
\subsection{Itineraries and roots of the kneading polynomial}

The $q$-principal vein kneading polynomial can be generalized to arbitrary words in the alphabet $\{0,1,2\}$; we call this generalization the \emph{finite word $q$-kneading polynomial}. 

\begin{definition}[finite word $q$-kneading polynomial]
Let $w$ be a finite word in $\{0, 1, 2\}^n$ and let $q$ be a natural number. Then the \emph{finite word $q$-kneading polynomial} $P^q_w$ is defined as 
\[
P^q_w(z) \coloneqq F_{w_{n-1},q ,z} \circ \ldots \circ F_{w_1,q, z}(1+z).
\]
\end{definition} 

\noindent When $q$ is clear from the context or the result does not depend on $q$, we will sometimes write $P_w$ instead of $P^q_w$. 
\medskip

The following facts are immediate from the definition (see also Lemma \ref{L:periodicP} and its proof). 

\begin{lemma}\label{lem:bounded}
The coefficients of $P^q_w$ are uniformly bounded, where the bound depends only on $q$. Moreover, for any $n$, the polynomial $P^q_{w^n}$ equals $P^q_w$ multiplied by a cyclotomic polynomial, hence has the same roots other than those on the unit circle.
\end{lemma}

Given any monic polynomial $P(z)=\sum_{i=0}^da_iz^i$, we let the \emph{reciprocal} of $P$ be
\[r(P)(y) :={y^d\over a_d}P(1/y).\]
Then, from the construction of $P_w$, we have

\begin{lemma}
Let $w, w'$ be two finite words in the alphabet $\{0, 1, 2\}$. 
  \begin{enumerate}
  \item If $w$ and $w'$ share a large common suffix, then the lower degree terms of $P_w$ and $P_{w'}$ are identical;
  \item If $w$ and $w'$ share a large common prefix, then the lower degree terms of $r(P_w)$ and $r(P_{w'})$ are identical.
  \end{enumerate}
\end{lemma}

Combining this with Rouch\'e's theorem we have:

\begin{lemma}\label{lem:same_prefix_suffix}
Let $w, w'$ be two finite words in the alphabet $\{0, 1, 2\}$. 
  \begin{enumerate}
  \item If $w$ and $w'$ share a large common suffix, then the roots of $P_w$ and $P_{w'}$ within the unit circle are close to one another;
  \item If $w$ and $w'$ share a large common prefix, then the roots of $P_w$ and $P_{w'}$ outside the unit circle are close to one another.
  \end{enumerate}
\end{lemma}

  Combining Lemma~\ref{lem:same_prefix_suffix} and Lemma~\ref{lem:bounded} we have
  \begin{proposition}\label{prop:lim}
    Let $w_1$, $w_2$ and $w'$ be finite words in $\{0, 1, 2\}$, and let $\lambda_i$ denote the leading positive real root of $P_{w_i}$. 
     Suppose that $w = w_1^Nw'w_2^N$ for some $N \gg 1$, and denote by $\lambda$ the leading real root of $P_{w}$.  
     Then:
    \begin{itemize}
    \item As $N\rightarrow\infty$, $\lambda$ converges to $\lambda_1$.
    \item As $N\rightarrow\infty$, the roots of $P_{w}$ within the unit circle converge to roots of $P_{w_2}$.
    \end{itemize}
 \end{proposition}

%%%%%%%%%%%%%%%%%%%%%%%%%%%%%%%%%%%%%%%
\subsection{Persistence}

Firstly, we review some results on unimodal maps from \cite{BrayDavisLindseyWu}.

\begin{definition} \label{d:unimodaldefs}
Let $w$ be a finite word in the alphabet $\{0, 1\}$. Then: 
  \begin{itemize}
  \item $w$ is said to be 
  \emph{irreducible} if it cannot be written as concatenation of more than $1$ copies of another word; 
  \item 
  $w$ is said to have \emph{positive cumulative sign}, if there are an even number of $1$s;
  \item
It is said to be \emph{admissible}, if given any decomposition $w=ab$, we have 
$$ba\leq_{lex} ab$$
  \item
  the \emph{Parry polynomial} of $w$ is 
  $$P^{Parry}_w(z) :=f_{w_n} \circ \dots\circ f_{w_1}\circ f_{w_0}(1)-1$$
  with $f_0(x)=zx$, $f_1(x)=2-zx$.
  \item we define $D(w)$ as the word obtained from $w$ by replacing $1$ with $10$ and $0$ with $11$.
  \end{itemize}
\end{definition}

The Milnor-Thurston kneading theory \cite[Theorem 12.1]{MilnorThurston} implies that if $w$ is admissible, then $w^\infty$ is the critical itinerary of some real critically periodic quadratic map. 
 
\label{d:minimalparam}
\begin{definition}[minimal binary word]
Let $w$ be a finite word in the alphabet $\{0,1\}$. If $w$ is irreducible and $w^{\infty}$ is the binary itinerary of a minimal real parameter (Definition  \ref{d:minimalparam}), we call $w$ \emph{minimal}. 
\end{definition}

The following characterizations of minimality are well-known (see e.g. \cite{BruinBrucks}):

\begin{lemma} \label{L:minimal}
  The following are equivalent for an irreducible admissible word $w\in \{0, 1\}^n$:
  \begin{enumerate}
  \item $w$ is minimal;
  \item If there is another irreducible admissible word $w'$ such that $(w')^\infty$ is the itinerary of some unimodal map with the same entropy, then $w^\infty\leq_{lex} (w')^\infty$;
   \item $w$ is the binary itinerary of some piecewise linear map with a single critical point and slope of constant absolute value.
   \item If $w^\infty$ is the binary itinerary of some quadratic map with entropy less than $\log(\sqrt{2})$, then $w$ is minimal iff $w=D(w')$ where $w'$ is minimal. (Here $D$ is the map defined in Definition \ref{d:unimodaldefs})
   \end{enumerate}
\end{lemma}

We use the following key fact from \cite{BrayDavisLindseyWu}.

\begin{proposition}[\cite{BrayDavisLindseyWu}, Proposition 2.10] \label{P:minimal}
Let $w$ be an irreducible admissible word of positive cumulative sign, suppose that the Parry polynomial $P^{Parry}_w(z)$ has leading real root $\lambda$ larger than $\sqrt{2}$, and $P^{Parry}_w(z) =(z-1)g(z)$, where $g(z)$ is irreducible. Then $w$ is minimal. 
\end{proposition}

Now we combine the Proposition above, Lemma 3.8, Proposition 4.4 and 5.3 in  \cite{BrayDavisLindseyWu} and get:

\begin{theorem} \label{T:concat-real}
  Let $(w_0)^\infty$ and $(w_1)^\infty$ be two binary itineraries of real unimodal maps with periodic critical orbit, with $(w_0)^\infty <_{lex} (w_1)^\infty$, and suppose $(w_1)^\infty$ is minimal.   Then for any $N>0$, there is some word $w$ such that $(w_1^Nww_0^N)^\infty$ is the binary itinerary of the critical value of some real quadratic map, and  $(w_1^Nww_0^N)$ is a minimal binary word. 
\end{theorem}
  
\begin{proof}
First, assume the entropy corresponding to $(w_1)^\infty$ is greater than $\log(\sqrt{2})$. Then, by \cite[Proposition 4.4]{BrayDavisLindseyWu}\, we can find a word $w_1^Nw$ which is ``dominant'', and by \cite[Lemma 5.3]{BrayDavisLindseyWu} there exists some $m>N$ such that $w'=w_1^Nww_2^m$ is admissible with positive cumulative sign, and the Parry polynomial has the form $P^{Parry}_{w'}(z)=(z-1)g(z)$, where $g(z)$ is irreducible. Now Proposition~\ref{P:minimal} implies the conclusion. If the entropy is less than $\log(\sqrt{2})$, by Lemma~\ref{L:minimal} (4) and \cite[Lemma 3.8]{BrayDavisLindseyWu} we can write
$(w_1)^\infty = D^k(w_2)$ for some $k \geq 1$ so that $(w_2)^\infty$ has entropy greater than $\log(\sqrt{2})$, 
and apply the previous argument to $(w_2)^\infty$. 
\end{proof}

Next, we push Theorem \ref{T:concat-real} through the surgery map (making use of Lemma~\ref{L:R-concat}) to obtain the following theorem:

\begin{theorem}\label{main_comb}
  Fix any principal vein $\mathcal{V}_{p/q}$.  Let $(w_0)^\infty$ and $(w_1)^\infty$ be the simplified itineraries of two critically periodic parameters $c_0 <_\mathcal{M} c_1$ in $\mathcal{V}_{p/q}$.
  Then, given any $N>0$, there is a critically periodic parameter $c_2$ such that the simplified itinerary of $c_2$ starts with $N$ copies of $w_1$ and ends with $N$ copies of $w_0$, and $c_2$ is the smallest parameter on the vein with the given core entropy.
\end{theorem}

\begin{proof}
Set $(w_0)^{\infty} = \widehat{\textrm{It}}(c_0)$ and $(w_1)^{\infty} = \widehat{\textrm{It}}(c_1)$ for $c_0,c_1 \in \mathcal{V}_{p/q}^{per}$.  
Since the surgery map  $\Psi_{p/q}: \mathcal{V}_{1/2} \to \mathcal{V}_{p/q}$ is surjective and sends critically periodic parameters to critically periodic parameters, we may fix parameters $b_0, b_1 \in \mathcal{V}_{1/2}^{per}$ such that $\Psi_{p/q}(b_i) = c_i$ for $i=0,1$. 
Let $(u_0)^\infty, (u_1)^\infty$ be the binary itineraries of $b_0, b_1$, respectively. By Theorem \ref{T:recode-summary}(c), the inequality $c_0 <_\mathcal{M} c_1$ 
implies $(u_0)^\infty <_{lex} (u_1)^\infty$. Then by Theorem \ref{T:concat-real} there exists a  word $u$ in the alphabet $\{0, 1\}$ such that $(u_1^n u u_0^m)^\infty$ is the binary itinerary of a critically periodic real unimodal map $f_{b_2}$. Then, by Lemma \ref{L:R-concat}, 
$$R(u_1^n u u_0^m) = R(u_1)^n R(u) R(u_0)^m = w_1^n w w_0^m$$
is also the first return  itinerary of a parameter $c_2$ in $\mathcal{V}_{p/q}^{per}$.

Now, to prove that $c_2$ is minimal, recall that by Lemma \ref{L:plateau} minimal parameters are precisely the ones which do not lie in the 
interior of any small Mandelbrot set with root of positive entropy. Since the surgery map commutes with renormalization and preserves the set of parameters with zero entropy, $c_2 = \Psi_{p/q}(b_2)$ does not lie in the interior of a small Mandelbrot set with root of positive entropy, hence it is minimal. 
\end{proof}

Now the Persistence Theorem follows from Theorem~\ref{main_comb} and Proposition~\ref{prop:lim}. 
More precisely, we restate and prove the Persistence Theorem as follows.

\begin{theorem} \label{t:persistencepf}
Fix coprime integers $0< p < q$.   If $(z, y)\in \Upsilon_{p/q}$, $|z|<1$, then for any $y'$ satisfying $y < y' < \lambda_{p/q}$, where $\lambda_{p/q}$ is the supremum of the growth rates achieved by parameters in $\mathcal{V}_{p/q}^{per}$, we have $(z, y')\in \Upsilon_{p/q}$.
\end{theorem}

\begin{proof}
Let $w_1$ and $w_2$ be two periodic first return map itineraries corresponding to points in $\mathcal{V}_{p/q}^{per}$ whose core entropy is close to $y$ and $y'$ respectively, and that the kneading polynomial of $w_1$ has a root close to $z$. By Theorem~\ref{main_comb}, we can find another point $c\in \mathcal{V}_{p/q}^{per}$ whose first return map itinerary $w$ is of the form $w_2^nw'w_1^n$, where $n$ can be arbitrarily large. 
By Proposition \ref{prop:lim}, the roots of the $q$-principal vein kneading polynomial $P_w(z)$ inside the unit circle get arbitrarily close to the roots of the kneading determinant for $w_2^\infty$, while the roots outside the unit circle get arbitrarily close to the roots of the kneading determinant for $w_1^\infty$, as $n$ goes to infinity. This implies the persistence.
\end{proof}

As a corollary of Theorem \ref{T:q-root} and Theorem \ref{t:persistencepf} we get: 

\begin{cylindercorollary}
The teapot $\Upsilon_{p/q}$ contains the unit cylinder 
$$[1, \lambda_q] \times S^1 = \{ (\lambda, z) \ : \ 1 \leq \lambda \leq \lambda_q, |z| = 1 \}.$$
\end{cylindercorollary}

\begin{proof}
Let $(\lambda, z) \in \Upsilon_{1/2}$ with $\lambda > 0$. 
Then by applying Theorem \ref{T:q-root} $n$ times with $p/q = 1/2$, the set 
$$\{ ( \sqrt[2^n]{\lambda}, w) \ : \  w^{2^n} = z \}$$
belongs to $\Upsilon_{1/2}$. 
Hence, by applying Theorem \ref{T:q-root} once more, the set 
$$S_n := \{ ( \sqrt[2^n q]{\lambda}, w) \ : \  w^{2^n q} = z \}$$
belongs to $\Upsilon_{p/q}$. Since the sets $S_n$ accumulate onto $\{ 1 \} \times S^1$, then 
$\{ 1 \} \times S^1$ is contained in $\Upsilon_{p/q}$. 
Then by persistence (Theorem \ref{t:persistencepf}), the whole set 
$[1, \lambda_q] \times S^1$ is contained in $\Upsilon_{p/q}$. 
\end{proof}

 %%%%%%%%%%%%%%%%%%%%%%%%%%%%%%%%%%%%%%
\section{Combinatorial veins and the Thurston set} \label{ss:combinatorialveins}
 
The \emph{$\frac{p}{q}$-principal combinatorial vein}, which we denote $\Theta_{p/q}$, consists of the  closure of the set of all angles 
$\theta \in \mathbb{Q}/\mathbb{Z}$ such that the external parameter ray $R_{\mathcal{M}}(\theta)$ lands on a point in $\mathcal{V}_{p/q}$.  
Recall that $\Theta_{p/q}^{per}$ is the set of all angles $\theta \in \Theta_{p/q}$ such that $\theta$ is periodic under the doubling map, 
and $\Theta_{p/q}$ is the closure of $\Theta_{p/q}^{per}$.

We define an equivalence relation $\sim$ on $\Theta_{p/q}$ as follows.  We say $\theta_1 \sim \theta_2$ if and only if both are rational and there is 
a chain of adjacent hyperbolic components so that $\theta_1$ lands on the first one and $\theta_2$ lands on the last one.  Set 
$$\mathcal{V}_{p/q}^{comb} \coloneqq \Theta_{p/q}/\sim.$$
We use $[\theta]$ to denote the $\sim$-equivalence class of $\theta \in \Theta_{p/q}$. 

Let $\theta_{p/q}$ be the angle of the external ray landing at the tip of the $\frac{p}{q}$-principal vein. 
Now, the set $\Theta_{p/q} \cap [0, \theta_{p/q}]$ is a closed subset of an interval, hence its complement is the countable union of open intervals.
Endpoints of such open intervals correspond to pairs of rays landing on the same component, hence they are identified under $\sim$. 
Thus, the quotient space $\mathcal{V}_{p/q}^{comb}$ is homeomorphic to an interval.  

We define the \emph{combinatorial $\frac{p}{q}$-Master Teapot} to be the set 
$$
\Upsilon_{p/q}^{comb} \coloneqq \overline{ \left \{(z,\eta) \in \mathbb{C} \times \mathcal{V}_{p/q}^{comb} \mid \textrm{there exists } \theta \in \Theta_{p/q}^{per} \textrm{ s.t. } \textrm{det}(M_{\theta} - zI) = 0, \eta = [\theta] \right \} }.
$$

 \begin{remark}
Multiple different critically periodic parameters in a principal vein can have the same core entropy while having different characteristic polynomials $\chi(t) =  \textrm{det}(M_c -t I)$.  The vertical coordinate of the Master Teapot $\Upsilon_{p/q}$ identifies all critically periodic parameters that have the same core entropy, and plots only the roots that are shared by all the different characteristic polynomials $\textrm{det}(M_c -t I)$ associated to that growth rate.  The vertical coordinate of the combinatorial Master Teapot $\Upsilon_{p/q}^{comb}$ distinguishes between different parameters that have the same growth rate.  
\end{remark}

\begin{lemma} \label{L:same-eigen}
Fix integers $0<p<q$ coprime.  For angles $\theta_1,\theta_2 \in \Theta^{per}_{p/q}$, if $[\theta_1] = [\theta_2]$ in $\mathcal{V}^{comb}_{p/q}$, then $M_{\theta_1}$ and $M_{\theta_2}$ have the same eigenvalues, except possibly on the unit circle.
\end{lemma}

\begin{proof} 
Suppose that $\theta_1$ and $\theta_2$ land on roots of adjacent hyperbolic components on the $p/q$-principal vein, 
and suppose by symmetry that $\theta_1$ lands closer to the main cardioid that $\theta_2$. Let $c_1, c_2$ be the corresponding 
critically periodic parameters.
Then $c_2$ is the tuning of $c_1$ by the basilica. Hence, by Lemma \ref{L:tuned}, since the basilica has kneading polynomial 
$P_{f_c}(t) = 1 - t^2$, we have $$P_{f_{c_2}}(t) = P_{f_{c_1}}(t) \frac{1 - t^{2\ell}}{1 + t^\ell} = P_{f_{c_1}}(t) ( 1 - t^\ell).$$
Therefore $P_{f_{c_1}}$ and $P_{f_{c_2}}$ have the same roots except possibly on the unit circle.
By Theorem \ref{T:equalpolys}, these roots are the same as the eigenvalues of $M_{\theta_1}$, $M_{\theta_2}$. 
\end{proof}

As a consequence of Theorem \ref{t:continuousdiskextension}, we obtain that the part of a  combinatorial Master Teapot outside the unit cylinder is connected:
 
\begin{proposition}  \label{P:connected}
For any $(p, q)$ coprime, the set 
$$\Upsilon_{p/q}^{comb,+}  := \Upsilon_{p/q}^{comb} \cap \{ (z, \eta) \in \mathbb{C} \times \mathcal{V}_{p/q}^{comb} : \ |z| \geq 1\}$$
is path connected.
\end{proposition}

\begin{proof}
Consider any point $(z_*, \eta_*) \in \Upsilon_{p/q}^{comb,+}$. Let $\eta_0 \in \mathcal{V}_{p/q}^{comb}$ denote the  $\sim$-equivalence class of the angle in $\mathbb{R}/\mathbb{Z}$ of the external ray that lands at the root (in the main cardioid) of $\mathcal{V}_{p/q}$. 
By Theorem \ref{t:continuousdiskextension}, the map $Z^+ : \Theta_{p/q} \to Com^+(\mathbb{C})$ is continuous, and by Lemma \ref{L:same-eigen}
the map $Z^+$ is constant on equivalence classes of $\sim$, so the map $Z^+$ factors to a continuous map 
$$ \overline{Z^+}: \mathcal{V}_{p/q}^{comb} \to Com^+(\mathbb{C}).$$ 
As a consequence, there exists a continuous path in $\Upsilon_{p/q}^{comb,+}$ 
joining $(z_*, \eta_*)$ to some point of the form $(z', \eta_0) \in \Upsilon_{p/q}^{comb,+}$.   Note that
$$\Upsilon_{p/q}^{comb,+} \subseteq \{ (z, \eta) \ : \ 1\leq  |z| \leq \lambda(\eta) \},$$ where 
$$\lambda(\eta) \coloneqq \sup \{|z|  \ : \  \textrm{det}(M_{\theta}-zI)=0, \theta \in \Theta_{p/q}, [\theta] = \eta\}$$
 denotes the largest eigenvalue of all matrices $M_{\theta}$ such that $[\theta] = \eta$.  Furthermore, $\lambda$ is monotone increasing on $\mathcal{V}_{p/q}^{comb}$ and 
 $\lambda(\eta_0) = 1$.  Hence $|z'|=1$. 
 
By Corollary \ref{C:cylinder} the Master Teapot $\Upsilon_{p/q}$ contains the unit cylinder. 
Moreover, taking the growth rate defines a continuous map 
$$\varphi:\mathcal{V}_{p/q}^{comb} \times \mathbb{C} \to [1, \lambda_{p/q}] \times \mathbb{C}$$
and by construction $\varphi^{-1}(\Upsilon_{p/q}) \subseteq \Upsilon_{p/q}^{comb}$, hence $\Upsilon_{p/q}^{comb}$
also contains the unit cylinder. Thus, every point in $\Upsilon_{p/q}^{comb,+}$ is connected by a continuous path in $\Upsilon_{p/q}^{comb,+}$ to the unit cylinder; hence, $\Upsilon_{p/q}^{comb,+}$ is path connected.
\end{proof}

\begin{proof}[Proof of Theorem \ref{T:bagel-connected}]
Since $\Sigma_{p/q} \cap \{ z \ : \ |z| \geq 1\}$ is the projection of $\Upsilon_{p/q}^{comb,+}$ onto the $z$-coordinate, connectivity of 
$\Sigma_{p/q} \cap \{ z \ : \ |z| \geq 1\}$ follows from connectivity of $\Upsilon_{p/q}^{comb,+}$.
\end{proof}

 %%%%%%%%%%%%%%%%%%%%%%%%%%%%%%%%%%%%%%
\section*{Appendix. Relating the Markov polynomial and Milnor-Thurston kneading polynomial for critically periodic real maps}

For real postcritically finite parameters, we obtain a closer relationship between the characteristic polynomial for the Markov partition and the 
Milnor-Thurston kneading polynomial. 

If $f$ is critically periodic of period $p$, Milnor-Thurston's kneading determinant $D_{MT}(t)$ is of the form 
$$D_{MT}(t) = \frac{P_{MT}(t)}{1-t^p}$$
where $P_{MT}(t)$ is a polynomial of degree $p-1$, which we call the \emph{kneading polynomial}. 

\begin{proposition} 
Let $f$ be a critically periodic real quadratic polynomial of period $p$, and let $A$ be the transition matrix for the partition of the Hubbard tree minus the postcritical set into its connected components. Then we have the identity
$$\det (I-tA) = P_{MT}(t).$$
\end{proposition}

\begin{proof}
Recall that the Artin-Mazur zeta function of $f$ is defined as 
$$\zeta(t) := \exp \left( \sum_{k = 1}^\infty \frac{\# \textup{Fix}(f^k)}{k} t^k \right).$$
Moreover, Milnor-Thurston \cite{MilnorThurston} also consider the \emph{reduced zeta function} 
$$\widehat{\zeta}(t) := \exp \left( \sum_{k = 1}^\infty \frac{ \widehat{n}(f^k)}{k} t^k \right)$$
where $\widehat{n}(f^k)$ is the number of monotone classes of fixed points of $f^k$.
Two points $x, y$ lie in the same monotone equivalence class for $f^k$ if $f^k$ maps the whole interval $[x, y]$ monotonically. 

Moreover, for any matrix $A$, recall the formula (see e.g. \cite{TiozzoContinuity}, Lemma 4.4)
$$\det(I - t A) = \exp \left( - \sum_{k = 1}^\infty \frac{\textup{Tr }A^k}{k} t^k \right).$$
Note that, if $A$ is the Markov matrix for $f$, then for each $k$ we have 
$$\textup{Tr}(A^k) = \widehat{n}(f^k).$$
Indeed, the trace of $A^k$ is the number of (based) closed paths of length $k$ in the graph associated to $A$, 
hence it corresponds to a cycle of length $k$ of intervals with respect to the partition given by the postcritical set, 
and two such points have the same coding if and only if the belong to the same monotonicity class. 
Hence, we obtain 
\begin{equation} \label{E:detA}
\det(I - t A) = \frac{1}{\widehat{\zeta}(t)}.
\end{equation}
Note that for the Markov matrix, there are two choices: 
\begin{enumerate}
\item Consider $f : I_0 \to I_0$ with $I_0 := [f(c), f^2(c)]$ the Hubbard tree, where $c$ is the critical point. 
In this case, the Markov matrix $A_0$ has size $p-1$. 
\item Consider $f : I_1 \to I_1$ with $I_1 := [- \beta, \beta]$ where $\beta$ is the $\beta$-fixed point.
In this case, the Markov matrix $A_1$ has size $p+1$. 
\end{enumerate}
Note that there is exactly one periodic point (the $\beta$-fixed point) which is in $I_0$ but not in $I_1$. 
Hence
\begin{equation}\label{E:A1}
\det(I - t A_1 ) = (1-t) \det(I - t A_0).
\end{equation}
Note that Milnor-Thurston take as domain of definition of $f$ the largest interval $I_1$. Now, by \cite[Corollary 10.7]{MilnorThurston} we have 
\begin{equation} \label{E:zeta}
\frac{1}{\widehat{\zeta}(t)} = (1-t) (1-t^p) D_{MT}(t)
\end{equation}
where $\widehat{\zeta(t)}$ is the reduced zeta function for the action on $I_1$. Then by comparing \eqref{E:detA}, \eqref{E:A1} and \eqref{E:zeta} we have 
$$\det( I - t A_0) = \frac{\det (I - t A_1)}{1 - t} = \frac{1}{\widehat{\zeta}(t) (1-t)} =  (1-t^p) D_{MT}(t) = P_{MT}(t)$$
which is the desired identity, since in the statement of the proposition we took $A = A_0$. 
\end{proof}

%%%%%%%%%%%%%%%%%%%%%%%%%%%%%%%%%%%%%%
 
\bibliographystyle{alpha}
\bibliography{CoreEntropyTeapots.bib}

\newcommand{\etalchar}[1]{$^{#1}$}
\begin{thebibliography}{BDLW21}

\bibitem[AdSR04]{kneadingtheoryfortreemaps}
Jo\~{a}o Alves and Jos\'e de~Sousa~Ramos.
\newblock Kneading theory for tree maps.
\newblock {\em Ergodic Theory Dynam. Systems}, 24(4):957--985, 2004.

\bibitem[Bal00]{Baladi-book}
Viviane Baladi.
\newblock {\em Positive transfer operators and decay of correlations}.
\newblock World Scientific, Singapore, 2000.

\bibitem[BB04]{BruinBrucks}
Karen Brucks and Henk Bruin.
\newblock {\em Topics from One-Dimensional Dynamics}.
\newblock London Mathematical Society Student Texts. Cambridge University
  Press, 2004.

\bibitem[BD88]{BrannerDouady}
Bodil Branner and Adrien Douady.
\newblock Surgery on complex polynomials.
\newblock In {\em Proceedings of the Symposium on Dynamical Systems, Mexico,
  1986}, volume 1345 of {\em Lecture Notes in Math.}, pages 11--72. Springer,
  1988.

\bibitem[Bd01]{BaillifCarvalho}
Mathieu Baillif and Andr\'e deCarvalho.
\newblock Piecewise linear model for tree maps.
\newblock {\em Int. J. Bif. Chaos}, 11(12):3163--3169, 2001.

\bibitem[BDLW21]{BrayDavisLindseyWu}
Harrison Bray, Diana Davis, Kathryn Lindsey, and Chenxi Wu.
\newblock The shape of {T}hurston's {M}aster {T}eapot.
\newblock {\em Adv. Math.}, 377:107481, 2021.

\bibitem[BH12]{BrouwerHaemers}
Andries Brouwer and Willem Haemers.
\newblock {\em Spectra of Graphs}.
\newblock Springer, 2012.

\bibitem[Bou88]{BouschPaires}
Thierry Bousch.
\newblock Paires de similtudes.
\newblock preprint, available from author's webpage, 1988.

\bibitem[Bou92]{BouschConnexite}
Thierry Bousch.
\newblock Connexit\'{e} locale et par chemins h\"{o}lderiens pour les
  syst\`{e}mes it\'{e}r\'{e}s de fonctions.
\newblock preprint, available from the author's webpage, 1992.

\bibitem[CKW17]{CalegariKochWalker}
Danny Calegari, Sarah Koch, and Alden Walker.
\newblock Roots, {S}chottky semigroups, and a proof of {B}andt's conjecture.
\newblock {\em Ergodic Theory Dynam. Systems}, 37(8):2487--2555, 2017.

\bibitem[DH84]{DHOrsay}
Adrien Douady and John Hubbard.
\newblock Exploring the {M}andelbrot set. {T}he {O}rsay notes, 1984.

\bibitem[DH85]{douady1984etude}
Adrien Douady and John Hubbard.
\newblock \'{E}tude dynamique des polyn{\^o}mes complexes.
\newblock 1984-1985.

\bibitem[Dou95]{Douady-entropy}
Adrien Douady.
\newblock Topological entropy of unimodal maps: monotonicity for quadratic
  polynomials.
\newblock In {\em Real and complex dynamical systems ({H}iller\o d, 1993)},
  volume 464 of {\em NATO Adv. Sci. Inst. Ser. C Math. Phys. Sci.}, pages
  65--87. Kluwer, Dordrecht, 1995.

\bibitem[DS20]{DudkoSchleicher}
Dzmitry Dudko and Dierk Schleicher.
\newblock Core entropy of quadratic polynomials.
\newblock {\em Arnold Math J.}, 2020.

\bibitem[Gao20]{Gao}
Yan Gao.
\newblock On {T}hurston's core entropy algorithm.
\newblock {\em Trans. Amer. Math. Soc.}, 373:747--776, 2020.

\bibitem[GT21]{GaoYanTiozzo}
Yan Gao and Giulio Tiozzo.
\newblock The core entropy for polynomials of higher degree.
\newblock {\em J. Eur. Math. Soc.}, 2021.

\bibitem[Jun14]{JungBiaccessibility}
Wolf Jung.
\newblock Core entropy and biaccessibility of quadratic polynomials.
\newblock Preprint online at https://arxiv.org/abs/1401.4792, 2014.

\bibitem[Li07]{TaoLi}
Tao Li.
\newblock {\em A monotonicity conjecture for the entropy of Hubbard trees}.
\newblock PhD thesis, SUNY Stony Brook, 2007.

\bibitem[LW19]{LindseyWu}
Kathryn Lindsey and Chenxi Wu.
\newblock A characterization of {T}hurston's {M}aster {T}eapot.
\newblock Preprint online at https://arxiv.org/abs/1909.10675, 2019.

\bibitem[Lyu97]{Lyubich}
Mikhail Lyubich.
\newblock Dynamics of quadratic polynomials, i-ii.
\newblock {\em Acta. Math.}, 178(2):185--297, 1997.

\bibitem[McM94]{McMullenRenormalization}
Curtis McMullen.
\newblock {\em Complex Dynamics and Renormalization}, volume 135 of {\em Annals
  of Mathematics Studies}.
\newblock Princeton University Press, Princeton, NJ, 1994.

\bibitem[MT88]{MilnorThurston}
John Milnor and William Thurston.
\newblock On iterated maps of the interval.
\newblock {\em Dynamical Systems}, 1342:465--563, 1988.

\bibitem[Rie01]{Riedl}
Johannes Riedl.
\newblock {\em Arcs in multibrot sets, locally connected Julia sets and their
  construction by quasiconformal surgery}.
\newblock PhD thesis, TU M\"unchen, 2001.

\bibitem[Sch00]{Schleicher}
Dierk Schleicher.
\newblock Rational parameter rays of the {M}andelbrot set.
\newblock In {\em G\'eom\'etrie complexe et syst\`emes dynamiques - Colloque en
  l'honneur d'Adrien Douady, Orsay, 1995}, volume 261 of {\em Ast\'erisque}.
  2000.

\bibitem[SP19]{PerezSilvestri}
Stefano Silvestri and Rodrigo P\'erez.
\newblock Accessibility of the boundary of the thurston set.
\newblock {\em Exp. Math. (accepted)}, 2019.

\bibitem[TBG{\etalchar{+}}19]{ThurstonPeople}
William Thurston, Hyungryul Baik, Yan Gao, John Hubbard, Kathryn Lindsey, Lei
  Tan, and Dylan Thurston.
\newblock Degree-$d$ invariant laminations.
\newblock In Dylan Thurston, editor, {\em What's Next? The Mathematical Legacy
  of William P. Thurston}, volume 380 of {\em Annals of Mathematics Studies}.
  Princeton University Press, 2019.

\bibitem[Tho17]{thompson}
Daniel Thompson.
\newblock Generalized {$\beta$}-transformations and the entropy of unimodal
  maps.
\newblock {\em Comment. Math. Helv.}, 92(4):777--800, 2017.

\bibitem[Thu14]{thurston}
William Thurston.
\newblock Entropy in dimension one.
\newblock In {\em Frontiers in complex dynamics}, volume~51 of {\em Princeton
  Math. Ser.}, pages 339--384. Princeton Univ. Press, Princeton, NJ, 2014.

\bibitem[Tio13]{TiozzoThesis}
Giulio Tiozzo.
\newblock {\em Entropy, dimension and combinatorial moduli for one-dimensional
  dynamical systems}.
\newblock PhD thesis, Harvard University, 2013.

\bibitem[Tio15]{TiozzoTopologicalEntropy}
Giulio Tiozzo.
\newblock Topological entropy of quadratic polynomials and dimension of
  sections of the {M}andelbrot set.
\newblock {\em Adv. Math.}, 273:651--715, 2015.

\bibitem[Tio16]{TiozzoContinuity}
Giulio Tiozzo.
\newblock Continuity of core entropy of quadratic polynomials.
\newblock {\em Invent. Math.}, 203(3):891--921, 2016.

\bibitem[Tio20]{TiozzoGaloisConjugates}
Giulio Tiozzo.
\newblock Galois conjugates of entropies of real unimodal maps.
\newblock {\em Int. Math. Res. Not. IMRN}, 2020:607--640, 2020.

\bibitem[Zen20]{Zeng-landing}
Jinsong Zeng.
\newblock Criterion for rays landing together.
\newblock {\em Trans. Amer. Math. Soc.}, (373):6479--6502, 2020.

\end{thebibliography}

\bigskip

%%%%%%%%%%%%%%%%%%%%%%%%%%%%%%%%%%%%%%

\noindent Kathryn Lindsey, Boston College,  Department of Mathematics, Maloney Hall, Chestnut Hill, MA, 02467,
USA, \mbox{\url{lindseka@bc.edu}}

\smallskip
\noindent Giulio Tiozzo, University of Toronto, Department of Mathematics, 40 St George St, Toronto, ON, Canada,
\mbox{\url{tiozzo@math.utoronto.ca}}

\smallskip
\noindent Chenxi Wu, University of Wisconsin at Madison, Department of Mathematics, Madison, 53706, WI, USA, 
\mbox{\url{wuchenxi2013@gmail.com}}

\end{document}